\documentclass[12pt, reqno]{amsart}
\usepackage{amssymb, mathrsfs}
\usepackage{latexsym}
\usepackage{amsmath}
\usepackage{amsthm}
\usepackage{amsfonts}
\usepackage{dsfont, upgreek}
\usepackage{mathtools}
\usepackage{epsfig}
\usepackage{cite}
\usepackage{amscd}
\usepackage{graphicx}
\usepackage{tikz-cd}
\usepackage{enumerate}
\usepackage{color}
\usepackage{tikz}
\usetikzlibrary{patterns}

\usepackage[left=1.4in,right=1.4in,top=1.2in,bottom=1.2in]{geometry}

\usepackage{tikz}
\usetikzlibrary{decorations.markings}
\usetikzlibrary{arrows.meta}

\usepackage{extarrows}

\usepackage[colorlinks=true,linkcolor=blue,citecolor=blue]{hyperref}

\usepackage[colorinlistoftodos,prependcaption,textsize=tiny]{todonotes}  

\usepackage[all,cmtip]{xy}
\theoremstyle{plain}
\newtheorem{theorem}{Theorem}[section]
\newtheorem{proposition}[theorem]{Proposition}
\newtheorem{lemma}[theorem]{Lemma}
\newtheorem{corollary}[theorem]{Corollary}
\newtheorem{question}[theorem]{Question}

\theoremstyle{definition} 
\newtheorem{definition}[theorem]{Definition}

\newtheorem{example}[theorem]{Example}
\newtheorem*{claim*}{Claim}

\theoremstyle{remark} 
\newtheorem{remark}[theorem]{Remark}

\numberwithin{equation}{section}

\newcommand{\ilim}{\varinjlim}
\newcommand{\ind}{\textup{Ind}}
\newcommand{\id}{\mathrm{id}}
\newcommand{\K}{\mathcal{K}}
\newcommand{\cK}{\mathcal{K}}

\newcommand{\R}{\mathbb{R}}

\newcommand{\C}{\mathbb{C}}
\newcommand{\bH}{\mathbb{H}}
\newcommand{\supp}{\mathrm{supp}}
\newcommand{\prop}{\mathrm{prop}}
\newcommand{\diam}{\mathrm{diam}}
\newcommand{\Z}{\mathbb{Z}}
\newcommand{\N}{\mathbb{N}}
\newcommand{\sC}{\mathcal{C}}

\newcommand{\sN}{\mathcal{N}}

\newcommand{\ev}{\mathrm{ev}}

\newcommand{\fR}{\mathfrak{R}}
\newcommand{\fD}{\mathfrak{D}}

\newcommand\descitem[1]{\item{\bfseries #1}}

\begin{document}
	
\title[Decay of scalar curvature]{Decay of scalar curvature on uniformly contractible manifolds with finite asymptotic dimension}

\author{Jinmin Wang}
\address[Jinmin Wang]{School of Mathematical Sciences and Shanghai Center for Mathematical Sciences, Fudan University}
\email{wangjinmin@fudan.edu.cn}
\thanks{The first author is partially supported by NSFC11420101001.}
\author{Zhizhang Xie}
\address[Zhizhang Xie]{ Department of Mathematics, Texas A\&M University }
\email{xie@math.tamu.edu}
\thanks{The second author is partially supported by NSF 1800737 and 1952693.}
\author{Guoliang Yu}
\address[Guoliang Yu]{ Department of
	Mathematics, Texas A\&M University}
\email{guoliangyu@math.tamu.edu}
\thanks{The third author is partially supported by NSF 1700021, 2000082, and the Simons Fellows Program.}


%
\begin{abstract}
  Gromov  proved a quadratic decay inequality of scalar curvature for a class of complete manifolds \cite{Gromovinequalities2018}.   In this paper, we prove that for any uniformly contractible manifold with finite asymptotic dimension, its scalar curvature decays to zero at a rate depending only on the contractibility radius of the manifold and the diameter control of the asymptotic dimension.   We  construct examples of uniformly contractible manifolds with finite asymptotic dimension whose scalar curvature functions decay arbitrarily slowly. This shows that our result is the best possible.   We prove our result by studying the index pairing between Dirac operators and compactly supported vector bundles with Lipschitz control. A key technical ingredient for the proof of our main result is a Lipschitz control for the topological $K$-theory of finite dimensional simplicial complexes.
\end{abstract}
\maketitle
\section{Introduction}
The scalar curvature of a Riemannian manifold measures the deviation of the volume of a local geodesic ball from the volume of the ball of the same radius in Euclidean space. There has been great progress on the existence and non-existence of positive scalar curvature on closed manifolds or complete manifolds \cite{MR535700,MR541332,GromovLawson,MR569070,MR720934,MR1189863}. In a recent article \cite{Gromovinequalities2018}, Gromov proved that for any complete metric on the manifold $X = T^{n}\times \R^2$ and any $x_0\in X$, there exists a constant $r_0$ such that 
\[ \inf_{x\in B(x_0, r)} k(x) \leq \frac{4\pi^2}{(r-r_0)^2} \]
for all $r\geqslant r_0$, where $\inf_{x\in B(x_0, r)} k(x)$ denotes the infimum of the scalar curvature $k$ of $X$ on the ball $B(x_0, r)$ of radius $r$ centered at $x_0$. Subsequently,  Zeidler proved a similar  decay inequality for manifolds with a certain codimension-two obstruction  \cite{Rudolf2019}. The goal of this article is to establish a general decay inequality for scalar curvature on a large class of complete Riemannian manifolds. Our main result is the following.
 \begin{theorem}\label{thm:main}
 	Let $X$ be a complete Riemannian manifold with bounded geometry. Let $\fR,\fD$ be non-decreasing functions on $\R$. Assume that $X$ is uniformly contractible with contractibility radius $\mathfrak R$, and has asymptotic dimension $\leqslant m$ with diameter control $\mathfrak D$. Then there exists a non-decreasing function $F\colon \R_{\geqslant 0}\to\R_{\geqslant 0}$ \textup{(}depending only on $\fR$, $\fD$ and $m$\textup{)} such that $F(r) \to \infty$,  as $r\to \infty$,  and for any $x_0\in X$ and $r>0$, we have
 	$$\inf_{x\in B(x_0,r)}k(x)\leqslant \frac{1}{F(r)},$$
 	where $k(x)$ is the scalar curvature  at $x$ and $B(x_0,r)$ is the ball of radius $r$ centered at $x_0$. 
 \end{theorem}

Let us recall the definitions of uniform contractibility and asymptotic dimension. We say a Riemannian manifold is uniformly contractible with contractibility radius $\fR$ if every ball with radius $r$ is contractible within the $\fR(r)$-ball of the same center. We say that a metric space $X$ has asymptotic dimension no more than $m$ with diameter control $\fD$ if for any $r>0$ there exists a cover $\sC_r=\{U_i\}_{i\in I}$ such that each member has diameter no more than $\fD(r)$, and the cover has $r$-multiplicity $(m+1)$, i.e., any ball with radius $r$ intersects with at most $(m+1)$ members of $\sC_r$.
The concept of asymptotic dimension was introduced by Gromov \cite{Gromov1993asymptotic} to study asymptotic behavior of discrete groups. Unlike the Lebesgue covering dimension, the asymptotic dimension is a large scale invariant that neglects the local structure of a metric space, cf. \cite{BellDrannishnikovasymdim} for a survey of its history. 

A complete Riemannian manifold is said to have bounded geometry if the injective radius (i.e. the radius of a ball where the exp map is a diffeomorphism) is  uniformly bounded below by a positive number, and the curvature tensor is uniformly bounded.
Gromov conjectured that the scalar curvature on a uniformly contractible complete Riemannian manifold with bounded geometry cannot be uniformly positive. This conjecture implies the Gromov-Lawson-Rosenberg conjecture, which states that a closed aspherical manifold cannot admit a positive scalar curvature metric. Gromov's conjecture holds for manifolds with finite asymptotic dimension \cite{Yu} and for manifolds that can be coarsely embedded into a Hilbert space \cite{Yucoarseembed}.

A uniformly contractible manifold can have a complete metric whose scalar curvature is positive everywhere but tends to zero at infinity. For instance, the Euclidean plane $\R^2$ admits a complete metric with positive (but not uniformly positive) scalar curvature inherited from a paraboloid of revolution. This was probably one of the most basic examples that motivated Gromov to study the phenomenon of quadratic decay of scalar curvature in general  \cite{Gromovinequalities2018}.

 Let $X$ be a complete Riemannian manifold and $k$ the associated  scalar curvature function. We say the metric on $X$ satisfies the quadratic decay inequality if for any point $x_0\in X$,  there exist positive constants $C$ and $r_0$ such that
$$\inf_{x\in B(x_0,r)}k(x)\leqslant \frac{C}{(r-r_0)^2},\ \forall r>r_0.$$
Gromov proved the quadratic decay inequality for $X=T^n\times \R^2$ \cite{Gromovinequalities2018}. In \cite{Rudolf2019}, Zeidler proved a similar quadratic decay inequality when $X$ admits a certain codimension-two obstruction. We will see from the explicit  construction of the function $F$ in Theorem \ref{thm:main} that if both $\fR$ and $\fD$  are linear functions, then $F$ is quadratic, hence  the scalar curvature satisfies the quadratic decay inequality in this case.


In general, the decay of scalar curvature on a uniformly contractible manifold with finite asymptotic dimension can be much slower than quadratic decay. More precisely, we have the following proposition, which shows that our Theorem \ref{thm:main} is the best possible.
 
 \begin{proposition}[{cf. Proposition \ref{prop:example1}}]
 	Given $m\geqslant 3$ and for any prescribed non-decreasing function $G\colon \R_{\geqslant 0}\to \R_{\geqslant 0}$ such that $G(r)\to \infty$ as $r\to \infty$, there exists a complete Riemannian metric with bounded geometry on $\R^m$ satisfying the following: 
 	\begin{enumerate}[$(a)$]
 		\item $\R^m$ with this metric is uniformly contractible;
 		\item $\R^m$ with this metric has asymptotic dimension $\leq m$ whose diameter control function is linear and independent of $G$;
 		\item the scalar curvature function $k$ of this metric is positive everywhere and satisfies
 		$$\lim_{r\to+\infty}\big( G(r)\cdot\inf_{x\in B(x_0,r)}k(x)\big)=+\infty$$
 		for any $x_0\in\R^m$.
 	\end{enumerate}
 \end{proposition}

For a metric space $X$, a subspace $Y$ is called a net (more precisely, a  $C$-net for some $C>0$) if the $C$-neighborhood of $Y$ is equal to $X$, here the  $C$-neighborhood of $Y$ is defined to be the subset of all points in $X$ whose distance from $Y$ is less than $C$. We emphasize that the function $F$ in Theorem \ref{thm:main} is independent of the base point $x_0$. As a consequence, we see that given any $\varepsilon >0$, there exists an $r>0$ such that the ball $B(x, r)$ of radius $r$ centered at any point $x\in X$ contains a point $y$ with scalar curvature $k(y) <\varepsilon$.  Let us summarize this in the  following corollary.
\begin{corollary}\label{coro:net}
	Let $X$ be a complete Riemannian manifold with bounded geometry. If $X$ is uniformly contractible and has finite asymptotic dimension, then for any $\varepsilon>0$, the points of which the scalar curvature is smaller than $\varepsilon$ form a net in $X$. 
\end{corollary}
This corollary motivates the following question.
\begin{question}
	Given a uniformly contractible complete Riemannian manifold with bounded geometry, for any $\varepsilon>0$, does the collection of the points with scalar curvature smaller than $\varepsilon$ form a net?
\end{question}

Our main strategy for proving Theorem $\ref{thm:main}$ is to relate the rate of decay for scalar curvature to the index pairing between Dirac operators and compactly supported vector bundles with Lipschitz control.  Some of the ideas in the proof are inspired by the work of Connes \cite{Connes}, Connes--Gromov--Moscovici \cite{MR1204787,CGM90}, Gromov--Lawson \cite{MR569070,GromovLawson}, and Kasparov \cite{GK88}.   In order to prove Theorem $\ref{thm:main}$, we shall further develop a Lipschitz controlled topological $K$-theory of simplicial complexes, which is of independent interest. More precisely, let $(X,d)$ be a locally compact metric space. We denote by $C_0(X)$ the algebra  of continuous functions on $X$ that vanish at infinity. We equip  $C_0(X)$ with the usual supremum norm.  An element in $M_n(C_0(X))$ is said to be $L$-Lipschitz if 
$$\|f(x)-f(y)\|\leqslant L\cdot d(x,y),\ \forall x,y\in X, $$
where the norm on the left hand side is the operator norm of matrices. By the Serre-Swan theorem, the topological $K$-theory of $X$ is naturally isomorphic to the $C^\ast$-algebraic $K$-theory of $C_0(X)$.

%
%

Let us now assume that $X$ is a simplicial complex endowed with the standard simplicial metric. This means if $x$ and $y$ are in the same simplex expressed by convex combinations of its vertices $x=\sum_{j}t_jv_j$ and $y=\sum_j t_j'v_j$, then the distance from $x$ to $y$ is defined to be $\sum_j|t_j-t_j'|$;  if two points are in different simplices, their distance is defined to be the length of the shortest path between them. If there does not exist any path connecting two points, that is, the two points are in two different connected components, then we define their distance to be infinity. We have the following theorem on Lipschtiz control for topological $K$-theory of  simplicial complexes. 
\begin{theorem}\label{thm:introLipschitzC(X)}
	Let $X$ be a locally compact $m$-dimensional simplicial complex equipped with the standard simplicial metric. There exists a constant $L_m$ \textup{(}depending only on $m$\textup{)} such that every class in $K_*(C_0(X))$ admits an $L_m$-Lipschitz representative. Furthermore, if  $\alpha \in K_*(C_0(X))$ can be represented by an element that is constant outside a compact set $K$, then one can choose an $L_m$-Lipschitz representative of $\alpha$ that is  constant outside the $1$-neighborhood of $K$.
\end{theorem}

We remark that a possible upper bound for the constant $L_m$ can be given by  $C\cdot m^{C'm}$, where $C$ and $C'$ are independent of $m$. Theorem \ref{thm:introLipschitzC(X)} is inspired by the  quantitative $K$-theory introduced by the third author in \cite{Yu}. See also \cite{Oyono-OyonoYu}. Although for simplicity we have only stated Theorem $\ref{thm:introLipschitzC(X)}$ for simplicial complexes, the same argument used in its proof in fact generalizes to arbitrary $C^\ast$-algebras that are equipped with a certain type of filtrations. We refer the reader to Appendix \ref{app:QKT-lipschitzfiltration} for a discussion of this more general setting.

The paper is organized as follows. In Section \ref{sec:preliminaries}, we review some geometric $C^\ast$-algebras and higher index theory. In Section \ref{sec:pairingandCAT(0)}, we review an explicit construction for the index pairing between $K$-theory and $K$-homology and relate the index pairing to the rate of decay for scalar curvature (Proposition \ref{prop:pairing-scalarcurvature}). As an application, we prove the quadratic decay inequality for scalar curvature of manifolds that are bi-Lipschitz equivalent to  CAT(0) spaces.
 In Section \ref{sec:Lipschitzcontrol}, we develop a Lipschitz controlled $K$-theory for finite dimensional simplicial complexes (Theorem \ref{thm:introLipschitzC(X)}), then apply it to prove the main result of the article (Theorem \ref{thm:main}) in Section \ref{sec:psc-fad}. In  Appendix \ref{app:QKT-lipschitzfiltration}, we develop a Lipschitz controlled $K$-theory for general $C^\ast$-algebras that are equipped with Lipschitz filtrations.
 
\section{Preliminaries}\label{sec:preliminaries}
In this section, we review the construction of the geometric $C^*$-algebras and the higher index. See \cite{Yulocalization,Roecoarse, willett2020higher} for more details.
\subsection{Roe algebras and localization algebras}\label{subsec:Roealgebra}
We will first review the definitions of some geometric $C^*$-algebras. 

Let $X$ be a proper metric space, i.e., every closed ball is compact. Let $C_0(X)$ be the $C^*$-algebra consisting of all complex-valued continuous functions on $X$ that vanish at infinity.
An $X$-module is a separable Hilbert space $H_X$ equipped with a $*$-representation of $C_0(X)$. It is called \emph{nondegenerate} if the $*$-representation is nondegenerate, and \emph{standard} if no nonzero function in $C_0(X)$ acts as a compact operator. 

For example, if $X$ is a complete Riemannian manifold, then $L^2(X)$ is a non-degenerate standard $X$-module.
	\begin{definition}\label{def:prop-locallycompact}
	Let $H_X$ be a nondegenerate standard $X$-module. Let $T$ be a bounded linear operator acting on $H_X$.
	\begin{enumerate}
		\item The \emph{propagation} of $T$ is defined by
		$$
		\prop(T)=\sup\{d(x,y) ~|~ (x,y)\in \text{supp}(T) \},
		$$
		where $\text{supp}(T)$ is the complement (in $X\times X$) of the set of points $(x,y) \in X\times X$ such that there exists $f_1,f_2\in C_0(X)$ such that $f_1Tf_2=0$ and $f_1(x)f_2(y)\neq 0$;
		\item $T$ is said to be \emph{locally compact} if both $fT$ and $Tf$ are compact for all $f\in C_0(X)$.
	\end{enumerate}
\end{definition}

\begin{definition}\label{def roe and localization}
	Let $H_X$ be a standard nondegenerate $X$-module and $B(H_X)$ the set of all bounded linear operators on $H_X$.
	\begin{enumerate}
		\item The \emph{$Roe$ algebra} of $X$, denoted by $C^*(X)$, is the $C^*$-algebra generated by all locally compact operators with finite propagation in $B(H_X)$.
		\item The \emph{localization algebra} $C_L^*(X)$ is the $C^*$-algebra generated by all bounded and uniformly norm-continuous functions $f\colon [1,\infty)\to C^*(X)$ such that 	
		\[
		\prop(f(t))<\infty \text{ and }\prop(f(t))\to 0\  \text{~as~}  t \to \infty.
		\]
	\end{enumerate}
\end{definition}
The Roe algebra and localization algebra of $X$ are independent (up to isomorphisms) of the choice of nondegenerate standard $X$-modules  $H_X$. 
\subsection{Higher index and local higher index}\label{subsec:higher_index}
In this subsection, we recall the definition of  higher index and  local higher index for Dirac operators. 

Let $\chi$ be a continuous function on $\R$. We say that $\chi$ is a \emph{normalizing function} if it is non-decreasing, odd (i.e. $\chi(-x) = -\chi(x)$) and  
\[ \lim_{x\to \pm \infty} \chi(x) = \pm 1. \]  

Let $X$ be a complete spin manifold and $D$ the  associated Dirac operator on $X$, which acts on the spinor bundle of $X$. Let $H$ be the Hilbert space of the $L^2$-sections of the spinor bundle, which is a standard non-degenerate $X$-module in the sense of Section \ref{subsec:Roealgebra}.
We first assume that $\dim X$ is even. In this case, the spinor bundle is naturally $\Z_2$-graded and the Dirac operator is odd and given by
$$D=\begin{pmatrix}
0&D_+\\D_-&0
\end{pmatrix}$$

Let $\chi$ be a normalizing function. As $\chi$ is odd, $\chi(t^{-1}D)$ is also a self-adjoint odd operator for any $t>0$ given by
\begin{equation}\label{eq:chi}
\chi(t^{-1} D)=\begin{pmatrix}
0&U_{t,D}\\V_{t,D}&0
\end{pmatrix}
\end{equation}
Set
$$W_{t,D}=\begin{pmatrix}
1& U_{t,D}\\0&1
\end{pmatrix}
\begin{pmatrix}
1&0\\-V_{t,D}&1
\end{pmatrix}
\begin{pmatrix}
1&U_{t,D}\\0&1
\end{pmatrix}
\begin{pmatrix}
0&-1\\1&0
\end{pmatrix},\ e_{1,1}=\begin{pmatrix}
1&0\\0&0
\end{pmatrix}
$$
and
\begin{equation}\label{eq:PtD}
\begin{split}
P_{t,D} = & W_{t,D}e_{1,1}W_{t,D}^{-1}\\
 = & \begin{pmatrix}
1-(1-U_{t,D}V_{t,D})^2 & (2-U_{t,D}V_{t,D})U_{t,D}(1-V_{t,D}U_{t,D})\\
V_{t,D}(1-U_{t,D}V_{t,D})&(1-V_{t,D}U_{t,D})^2
\end{pmatrix}.
\end{split}
\end{equation}
The path $(P_{t,D})_{t\in[1,+\infty)}$ defines an element in  $M_2(C^*_L(X)^+)$, and the difference $P_{t,D}-e_{1,1}$ lies in $M_2(C^*_L(X))$.
\begin{definition}\label{def:lind}
	 If $M$ is even dimensional, then  the local higher index $\ind_L(D)$ of $D$ is defined to be  
	 \[ \ind_L(D)\coloneqq [P_{t,D}]-[e_{1,1}] \in K_0(C^*_L(X)). \] The higher index $\ind(D)$ of $D$ is defined to be  
	 \[ \ind(D) \coloneqq [P_{1,D}]-[e_{1,1}] \in  K_0(C^*(X)).\]
\end{definition}

The construction of higher index and local higher index for the odd dimensional case is as follows.
\begin{definition}\label{def:ind-odd}
	If $M$ is odd dimensional, then the local higher index  $\ind_L(D)$ of $D$  is defined to be 
	\[ [e^{2\pi i\frac{\chi(t^{-1}D)+1}{2}}] \in K_1(C^*_L(X)). \] The higher index  $\ind(D)$ of $D$ is defined to be  
	\[ [e^{2\pi i\frac{\chi(D)+1}{2}}] \in K_1(C^*(X)). \]
\end{definition}

Note that the higher index and the local higher index are independent of the choice of normalizing functions. The $K$-theory $K_\ast(C^*_L(X))$ of the localization algebra $C^*_L(X)$ is naturally isomorphic to  the $K$-homology of $X$. Under this isomorphism, the local higher index of $D$ is identified with the $K$-homology class of $D$ (cf. \cite{Yulocalization,QiaoRoe}).

Throughout the paper, we  shall fix the following  specific choice of  normalizing function
$$\chi(x)=\frac 2 \pi\int_0^x \frac{1-\cos y}{y^2}dy-1.$$
Note that this specific function $\chi$ satisfies that
\begin{enumerate}
	\item if $x>0$, then we have $|\chi(x)-1|<\frac{4}{\pi x}$,
	\item the distributional Fourier transform of $\chi$ is supported on $[-1,1]$.
\end{enumerate}
The choice of the above normalizing function is essentially only for convenience. A different choice of normalizing function (whose distributional Fourier transform also has compact support) will not change the discussion that follows but only the constants that appear at various steps. 

The following lemma is an easy consequence from the formula of $P_{t,D}$ from line $\eqref{eq:PtD}$ and the above properties of $\chi$.
\begin{lemma}\label{lemma:PtD} With the same notation as above, the followings hold for all $t\geqslant 1$.
	\begin{enumerate}[$(1)$]
		\item $\|P_{t,D}\|\leqslant 64$.
		\item $\prop(P_{t,D})\leqslant \frac{5}{t}$.
		\item $\|P_{t,D}-e_{1,1}\|\leqslant 5\|1-\chi(t^{-1}D)^2\|$. In particular, if $D^2>\lambda^2\cdot \mathrm{I}$ for some $\lambda>0$, then
		$$\|P_{t,D}-e_{1,1}\|\leqslant \frac{20t}{\pi\lambda}.$$
		\item Let $D_n=D\otimes I_n$. Then 
		\[ \lim_{t\to \infty} \|P_{t,D_n}\cdot f-f\cdot P_{t,D_n}\|=0\] for each $f\in M_n(C_0(X))$. Furthermore, if $f\in M_n(C_0(X))$ is $L$-Lipschitz, then
		$$\|P_{t,D_n}\cdot f-f\cdot P_{t,D_n}\|\leqslant 2560t^{-1}L.$$
	\end{enumerate}
\end{lemma}
\begin{proof}
	As $|\chi|$ is uniformly bounded by $1$, both $U_{t,D}$ and $V_{t,D}$ from line \eqref{eq:chi}  have their norms $\leqslant 1$. Therefore, part (1) follows from the explicit formula of $P_{t,D}$ in line \eqref{eq:PtD}.
	
	The distributional Fourier transform of $\chi(x)$ is supported on $[-1,1]$. By the inverse Fourier transformation formula
	$$\chi(D)=\frac{1}{2\pi }\int \hat{\chi}(\xi) e^{i\xi D}d\xi$$
	and the finite propagation speed of the wave operator $e^{i\xi D}$, we see that $\chi(D)$ has propagation no more than $1$. Replacing $\chi(x)$ by $\chi(t^{-1}x)$, we see that the propagation of $\chi(t^{-1}D)$ is no more than $t^{-1}$. In particular,  $U_{t,D}$ and $V_{t,D}$  from line \eqref{eq:chi} also have propagation no more than $t^{-1}$. Hence part (2) follows from line \eqref{eq:PtD}.

	Each entry of the matrix $P_{t,D}-e_{1,1}$ (cf. line \eqref{eq:PtD}) contains either $(1-U_{t,D}V_{t,D})$ or $(1-V_{t,D}U_{t,D})$ as a factor. Note that from line \eqref{eq:chi}, we have 
	$$1-\chi(t^{-1}D)^2=\begin{pmatrix}
	1-U_{t,D}V_{t,D}&0\\0&1-V_{t,D}U_{t,D}
	\end{pmatrix}.$$
	Hence 
	$$\max\{\|1-U_{t,D}V_{t,D}\|,~\|1-V_{t,D}U_{t,D}\| \}\leqslant \|1-\chi(t^{-1}D)^2\|.$$
Hence follows part (3).
	
 To prove part (4), we consider
	$$g(s)=e^{isD_n}fe^{-isD_n}-f.$$
	We have
	$$g'(s)=e^{isD_n}iD_nf e^{-isD_n}-e^{isD_n}f(iD_n) e^{-isD_n}=ie^{isD_n}[D_n,f] e^{-isD_n}$$
	and $f(0)=0$. Therefore 
	$$[e^{isD_n},f]=g(s)e^{isD_n}=\int_0^s ie^{itD_n}[D_n,f] e^{i(s-t)D_n} dt.$$
	Hence
	$$[\chi(D_n),f]=\frac{1}{2\pi}\int \hat\chi(\xi)\int_0^\xi ie^{itD_n}[D_n,f] e^{i(\xi-t)D_n}dtd\xi.$$
	When $f$ is $L$-Lipschitz, we have 
	$$\|[D_n,f]\|\leqslant L.$$
It follows that 
	$$\max\{\|[U_{1,D_n},f]\|,~\|[V_{1,D_n},f]\|\}=\|[\chi(D_n),f]\|\leqslant 2L$$
Similarly, using $\chi(t^{-1}D_n)$ instead of $\chi(D_n)$,  we have 
	$$\max\{\|[U_{t,D_n},f]\|,~\|[V_{t,D_n},f]\|\}=\|[\chi(t^{-1}D_n),f]\|\leqslant 2t^{-1} L.$$Now part (4) easily follows from this. 

\end{proof}

\section{Decay of scalar curvature on CAT(0)-like spaces}\label{sec:pairingandCAT(0)}
In this section, we discuss how to estimate the scalar curvature via the pairing between Dirac operator and  $K$-theory classes. As an  application, we prove the quadratic decay inequality of scalar curvature for complete manifolds that  are bi-Lipschitz equivalent to CAT(0) spaces.
\subsection{An index pairing formula between $K$-homology and $K$-theory}
In this subsection, we review an explicit construction of the index pairing between $K$-homology and $K$-theory. There are several different ways to formulate the index pairing. They are all equivalent. In particular,  in the case of Dirac operators and compactly supported vector bundles over spin manifolds, the index pairing simply gives the relative indices of the corresponding twisted Dirac operators. The reason for the specific formalism we adopt here is mainly because it  seems to be the most convenient to work with for the geometric applications in this article.

\begin{lemma}\label{lemma:double}
	Let $A$ be a $C^*$-algebra and $I$ an ideal of $A$. Define
	$$D_A(I)=\{(a,a')\in A\oplus A: a-a'\in I \}.$$
	Then there is a natural surjective homomorphism from $K_*(D_A(I))$ to $K_*(I)$. In other words, each $K$-theory class of the ideal $I$ can be expressed by $[\alpha]-[\beta]$ for some  $\alpha,\beta\in M_n(A^+)$ with $\alpha-\beta\in M_n(I)$.
\end{lemma}
\begin{proof}
	Consider the following short exact sequence
	$$0\longrightarrow I\longrightarrow D_A(I)\longrightarrow A\longrightarrow 0,$$
	where the map $I\to D_A(I)$ is given by $a\mapsto (a,0)$. This exact sequence splits, since  there exists a section $A\to D_A(I)$ given by $a\mapsto (a,a)$. Therefore we have $K_*(D_A(I))\cong K_*(I)\oplus K_*(A)$,  and the required surjective homomorphism from $K_*(D_A(I))$ to $K_*(I)$ is given by the projection from $K_*(D_A(I))$ to the first component.
\end{proof}	

Now suppose $[\alpha]-[\beta]\in K_0(I)$ is a $K$-theory class as in Lemma $\ref{lemma:double}$. We shall review the difference construction \cite{KasparovYuconvex}, which gives an explicit way to construct idempotents of the unitization $I^+$ of $I$ which  represent  the same class as $[\alpha]-[\beta]$ in $K_0(I)$.  Without loss of generality, assume $\alpha,\beta$ are idempotents in $A^+$ such that $\alpha-\beta\in I$. Set
$$Z(\beta)=\begin{pmatrix}
\beta &0&1-\beta &0\\ 1-\beta &0&0&\beta\\0&0&\beta&1-\beta\\0&1&0&0
\end{pmatrix}\in M_4(A^+).$$
It is easy to check that $Z(\beta)$ is invertible with inverse
$$Z(\beta)^{-1}=\begin{pmatrix}
\beta &1-\beta &0&0\\
0&0&0&1\\
1-\beta &0&\beta&0\\0&\beta&1-\beta&0
\end{pmatrix}\in M_4(A^+).$$
Let us define 
\begin{equation}\label{eq:diffe}
\begin{split}
d(\alpha,\beta) \coloneqq  & Z(\beta)^{-1}\begin{pmatrix}
\alpha&&&\\
&1-\beta&&\\
&&0&\\
&&&0
\end{pmatrix}Z(\beta)\\=&\begin{pmatrix}
1+\beta(\alpha-\beta)\beta&0&\beta(\alpha-\beta)&0\\
0&0&0&0\\
(\alpha-\beta)\alpha\beta&0&(1-\beta)(\alpha-\beta)(1-\beta)&0\\
0&0&0&0
\end{pmatrix}\in M_4(I^+)
\end{split}
\end{equation}
and 
$$
e_{1,3} \coloneqq \begin{psmallmatrix}
1&&&\\
&0&&\\
&&0&\\
&&&0
\end{psmallmatrix}.
$$
Then $[\alpha]-[\beta] = [d(\alpha,\beta)]-[e_{1,3}] \in K_0(I)$,
where both $d(\alpha,\beta)$ and $e_{1, 3}$ are idempotents in
$M_4(I^+)$.

For a $C^*$-algebra $A$, let $C_{uc}([1,+\infty),A)$ be the $C^*$-algebra generated by uniformly continuous bounded  functions from $[1,+\infty)$ to $A$. Define $C_{0}([1,+\infty),A)$ to be the ideal of $C_{uc}([1,+\infty),A)$ generated by continuous functions that vanish at infinity. We define
\begin{equation}\label{eq:Ainf}
A_\infty=C_{uc}([1,+\infty),A)/C_{0}([1,+\infty),A).
\end{equation}
There is a natural map $i\colon A\to A_\infty$ by mapping each element $a\in A$ to the constant function on $[0, \infty)$ with the  constant value $a$.
\begin{proposition}\label{lemma:Ainf}
	There is a natural map $j_*\colon K_*(A_\infty)\to K_*(A)$ with $j_*i_*=\id$. Furthermore, $j_0$ and $i_0$ are isomorphisms for $K_0$-groups.
\end{proposition}
\begin{proof}
	Consider the short exact sequence
	$$0\to C_{0}([1,+\infty),A)\to C_{uc}([1,+\infty),A)\to A_\infty\to 0.$$
	As $C_{0}([1,+\infty),A)$ is contractible, we have $K_*(A_\infty)\cong K_*(C_{uc}([1,+\infty),A))$. The map $j_*\colon K_*(A_\infty) \to K_*(A)$ is defined  by the evaluation map
	$$\ev\colon C_{uc}([1,+\infty),A)\to A,\ f\mapsto f(1),$$
followed by the isomorphisms $K_*(A_\infty)\cong K_*(C_{uc}([1,+\infty),A))$ above. 

	To prove that $j_0$ and $i_0$ are isomorphisms, it suffices to show that the 
	\[ K_0(C_{uc,0}([1,+\infty),A)) = 0, \] where $C_{uc,0}([1,+\infty),A)$ is the ideal of $C_{uc}([1,+\infty),A)$ consisting of  functions  that vanish at $t = 0$. Without loss of generality, we assume that $A$ is unital. Let $[p]-[e]$ be a $K$-theory element in $K_0(C_{uc,0}([1,+\infty),A))$, where both $p$ and $e$ are projections in $M_n(C_{uc}([1,+\infty),A))$ with $p-e\in  M_n(C_{uc,0}([1,+\infty),A))$ and $e$ a matrix-valued constant function on $[1,+\infty)$. In particular, we have $p(1)=e$. As $p$ is uniformly continuous, there exists $\varepsilon>0$ such that $\|p(x)-p(y)\|<1/2$ if $|x-y|\leqslant \varepsilon$.

Let us define the following intervals for $n\in \N$
$$\begin{cases}
I_0=\{1\};\\
I_n=\{x\in[1,+\infty):1+(n-1)\varepsilon\leqslant x\leqslant 1+n\varepsilon \}.
\end{cases}
$$
We claim that there exists a family of equicontinuous unitary-valued maps
$$\{u_n\colon I_n\to M_k(A) \}_{n\in \N}$$
such that
\begin{enumerate}
	\item $u_{n}(1+n\varepsilon)=u_{n+1}(1+n\varepsilon)$,
	\item $p(t)=u_n(t)eu_n(t)^*$, $\forall t\in I_n$.
\end{enumerate}

We will construct this family by induction on $n$. When $n=0$, $u_0(1)=1_A$ is as required, since $p(1)=e$. Assume that $u_{n}$ is already defined. For each $t\in I_{n+1}$, we define
\begin{align*}
z(t)\coloneqq& 2p(t)p(1+n\varepsilon)-p(t)-p(1+n\varepsilon)+1\\
=&1+\big(p(t)-p(1+n\varepsilon)\big)\big(2p(1+n\varepsilon)-1\big).
\end{align*}
Since $\|p(t)-p(1+n\varepsilon)\|<1/2$ for all $t\in I_{n+1}$, we see that $z(t)$ is invertible. Set 
$$v(t)=z(t)\big(z^*(t) z(t)\big)^{-1/2}.$$ 
It follows that $v(t)$ is unitary and $p(t)=v(t)p(1+n\varepsilon)v(t)^*$ for any $t\in I_{n+1}$. Now that $u_n\colon I_n\to M_k(A)$ is already defined and 
$$p(1+n\varepsilon)=u_n(1+n\varepsilon)eu_n(1+n\varepsilon)$$
by the induction hypothesis, we define
$$u_{n+1}(t)=v(t)u_n(1+n\varepsilon).$$	
It is easy to verify that $u_n$ are equicontinuous for all $n$. This completes the construction of the family $\{u_n\}_{n\in \N}$.

Now we define a unitary $u\in M_k(C_{uc,0}([1,\infty),A))^+$ by 
\[ u(t) \coloneqq  u_n(t), \textup{ if } t\in I_n.\] 
Since  we have $p=ueu^*$ by construction, it follows that \[ [p]-[e] = 0 \in K_0(C_{uc,0}([1,\infty),A)). \]
This finishes the proof. 

\end{proof}
\begin{remark}
	For $K_1$-groups, the map $j_1\colon  K_1(A_\infty)\to K_1(A)$ is not an isomorphism in general. However, if $A$ is stable, i.e. $A\otimes \cK\cong A$, then $j_1$ is  an isomorphism. This for example can be proved by a standard Eilenberg swindle argument.
\end{remark}

Let $X$ be a complete spin Riemannian manifold and $D$ the Dirac operator on $X$. As the $K$-theory of $C^*_L(X)$ is naturally isomorphic to the $K$-homology of $X$ (cf. \cite{Yulocalization,QiaoRoe}), the local higher index of $D$ naturally pairs with the $K$-theory $K^\ast(X)$ of $X$. In the following, we recall a concrete pairing formula between  $\ind_L(D)$ and elements of $K^*(X)$ from \cite[\S 9.1]{willett2020higher}. We will only give the details for the even dimensional case. The odd dimensional case can for example be dealt with by a standard suspension argument.

Let $[p]-[q]$ be a class in $K_0(C_0(X))$, where $p,q$ are projections in $M_n(C_0(X)^+)$ such that $p-q$ lies in $M_n(C_0(X))$. Recall that the local higher index $\ind_L(D)$ is represented by the formal difference of two  idempotents 
\[ [P_{t,D}]-[e_{1,1}]\in M_2(C^*_L(X)^+), \] where $e_{1,1}=\begin{psmallmatrix}1&0\\0&0\end{psmallmatrix}$ and 
$P_{t,D}-e_{1,1} \in  M_2(C^*_L(X))$ (cf. Definition $\ref{def:lind}$).  
For simplicity, let us set $P_t = P_{t, D_n}$, where  $D_n=D\otimes I_n$. From part (4) of Lemma \ref{lemma:PtD}, we have
$$\|(P_{t}\cdot p)^2-P_{t}\cdot p\| \to 0 \textup{ and }  \|(P_{t}\cdot q)^2-P_{t}\cdot q\| \to 0,  \textup{ as }   t\to \infty. $$

Recall that for every $C^*$-algebra $A$, we define
$$A_\infty \coloneqq C_{uc}([1,+\infty),A)/C_{0}([1,+\infty),A)$$
as in line \eqref{eq:Ainf}. By the local compactness condition in Definition \ref{def:prop-locallycompact}, we see that $P_{t}\cdot  p$ and $P_{t}\cdot q$ are idempotents in $M_{2n}((\cK+C_0(X)+C^*(X))_\infty^+)$, where   $\cK+C_0(X)+C^*(X)$ is the $C^*$-algebra in $B(H)$ generated by $\cK$, $C_0(X)$ and $C^*(X)$.

Now we apply the difference construction in line \eqref{eq:diffe} to the ideal sequence
$$\cK_\infty \triangleleft (\cK+C_0(X))_\infty \triangleleft 	(\cK+C_0(X)+C^*(X))_\infty.$$
Let $a_{t,p,q} = d(P_t\cdot p,P_t\cdot q)$ be the difference idempotent of  $P_t\cdot p$ and $P_t\cdot q$ obtained by the difference construction in \eqref{eq:diffe}. Similarly, we denote $b_{p,q} = d(e_{1,1}\otimes p,e_{1,1}\otimes q)$, where $e_{1,1}=\begin{psmallmatrix}1&0\\0&0\end{psmallmatrix}$.  By construction, the difference idempotent 
$d(a_{t,p,q}, b_{p,q})$ is an idempotent in $M_{32n}((\cK_\infty)^+)$ and 
$$d(a_{t,p,q}, b_{p,q})-e_{1,3}\otimes I_{2n}\otimes I_4\in M_{32n}(\cK_\infty),$$
where $I_{m}$ denotes the $(m\times m)$ identity matrix  and 
\[ e_{1,3} = \begin{psmallmatrix}
	1 & & & \\ 
	& 0 & & \\
	& & 0 & \\
	& & & 0
\end{psmallmatrix}. \] 
Therefore, the pairing of $[D]$ and $[p]-[q]$ is given by the image of
$$[d(a_{t,p,q}, b_{p,q})]-[e_{1,3}\otimes I_n\otimes I_4]\in
K_0(\cK_\infty)$$
under the isomorphism $j_0\colon K_0(\cK_\infty)\to K_0(\cK)\cong\Z$.

For notional simplicity, let us denote 
$$d_{t,p,q}\coloneqq d(a_{t,p,q}, b_{p,q})\text{  and }e\coloneqq e_{1,3}\otimes I_n\otimes I_4.$$ We have the following lemma. 
\begin{lemma}\label{lemma:key} 
	With the same notation as above, there exist positive numbers\footnote{The constants  $\{\lambda_i:i=1,2,3,4\}$ are independent of $t\in [1, \infty)$ and $n\in \mathbb N$, where $p$ and $q$ are matrices of size $(n\times n)$. } $\{\lambda_i:i=1,2,3,4\}$ such that the following are satisfied. 
	\begin{enumerate}[$(1)$]
		\item $\|d_{t,p,q}\|\leqslant \lambda_1$ for all $t\geq 1$. 
		\item We have 
		$$\lim_{t\to\infty} \|d_{t,p,q}^2-d_{t,p,q}\|=0.$$
		Furthermore, if $p,q \in M_n(C_0(X)^+)$ are Lipschitz functions with Lipschitz constant $L$, then
		$$\|d_{t,p,q}^2-d_{t,p,q}\|\leqslant \frac{\lambda_2L}{t}$$
		for all $t\geqslant 1$. 
		\item If $p-q$ is supported in $B(x_0,R)$ for some $x_0\in X$ and $R>0$, then
		$$\|d_{t,p,q}-e\|\leqslant \frac{\lambda_3t}{\sqrt{\max\{k_{x_0}(R+\frac{\lambda_4}{t}),0\}}}$$
		for all $t\geqslant 1$, 
		where 
		$$k_{x_0}(r)=\inf\{k(x): x\in B(x_0,r) \}$$
		with $k(x)$ the scalar curvature at $x\in X$.
	\end{enumerate}
\end{lemma}
\begin{proof}
	By construction,   $d_{t,p,q}$ is a $(16\times 16)$ matrix with entries in $M_{2n}(\cK^+)$ and $d_{t,p,q}-e$ lies in $ M_{16}(M_{2n}(\cK))$ for all $t\geqslant 1$. Furthermore, each entry of $d_{t,p,q}-e$ can be expressed by some polynomial of $P_t$, $e_{1,1}\otimes I_n$, $p$ and $q$.

Note that $P_t = P_{t, D_n}$ is uniformly bounded for  $t\in [1, \infty)$, thus follows part (1). 

 Observe that  $d_{t,p,q}$ would actually be an idempotent if   $P_t$ were to commute with $p$ and $q$. It follows that every non-zero entry of $d_{t,p,q}^2-d_{t,p,q}$ is of the form
	$$\sum_{j=1}^m a_j(P_t\cdot p-p\cdot P_t)b_j+\sum_{j=1}^{m'} a_j'(P_t\cdot q-q\cdot P_t)b_j'.$$
	for some $a_j, b_j, a'_j$ and $b'_j$ that are expressed by polynomials of $P_t$, $e_{1,1}\otimes I_n$, $p$ and $q$. 
	Now part (2) follows from part (4) of Lemma \ref{lemma:PtD}. 
	
	It remains to show part (3). If we view 
	\[ d(P_t\cdot p,P_t\cdot q)-e_{1,3}\otimes I_{2n} \textup{ and }  d(e_{1,1}\otimes p,e_{1,1}\otimes q)-e_{1,3}\otimes I_{2n} \] as $(4\times 4)$ matrices on $M_{2n}(\cK)$, then it follows from the explicit formula in line \eqref{eq:diffe} that every nonzero entry of both matrices contains a factor $(p-q)$. Therefore, every non-zero entry of the matrix $d_{t,p,q}-e$ is of the  form
	$$\sum_{j=1}^m f_j(p-q)g_j(P_t-e_{1,1})h_j+\sum_{j=1}^{m'} f_j'(P_t-e_{1,1})g_j'(p-q)h'_j$$
	for some $f_j, g_j, h_j, f'_j, g'_j$ and $h'_j$ that are expressed by polynomials of $P_t$, $e_{1,1}\otimes I_n$, $p$ and $q$. 
	Recall that $P_t$ has finite propagation with $\prop(P_t) \leq 5/t$ (cf. Lemma $\ref{lemma:PtD}$). In particular, it follows that there exists a positive number $c$ such that the propagations of  $f_j, g_j, h_j, f'_j, g'_j$ and $h'_j$ are all bounded by $c/t$. Now we shall apply part (3) of Lemma $\ref{lemma:PtD}$ to estimate the norm of $d_{t,p,q}-e$, which in turn requires us to estimate the lower bound for $D^2$.  If  $p-q$ is supported in $B(x_0,R)$, then by finite propagation there exists some $\lambda_4 >0$ such that it suffices to estimate the lower bound of $D^2$ on $B(x_0,R+\frac{\lambda_4}{t})$. Now by the Lichnerowicz's formula
	\[ D^2 = \nabla^\ast\nabla  + \frac{k}{4} \geqslant \frac{k}{4}, \]
	we have 
	\[ D^2 \geqslant \frac{k_{x_0}(R + \frac{\lambda_4}{t})}{4} \textup{ on } { \textstyle B(x_0,R+\frac{\lambda_4}{t}) },  \]
	  where $k_{x_0}(r)=\inf\{k(x): x\in B(x_0,r) \}.$
	Now part (3) follows from part (3) of Lemma \ref{lemma:PtD}. 
\end{proof}

\subsection{Quadratic decay of scalar curvature on CAT(0)-like spaces}
In this subsection, we relate the rate of decay for scalar curvature to the index pairing between $K$-homology and $K$-theory. As an immediate application, we  prove the quadratic decay inequality for scalar curvature on CAT(0)-like spaces.

\begin{proposition}\label{prop:pairing-scalarcurvature}
	Let $X$ be an even-dimensional complete spin Riemannian manifold and $x_0\in X$. Let $[D]$ be the K-homology class of the Dirac operator and $[p]-[q]\in K^0(X)=K_0(C_0(X))$  a K-theory class of $X$, where $p,q$ are projections in $M_n(C_0(X)^+)$. Then there are universal positive constants $C_1,C_2$ such that if
	\begin{itemize}
		\item $p-q$ is supported in $B(x_0,R)$,
		\item $p,q$ are $L$-Lipschitz functions on $X$,
		\item the index pairing between $[D]$ and $[p]-[q]$ is non-zero,
	\end{itemize}
then $$\textstyle k_{x_0}(R+\frac{C_1}{L})\leqslant C_2\cdot L^2,$$
where $k_{x_0}(r)=\inf_{x\in B(x_0,r)}k(x)$.
\end{proposition}
\begin{proof}
		Recall that an explicit formula for the index pairing between $[D]$ and $[p]-[q]$ is given by (cf. Lemma \ref{lemma:key})
	$$[d_{t,p,q}]-[e]\in K_0(\cK_\infty)\cong \Z, $$
	which is non-zero by our assumption. We define a function on $\C$ by
	$$\Theta(z)=\begin{cases}
	0& \mathrm{Re}(z)<\frac 1 2;\\
	1& \mathrm{Re}(z)\geqslant \frac 1 2.\\
	\end{cases}$$
	Let $\{\lambda_i : i = 1, 2, 3, 4\}$ be the constants from Lemma $\ref{lemma:key}$. By part (2) of Lemma \ref{lemma:key}, if $t>4\lambda_2L$, then
	$$\| d_{t,p,q}^2-d_{t,p,q}\|< \frac{1}{4}.$$
	It follows that $\Theta$ is a holomorphic function on the spectrum of $d_{t,p,q}$ in this case. By the holomorphic functional calculus, for any $t>4\lambda_2L$,  we define
	$$\Theta(d_{t,p,q})=\frac{1}{2\pi i}\int_\Gamma (d_{t,p,q}-\xi)^{-1}d\xi,$$
	where $\Gamma=\{z:|z-1|=1/2 \}$.  Note that 
	$\Theta(d_{t,p,q})$ is an idempotent for any $t>4\lambda_2L$. The image of $[d_{t,p,q}]-[e]$ in $ K_0(\cK)$ under the isomorphism $K_0(\cK_\infty)\cong K_0(\cK)$ can be represented by 
	$$ [\Theta(d_{t,p,q})]-[e]\in K_0(\cK)$$
	with any $t>4\lambda_2L$.
	\begin{claim*} For any $t>4\lambda_2L$, we  have 
		$$\frac{\lambda_3t}{\sqrt{k_{x_0}(R+\frac{\lambda_4}{t})}}\geqslant\frac 1 4.$$
	\end{claim*}	
If we assume the claim for the moment, then we can conclude that 
$$
\textstyle k_{x_0}(R+\frac{\lambda_4}{4\lambda_2L})\leqslant 256 \lambda_2^2\lambda_3^2 \cdot L^2
$$ 
by letting $t$ go to $4\lambda_2L$. This would finish the proof by setting $C_1 = \frac{\lambda_4}{4\lambda_2}$ and $C_2 = 256 \lambda_2^2 \lambda_3^2$. Hence it remains to prove the claim.

	Assume to the contrary that  
	$$\frac{\lambda_3t_0}{\sqrt{k_{x_0}(R+\frac{\lambda_4}{t_0})}}<\frac 1 4,$$  
	for some $t_0>4\lambda_2L$. We have 
	\begin{align}
	\Theta(d_{t_0,p,q})-e=&
	\frac{1}{2\pi i}\int_\Gamma ((d_{t_0,p,q}-\xi)^{-1}-(e-\xi)^{-1})d\xi \notag \\
	=&\frac{1}{2\pi i}\int_\Gamma (d_{t_0,p,q}-\xi)^{-1}(e-d_{t_0,p,q})(e-\xi)^{-1}d\xi. \label{eq:holo}
	\end{align}
	Note that
	$$
	(d_{t_0,p,q}-\xi)^{-1}=(e-\xi)^{-1}(1+(d_{t_0,p,q}-e)(e-\xi)^{-1})^{-1}.
 	$$ 
		By part (3) of Lemma \ref{lemma:key}, we have 
	$$\|d_{t_0,p,q}-e\|\leqslant \frac{\lambda_3t_0}{\sqrt{k_{x_0}(R+\frac{\lambda_4}{t_0})}}<\frac 1 4.$$
Also, $\|(e-\xi)^{-1}\|\leqslant 2$ for all $\xi\in \Gamma=\{z:|z-1|=1/2 \}$, since $e$ is a projection.  It follows that 
	$$\|(d_{t_0,p,q}-\xi)^{-1}\|<4.$$
Applying the above estimates to the integral in line $\eqref{eq:holo}$, we conclude that 
\begin{equation}\label{eq:idemestimate}
	\|\Theta(d_{t_0,p,q})-e\| <1.
\end{equation}
Recall that if two idempotents $f_1$ and $f_2$ in a $C^\ast$-algebra  satisfies the inequality 
\[ \|f_1 - f_2\| < \frac{1}{\|2f_1 - 1\|}, \]
then $f_1$ is equivalent to $f_2$ (cf. \cite[Proposition  4.3.2]{Blackadar}). In our case, since $e$ is a projection, we have 
\[ \|2e-1\| = 1. \]	In particular, the inequality in line $\eqref{eq:idemestimate}$ implies that 
\[ \|\Theta(d_{t_0,p,q})-e\|<\frac{1}{\|2e -1\|}. \]
It follows that $\Theta(d_{t_0,p,q})$ is equivalent to $e$. 
	Thus $[\Theta(d_{t_0,p,q})]-[e] = 0 \in K_0(\cK)$. This contradicts the assumption that  $[\Theta(d_{t_0,p,q})]-[e]$ is nonzero in $K_0(\cK)$. This proves the claim, hence the proposition. 
\end{proof}

Now we are ready to apply Proposition \ref{prop:pairing-scalarcurvature} to obtain the quadratic decay  of scalar curvature for manifolds that are bi-Lipschitz equivalent to CAT(0) spaces. Let us recall the notion of CAT(0) spaces.

Let $X$ be a geodesic metric space, i.e., for any two points there is a geodesic between them. For any geodesic triangle $\triangle xyz$ in $X$, a triangle $\triangle \bar{x}\bar{y}\bar{z}$ in the standard Euclidean space $\R^2$ is called a comparison triangle of $\triangle xyz$ if 
\[ d_X(x,y)=d_{\mathbb R^2}(\bar{x}, \bar{y}), d_X(y,z)=d_{\mathbb R^2}(\bar{y}, \bar{z}), \textup{ and }  d_X(z,x)=d_{\mathbb R^2}(\bar{z},\bar{x}).  \]	
A point $\bar{a}$ on the geodesic segment $[\bar x, \bar y]$ is called a comparison point of $a\in [x, y]$ if $d_X(x, a) = d_{\mathbb R^2}(\bar x, \bar a)$. Comparison points on $[\bar y, \bar z]$ and $[\bar z, \bar x]$ are defined in the same way. 
\begin{definition}
Let $X$ be a geodesic metric space. Given a geodesic triangle $\triangle$ in $X$, let $\overline \triangle$ be a comparison triangle of $\triangle$. The triangle $\triangle$ is said to satisfy  the  CAT(0) inequality if for all $a, b\in \triangle$ and all comparison points $\bar a, \bar b \in \overline \triangle$, we have 
\[ d_X(a, b) \leqslant d_{\mathbb R^2}(\bar a, \bar b).  \]
Now $X$ is said to be a CAT(0) space all of its geodesic triangles  satisfy the  CAT(0) inequality.
\end{definition}

Suppose $X$ is a CAT(0) space. Fix a point $x_0\in X$. For any $x\in X$, there is a unique geodesic $\gamma_{x}\colon [0,d(x,x_0)]\to X$ with $\gamma(0)=x_0$ and $\gamma(d(x,x_0))=x$. This allows us to define the following homotopy 
\begin{equation}\label{eq:homotopyCAT0}
H\colon X\times[0,+\infty)\to X,\ 
H(x,t)=\gamma_x\big(t^{-1}d(x,x_0)\big).
\end{equation}
Since $X$ is a CAT(0) space, we have
$$
d(H(x,t),H(y,t))\leqslant t^{-1}d(x,y), 
$$
which implies $X$ is uniformly contractible. In particular,  we have the following proposition (cf.  \cite[Corollary 9.6.12]{willett2020higher}).
\begin{proposition}\label{prop:nontrivial_cat0}
	Let $X$ be an $n$-dimensional complete Riemannian manifold. If $X$ is bi-Lipschitz equivalent to a \textup{CAT(0)} space, then
	$$K_n(C^*_L(X))\cong K_n(X)\cong Hom(K^n(X),\mathbb Z)\cong\mathbb Z,$$
	which is generated by the local higher index of the Dirac operator on $X$.
\end{proposition}
\begin{theorem}\label{thm:cat0}
	Let $X$ be a complete Riemannian manifold. If $X$ is bi-Lipschitz equivalent to a \textup{CAT(0)} space, then the scalar curvature of $X$ has quadratic decay, i.e., for any $x_0\in X$ there exists $C>0$ such that 
	$$\inf_{x\in B(x_0,r)} k(x)\leqslant \frac{C}{r^2}.$$
\end{theorem}
\begin{proof}
	If $X$ is an odd dimensional manifold, we can consider $X\times\R$ instead, which is still bi-Lipschitz equivalent to a CAT(0) space. Thus without loss of generality, we assume that $\dim X$ is even. We fix a point $x_0$ in $X$ and write
	$$k_{x_0}(r)=\inf_{x\in B(x_0,r)} k(x).$$
	
	Let $d_X$ be the Riemannian metric on $X$ and $d_0$ a CAT(0) metric on $X$. By assumption, there are positive constants $L_1$ and $L_2$ such that 
	$$L_1d_0(x,y)\leqslant d_X(x,y)\leqslant L_2d_0(x,y),~\forall x,y\in X.$$
	
		Fix $R>0$. Let $[p]-[q]$ be a representative of a  generator  (usually called a Bott element) of $K_0(C_0(X)) \cong K_0(C_0(B(x_0,R))) \cong  \mathbb Z$  such that both $p$ and $q$ are $L$-Lipschitz functions in $M_n(C_0(X)^+)$, and $p-q\in M_n(C_0(X))$ is supported in $B(x_0,R)$. 
		
	Let $H$ be the homotopy on $X$ defined as in line \eqref{eq:homotopyCAT0} by using the metric $d_0$. Define the map $H_s \colon X\to X$ by setting $H_s(x) = H(x, s)$.  Let us denote 
	$$
	p_s= H^\ast_s(p) = p\circ H_s\text{ and } q_s=H^\ast_s(q) = q\circ H_s.
	$$
	Then we see that
	\begin{enumerate}
		\item $\supp(p_s-q_s)$ is contained in $B(x_0,\frac{sRL_2}{L_1})$;
		\item $p_s$ and $q_s$ are Lipschitz functions on $X$ with Lipschitz constant $\frac{L_2}{sL_1} \cdot L$;
		\item $[p_s]-[q_s]=[p]-[q] \neq 0$ in  $K_0(C_0(X))$, which in particular implies that the index pairing of $[D]\in K_0(X)$ between $[p_s]-[q_s]$ is non-zero, because of Proposition \ref{prop:nontrivial_cat0}.
	\end{enumerate}
It follows from Proposition \ref{prop:pairing-scalarcurvature} that 
$$\textstyle k_{x_0}\big((\frac{RL_2}{L_1}+\frac{C_1 L_1}{LL_2}) \cdot s\big)\leqslant \frac{C_2L^2L_2^2}{s^2L_1^2}.$$
Now the theorem follows by setting
$$\textstyle C=\frac{C_2L^2L_2^2}{L_1^2}\big( \frac{RL_2}{L_1}+\frac{C_1 L_1}{LL_2}\big)^2.$$
\end{proof}

\section{Lipschitz control for topological $K$-theory of simplicial complexes}\label{sec:Lipschitzcontrol}
In this section, we prove Theorem \ref{thm:introLipschitzC(X)}, which  gives a Lipschitz controlled $K$-theory for locally compact finite dimensional simplicial complexes.

\subsection{Lipschitz controlled $K$-groups}\label{subsec:LipQKT}

In this subsection, as a preparation, we introduce a notion of $C^*$-algebras with Lipschitz filtration and fix some notation. For simplicity, we shall mainly focus our discussion on commutative $C^\ast$-algebras, that is, $C_0(X)$ of some locally compact Hausdorff space $X$. For general $C^\ast$-algebras, we refer the reader to Appendix \ref{app:QKT-lipschitzfiltration}.

Let $(X,d)$ be a locally compact metric space and $C_0(X)$ the algebra of continuous function on $X$ that vanish at infinity (equipped with the sup-norm). An element in $M_n(C_0(X)) = M_n(\mathbb C)\otimes C_0(X)$ is said to be $L$-Lipschitz if 
\begin{equation}\label{eq:lipschitz}
\|f(x)-f(y)\|\leqslant L\cdot d(x,y),\ \forall x,y\in X,
\end{equation}
where the norm on the left hand side is the operator norm of matrices. We denote by $C_0(X)_L$ (resp. $M_n(C_0(X))_L$) the collection of $L$-Lipschitz functions in $C_0(X)$ (resp. $M_n(C_0(X))$). The collection $\{ C_0(X)_L\}_{L\geqslant 0}$ gives a Lipschitz filtration of $C_0(X)$ (cf. Definition \ref{def:Lipfiltration}). 

Denote the unitization of $C_0(X)$ by $C_0(X)^+$, which is precisely the algebra of continuous functions on the one point compactification of $X$. If $X$ is non-compact, we define $\pi\colon C_0(X)^+\to \C$ to be the homomorphism that maps the extra unit to $1$ and kills $C_0(X)$. Equivalently, $\pi$ is the evaluation map at the point of infinity.
\begin{itemize}
	\item For any $L\geqslant 0$, let $P^L_n(C_0(X)^+)$ be the set of projections in $M_n(C_0(X)^+)_L$, with the following natural inclusion given by
	$$P_n^L(C_0(X)^+)\to P_{n+1}^L(C_0(X)^+),\ p\mapsto \begin{pmatrix}
	p &0\\0&0
	\end{pmatrix}.$$
	\item For any $L\geqslant 0$, let $U_n^L(C_0(X)^+)$ be the set of unitaries in $M_n(C_0(X)^+)_L$, with the following natural inclusion given by
	$$U_n^L(C_0(X)^+)\to U_{n+1}^L(C_0(X)^+),\ u\mapsto \begin{pmatrix}
	u &0\\0&1
	\end{pmatrix}.$$
\end{itemize}
We define 
\[  P^L(C_0(X)^+) = \varinjlim P^L_n(C_0(X)^+) \textup{ and } U^L(C_0(X)^+) = \varinjlim U^L_n(C_0(X)^+)\]
under the natural inclusions from above. 
We define the following equivalence relations. 
\begin{itemize}
	\item For $(p,k)$ and $(q,k')$  in $P^L(C_0(X)^+)\times\N$, we say  $(p,k)\sim (q,k')$ if $p\oplus I_{j+k'}$ and $q\oplus I_{j+k}$ are homotopic in $P^{2L}(C_0(X)^+)$ for some $j\in\N$, where $\oplus$ means direct sum.
	\item For $u$ and $v$ in $U^L(C_0(X)^+)$, we say $u\sim v$ if $u$ and $v$ are homotopic in $U^{2L}(C_0(X)^+)$.
\end{itemize}
\begin{definition}
	We define the Lipschitz controlled $K$-groups of $X$ by 
	\begin{align*}
	    K_0^L(C_0(X))&\coloneqq \{(p,k)\in P^L(C_0(X)^+):\mathrm{rank}(\pi(p))=k \}/\sim,\\
	    K_1^L(C_0(X))&\coloneqq U^L(C_0(X)^+)/\sim.
	\end{align*}

\end{definition}
A prior, $K_0^L(C_0(X))$ and $K_1^L(C_0(X))$ are only semigroups under direct sum. 
The fact that $K_*^L(C_0(X))$ are actually abelian groups for any $L\geqslant 0$ will be verified in Lemma \ref{lemma:K_*^Lisabelian}. 

Let us now recall the definition of inductive systems. 
	\begin{definition}\label{def:inducsys1}
Suppose  $\{M_L\}_{L\geqslant 0}$ is a collection of abelian groups  equipped with a collection of group homomorphisms $\{i_{L,L'}\}$, where 
	$$i_{L,L'}\colon M_L\to M_{L'}$$
is defined whenever   $L<L'$.  We say $\{M_L\}_{L\geqslant 0}$ is 
	an \emph{inductive system}	 if 
	\[ i_{L,L''}=i_{L',L''}\circ i_{L,L'},\] whenever $L<L'<L''$. 
\end{definition}

We have the following notion of controlled homomorphisms between inductive systems. 
\begin{definition}\label{def:controlledhom1}
	A function $F\colon\R_{\geqslant 0}\to \R_{\geqslant 0}$ is called a \emph{control function} if $F$ is non-decreasing and  $F(x) \to \infty$, as $x\to \infty$. Let $\{M_L\}_{L\geqslant 0}$ and $\{M_L'\}_{L\geqslant 0}$ be two inductive systems. We say a collection of group homomorphisms $\{\xi_L\colon M_L\to M_{F(L)}'\}$ is a \emph{controlled homomorphism} from $\{M_L\}_{L\geqslant 0}$ to $\{M_L'\}_{L\geqslant 0}$   with control function $F$ if the following diagram commutes:
	$$\xymatrix{M_L\ar[rr]^{\xi_L}\ar[d]_{i_{L,L'}}&&M_{F(L)}'\ar[d]^{i'_{F(L),F(L')}}\\
		M_{L'}\ar[rr]^{\xi_{L'}}&&M_{F(L')}'
	}
	$$
	for all $L< L'$.  From now on, we shall denote  such a controlled homomorphism by  $\xi\colon \{M_L\}_{L \geqslant 0} \xlongrightarrow[F]{} \{M_L'\}_{L \geqslant 0}$ or
	${\{M_L\}_{L \geqslant 0}\xlongrightarrow[F]{\xi}\{M_L'\}_{L \geqslant 0} }$. If the control function $F$ is clear from the context, we will simply write  $\xi\colon \{M_L\}_{L \geqslant 0}\to \{M_L'\}_{L \geqslant 0}$ or
	${\{M_L\}_{L \geqslant 0}\xlongrightarrow{\xi}\{M_L'\}_{L \geqslant 0} }$ instead. 
\end{definition}
We shall also need the following notion of controlled equivalences between two controlled homomorphisms.  
	\begin{definition}\label{def:similar1}
	Let $F$, $G$ and $H$ be control functions such that $H(x)\geqslant G(x)$ and $H(x)\geqslant F(x)$ for all $x\in\mathbb R_{\geqslant 0}$. Given two controlled homomorphisms		$\xi\colon \{M_L\}_{L \geqslant 0}\xlongrightarrow[F]{}\{M_L'\}_{L \geqslant 0}$ and $\eta\colon \{M_L\}_{L \geqslant 0} \xlongrightarrow[G]{}\{M_L'\}_{L \geqslant 0} $, we say $\xi$ is controlled equivalent to $\eta$ with control function $H$ if the following diagram commutes:
	$$\xymatrix{M_L\ar[rr]^{\xi_L}\ar[d]_{\eta_L}&&M_{F(L)}'\ar[d]^{i'_{F(L),H(L)}}\\
		M_{G(L)}'\ar[rr]^{i'_{G(L),H(L)}}&&M_{H(L)}'
	}
	$$
	In this case, we write $\xi\sim_H\eta$ or simply $\xi\sim\eta$.
\end{definition}	
Now let us introduce the following notion of asymptotically exact sequences. 
\begin{definition}\label{def:asymexact1}
Let $F$, $G$, $F_1$ and $F_2$ be control functions.  Given two controlled homomorphisms		$\xi\colon \{M_L\}_{L \geqslant 0}\xlongrightarrow[F]{}\{M_L'\}_{L \geqslant 0}$ and $\eta\colon \{M'_L\}_{L \geqslant 0}\xlongrightarrow[G]{}\{M_L''\}_{L \geqslant 0} $, we say  the sequence 
$$\{M_L\}_{L \geqslant 0}\xlongrightarrow[F]{\xi}\{M_L'\}_{L \geqslant 0}\xlongrightarrow[G]{\eta}\{M_L''\}_{L \geqslant 0} $$
is \emph{asymptotically exact} at $\{M_L'\}_{L \geqslant 0}$ with control functions $F_1$ and $F_2$ if
\begin{itemize}
	\item $\eta\circ\xi\sim_{F_1}0$;
	\item for any $m'\in M_L'$ with $\eta_L(m')=0$, there exists $m\in M_{F_2(L)}$ such that 
	\[   \xi_{F_2(L)} (m) = i'_{L, F(F_2(L))}(m')\] 
	in $M_{F(F_2(L))}'$.
\end{itemize} 
\end{definition}

In the case of a locally compact metric space $X$, we have the natural inclusion $C_0(X)_L\subset C_0(X)_{L'}$ for $L<L'$, which induces a homomorphism 
$$i_{L,L'}\colon K_*^L(C_0(X))\longrightarrow K_*^{L'}(C_0(X)).$$
In particular, $\{K_*^L(C_0(X))\}_{L\geqslant 0}$ together with the homomorphisms $\{i_{L, L'}\}_{0\leqslant L<L'}$ is an inductive system, whose inductive limit is 
\[
\varinjlim K_*^L(C_0(X)) = K_*(C_0(X)), \]
since the union $\bigcup_{L \geqslant 0} C_0(X)_L$ is dense in $C_0(X)$,

Let $X$ and $Y$ be two locally compact metric spaces. Suppose $\xi\colon Y\to X$ is a $L_0$-Lipschitz proper map, then $\xi^*\colon C_0(X)\to C_0(Y)$ induces a controlled homomorphism $\xi^*\colon\{K_*^L(C_0(X))\}_{L \geqslant 0} \xlongrightarrow[F]{} \{K_*^L(C_0(Y))\}_{L \geqslant 0}$ with control function $F$ given by $F(L) = L_0\cdot L$. Furthermore, if $\xi_t \colon Y \to X$, $t\in [0, 1]$,  is a continuous family of $L_0$-Lipschitz proper maps,  then the induced controlled homomorphisms  
\[ \xi^*_t \colon\{K_*^L(C_0(X))\}_{L\geqslant 0}\to \{K_*^L(C_0(Y))\}_{L\geqslant 0}\] 
is independent of $t\in [0, 1]$. 

\begin{definition}\label{def:controlled surj}
	A homomorphism $\varphi \colon C_0(X)\to C_0(Y)$ is called a \emph{controlled surjection} with control function $F$ if for any $f\in M_n(C_0(Y))_L$ there exists a lift $g\in M_n(C_0(X))_{F(L)\cdot \|f\|}$ such that $\varphi(g)=f$ and $\|g\|\leqslant 2\|f\|$.	
\end{definition}

Let us now introduce a notion of uniform control for inductive systems, which will be useful for our construction of the Lipschitz controlled $K$-theory for locally compact finite dimensional simplicial complexes.
	\begin{definition}\label{def:fullyfaithful1} 
	Given an inductive system  $\{M_L\}_{L \geqslant 0}$, let us denote by $i_L$ the natural map $M_L\to \varinjlim M_L$. We say $\{M_L\}_{L\geqslant 0}$ is \emph{uniformly controlled} if there exist $L_0\geqslant 0$ and a control function $F\colon \mathbb R_{\geqslant 0}\to \mathbb R_{\geqslant 0} $ such that 
	\begin{itemize}
		\item for any $L\geqslant L_0$, the map $i_L\colon M_L\to \varinjlim M_L$ is surjective;
		\item if $i_L(x)=0$ in $\ilim M_L$ for some $x\in M_L$, then $i_{L,F(L)}(x)=0$ in $M_{F(L)}$.
	\end{itemize}
	We call $(L_0,F)$ a \emph{uniform control pair} of $\{M_L\}_{L \geqslant 0}$. And we shall say $\{M_L\}_{L \geqslant 0}$ is $(L_0,F)$-uniformly controlled if we want to specify the uniform control pair $(L_0,F)$.
\end{definition}

The following lemmas show that the above notion of uniform control for inductive systems is preserved by controlled isomorphisms and satisfies a five-lemma-type property.  We refer the reader to Appendix \ref{app:QKT-lipschitzfiltration} for the detailed proofs. 
\begin{lemma}\label{lemma:fullyfaithful-iso1}
	 Suppose two inductive systems $\{M_L\}_{L \geqslant 0}$ and $\{M_L'\}_{L \geqslant 0}$ are controlled isomorphic, i.e., there are controlled homomorphisms 
	\[ \xi\colon \{M_L\}_{L \geqslant 0}\xrightarrow{}\{M_L'\}_{L \geqslant 0} \textup{ and } \eta\colon \{M_L'\}_{L \geqslant 0}\xrightarrow{}\{M_L\}_{L \geqslant 0} \] such that $\xi\circ\eta$ and $\eta\circ\xi$ are controlled equivalent to the identity homomorphism respectively. If $\{M_L\}_{L \geqslant 0}$ is uniformly controlled, then $\{M_L'\}_{L \geqslant 0}$ is also uniformly controlled. Moreover,  the uniform control pair of $\{M_L'\}_{L \geqslant 0}$ only depends on the uniform control pair of $\{M_L\}_{L \geqslant 0}$, the control functions of $\xi,\eta$, and the control functions of $\xi\circ\eta\sim\id$ and $\eta\circ\xi\sim\id$.
\end{lemma}
\begin{proof}
	See Lemma \ref{lemma:fullyfaithful-iso}. 
\end{proof}
	\begin{lemma}\label{lemma:fivelemma1}
	Suppose we have an asymptotically exact sequence 
	$$\{M_L^1\}_{L \geqslant 0} \xrightarrow{\xi^1}\{M_L^2\}_{L \geqslant 0} \xrightarrow{\xi^2}\{M_L^3\}_{L \geqslant 0}\xrightarrow{\xi^3}\{M_L^4\}_{L \geqslant 0}\xrightarrow{\xi^4}\{M_L^5\}_{L \geqslant 0} $$
	of inductive systems. If $\{M_L^i\}_{L \geqslant 0}$ is uniformly controlled for $i=1,2,4,5$, then so is $\{M_L^3\}_{L \geqslant 0}$. Moreover,  the uniform control pair of $\{M_L^3\}_{L \geqslant 0}$ only depends on   the uniform control pairs of $\{M_L^i\}_{L \geqslant 0}$, $i=1,2,4,5$ and the control functions of the asymptotic homomorphisms $\{\xi_i\}_{i=1,2,3,4}$ and the exactness at $\{M_L^i\}_{L \geqslant 0}$, $i=2,3,4$. 
\end{lemma}
\begin{proof}
	See Lemma \ref{lemma:fivelemma}. 
\end{proof}
\subsection{Lipschitz control for $K$-theory: compact case}
In this section, we will prove Theorem \ref{thm:introLipschitzC(X)} for compact simplicial complexes. 

Let us assume $X$ is a compact simplicial complex endowed with the standard simplicial metric. This means if $x$ and $y$ are in the same simplex expressed by convex combinations of its vertices $x=\sum_{j}t_jv_j$ and $y=\sum_j t_j'v_j$, then the distance from $x$ to $y$ is defined to be $\sum_j|t_j-t_j'|$;  if two points are in different simplices, their distance is defined to be the length of the shortest path between them. If there does not exist any path connecting two points, that is, the two points are in two different connected components, then we define their distance to be infinity. 

To prove Theorem \ref{thm:introLipschitzC(X)}   for compact simplicial complexes, we need the following six-term asymptotically exact sequence  of the Lipschitz controlled $K$-theory, which a special case of Theorem \ref{thm:six-term-for-pullback}. We refer the reader to the appendix for a detailed proof.  

\begin{proposition}\label{prop:pullbacksixterm1}
	Let $X$ be a compact metric space. Let $X_1,X_2$ be two compact  subspaces of $X$ such that  $X_1\cup X_2=X$. If either of the restriction maps 
	\[ \pi_i: C(X_i)\to C(X_1\cap X_2),  \quad  i = 1, 2, \]
	is a controlled surjection with control function $F$ in the sense of Definition \ref{def:controlled surj}, then we have the following six-term asymptotically exact sequence
	$$\scalebox{0.93}{\xymatrixcolsep{1pc}\xymatrix{
			\{K_0^L(C(X))\}_{L \geqslant 0}\ar[r]&\{K_0^L(C(X_1))\oplus K_0^L(C(X_2))\}_{L \geqslant 0} \ar[r]&\{K_0^L(C(X_1\cap X_2))\}_{L \geqslant 0} \ar[d]\\
			\{K_1^L(C(X_1\cap X_2))\}_{L \geqslant 0} \ar[u]&\{K_1^L(C(X_1))\oplus K_1^L(C(X_2))\}_{L \geqslant 0} \ar[l]&\{K_1^L(C(X))\}_{L \geqslant 0} \ar[l]
	}	}
	$$
	that is, the above sequence is asymptotically exact at each term with respect to certain control functions. Moreover, these control functions are solely determined by $F$.
\end{proposition}

Now we are ready to prove Theorem \ref{thm:introLipschitzC(X)} for compact simplicial complexes.
\begin{theorem}\label{prop:fullyfaithfulofC(X)}
	Let $X$ be a compact $m$-dimensional simplicial complex equipped with the standard simplicial metric. Then the inductive system $\{K_*^L(C(X))\}_{L \geqslant 0}$ is uniformly controlled in the sense of Definition $\ref{def:fullyfaithful1}$,  where the uniform control pair only depends on $m$. Consequently, there exists a constant $L_m$ \textup{(}depending only on $m$\textup{)} such that every class in $K_*(C(X))$ admits an $L_m$-Lipschitz representative. 
\end{theorem}

\begin{proof}
We prove the theorem by induction on the dimension $m$. The case where $m=0$ is trivial. 

 Suppose that the theorem holds for all $(m-1)$-dimensional compact simplicial complexes. Let $X$ be an $m$-dimensional compact simplicial complex and $V$ the set of vertices in $X$. Then every point in $X$ can be written as convex combination
$\sum_{v\in V}t_v v$ with $\sum_{v\in V}t_v=1$ and at most $(m+1)$ members of $\{t_v\}$ are nonzero. 

Let $X_1$ be a compact neighborhood of the $(m-1)$-skeleton of $X$, and $X_2$  a compact neighborhood of the set of barycenters of $m$-simplices in $X$. More precisely, we define
\begin{align*}
X_1&=\Big\{\sum_{j=0}^{m}t_{v_{j}}v_{j}\in X:
\begin{matrix}
\{v_{0},\cdots,v_{m}\}\text{ spans an } m\text{-simplex }\\
\text{and }\min\{t_{v_j}\}\leqslant \frac{1}{3(m+1)}
\end{matrix}\Big\},\\
X_2&=\Big\{\sum_{j=0}^{m}t_{v_{j}}v_{j}\in X:
\begin{matrix}
\{v_{0},\cdots,v_{m}\}\text{ spans an } m\text{-simplex }\\
\text{and }\min\{t_{v_j}\}\geqslant \frac{1}{4(m+1)}
\end{matrix}\Big\},
\end{align*}
and equip $X_1$, $X_2$ and $X_1\cap X_2$ with the induced metric from $X$ (cf. Figure \ref{fig:X1X2}). 

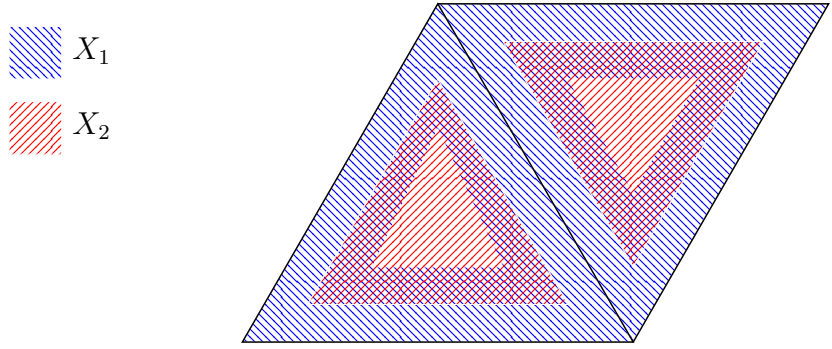
\begin{figure}
	\centering
	\begin{tikzpicture}
	\filldraw[pattern=north west lines, pattern color=blue,draw=white]
	(-{3*sin(60)},-{3*cos(60)}) -- (0,3) -- ({3*sin(60)},-{3*cos(60)}) -- cycle
	(-{3*sin(60)/3},-{3*cos(60)/3}) -- ({3*sin(60)/3},-{3*cos(60)/3}) -- (0,4/3) -- cycle;
	\filldraw[pattern=north east lines, pattern color=red,draw=white]
	(-{6*sin(60)/3},-{6*cos(60)/3}) -- (0,6/3) -- ({6*sin(60)/3},-{6*cos(60)/3}) -- cycle;
	\filldraw[pattern=north west lines, pattern color=blue,draw=white]
	({3*sin(60)},-3/2)-- (0,3) -- ({6*sin(60)},3) -- cycle
	({2*sin(60)},2) -- ({3*sin(60)},1/2) -- ({4*sin(60)},2) -- cycle;
	\filldraw[pattern=north east lines, pattern color=red,draw=white]
	({3*sin(60)},-1/2) -- ({sin(60)},5/2) -- ({5*sin(60)},5/2) -- cycle;
	\draw[line width=0.2mm]
	(-{3*sin(60)},-{3*cos(60)}) -- (0,3) -- ({3*sin(60)},-{3*cos(60)}) -- cycle;
		\draw[line width=0.2mm]
	({3*sin(60)},-{3*cos(60)})--({6*sin(60)},3) -- (0,3) ;
	\filldraw[pattern=north west lines, pattern color=blue, draw=white]
	(-5.7,2) rectangle (-5,2.7)  node[anchor=north west] {$X_1$};
	\filldraw[pattern=north east lines, pattern color=red, draw=white]
	(-5.7,1) rectangle (-5,1.7)  node[anchor=north west] {$X_2$};

	\end{tikzpicture}
	\caption{$X_1$ and $X_2$ in a simplicial complex}
	\label{fig:X1X2}
\end{figure}
 
\begin{claim*}
The restriction map $\pi_2\colon C(X_2)\to C(X_1\cap X_2)$ is a controlled surjection with control function
$F(L)=9m(m+1)L+12(m+1)$ for $L\geqslant 0$ in the sense of Definition \ref{def:controlled surj}.
\end{claim*} 
Let us prove the claim. Suppose we have $f\in M_n(C(X_1\cap X_2))_L$ with $\|f\|=1$. For each $\alpha\in(0,1]$, we define 
\begin{equation}\label{eq:S_alpha}
S_\alpha=\Big\{\sum_{j=0}^{m}t_{v_{j}}v_{j}\in X:
\begin{matrix}
\{v_{0},\cdots,v_{m}\}\text{ spans an } m\text{-simplex }\\
\text{and }\min\{t_{v_j}\}=\frac{\alpha}{(m+1)}
\end{matrix}\Big\}.
\end{equation}
Then $S_1$ is the disjoint union of the barycenters of all $m$-simplices of $X$, and $S_{\frac 1 3}\cup S_{\frac 1 4}$ is the boundary of $X_1\cap X_2$. 

Now we define a homotopy that deforms $S_{\frac 1 3}$ to $S_{\frac 1 2}$ by
\begin{equation}\label{eq:homotopy-simplicial}
\textstyle H(x,\lambda)\coloneqq\frac{1}{m+1} \sum_{j=0}^m v_j + (1-\frac{\lambda}{4})(x-\frac{1}{m+1} \sum_{j=0}^m v_j),
\end{equation}
for $\forall \lambda\in [0,1]$ and $\forall x=\sum_{j=0}^{m}t_{v_{j}}v_{j}\in S_{\frac1 3}$. 
For each $\lambda\in[0,1]$, $H(\cdot,\lambda)$ is a homeomorphism of $S_{\frac 1 3}$ and $S_{\frac{2+\lambda}{6}}$. Let $g$ be an element in $ M_n(C(X_2))$ given by
\begin{equation}\label{eq:g(y)}
    g(y)=\begin{cases}
f(y)& \textup{ if } y\in X_1\cap X_2,\\
\lambda f(x)& \textup{ if }  y=H(x,\lambda)\text{ for some }(x,\lambda)\in S_{\frac 1 3}\times[0,1],\\
0& \text{ otherwise.}
\end{cases}
\end{equation}

Obviously we have $\|g\|=\|f\|$ and the restriction of $g$ to $X_1\cap X_2$ is $f$. If  $y$ and $y'$ are in the image $H(S_\frac{1}{3}\times [0,1])$ with $y=H(x,\lambda)$ and $y'=H(x',\lambda')$, then it easily follows that 
$$d(y,y')\leqslant \frac{3}{4}d(x,x')+\frac{|\lambda-\lambda'|}{6(m+1)}.$$
On the other hand, we shall show that 
\begin{equation}\label{eq:dyy'>=lam-lam'}
d(y,y')\geqslant \frac{|\lambda-\lambda'|}{12(m+1)}
\end{equation}
and
\begin{equation}\label{eq:dyy'>=dxx'}
d(y,y')\geqslant \frac{d(x,x')}{9m(m+1)}.
\end{equation}

Let us first assume that $x, x'\in S_{\frac{1}{3}}$ lie in the same $m$-simplex and write 
\[  x=t_0v_0+t_1v_1+\cdots+t_mv_m
\text{ and } x'=t'_0v_0+t_1'v_1+\cdots+t_m'v_m.\] Let us further assume that $t_0$ (resp. $t'_0$) is the minimum of the coefficients of $x$ (resp. $x'$),  which is $\frac{1}{3(m+1)}$, i.e.,  
$$\textstyle x=\frac{1}{3(m+1)}v_0+t_1v_1+\cdots+t_mv_m
\text{ and }x'=\frac{1}{3(m+1)}v_0+t_1'v_1+\cdots+t_m'v_m.$$
Then $d(y,y')$ equals the $\ell^1$-norm of the following vector
\begin{align*}  y-y'=&{ \textstyle\frac{\lambda'-\lambda}{12(m+1)}v_0+\sum_{j=1}^m\Big(t_j(1-\frac\lambda 4)-t_j'(1-\frac{\lambda'}{4})+\frac{\lambda-\lambda'}{4(m+1)}\Big)v_j }\\
=& {\textstyle \frac{\lambda'-\lambda}{12(m+1)}v_0+\sum_{j=1}^m
\Big((t_j-t_j')(1-\frac\lambda 4)-\frac{t_j'(\lambda-\lambda')}{4}+\frac{\lambda-\lambda'}{4(m+1)}\Big)v_j }
\end{align*}
which in particular implies the inequality  \eqref{eq:dyy'>=lam-lam'}. Now observe that, if we have  $|\lambda-\lambda'|\geqslant \frac{4d(x,x')}{3m}$, then 
$$d(y,y')\geqslant \frac{|\lambda-\lambda'|}{12(m+1)}\geqslant \frac{d(x,x')}{9m(m+1)},$$
which implies the inequality  \eqref{eq:dyy'>=dxx'} in this case. 
So it remains to consider the case where $|\lambda-\lambda'|\leqslant \frac{4d(x,x')}{3m}$. Without loss of generality, we assume 
$$|t_1-t_1'|=\max\{ |t_j-t_j'|:j=1,2,\cdots m\}.$$
It follows that 
\[ \textstyle |t_1-t_1'|\geqslant \frac{1}{m}\big(\sum_{j=0}^{m} |t_j -t'_j|\big) = \frac{d(x,x')}{m}. \] 
We conclude that 
\begin{align*}
d(y,y')\geqslant&  |t_1-t_1'|(1-\frac\lambda 4)-\frac{t_1'|\lambda-\lambda'|}{4}-\frac{|\lambda-\lambda'|}{4(m+1)} \\
\geqslant& \frac{3d(x,x')}{4m}-\frac{d(x,x')}{3m}-\frac{d(x,x')}{3(m+1)}\geqslant\frac{d(x,x')}{2m}\geqslant  \frac{d(x,x')}{9m(m+1)}.
\end{align*}
which proves the inequality  \eqref{eq:dyy'>=dxx'} in this case.

Note that the distance between two general points of $X$ is the length of the shortest piecewise linear path between them. Hence applying the above special case to the end points of the line segments appearing in such a shortest piecewise linear path, the general case of the inequalities  \eqref{eq:dyy'>=lam-lam'} and \eqref{eq:dyy'>=dxx'} follows from the standard triangle inequality. To summarize, we see that the function  $g$ as defined in  line \eqref{eq:g(y)} is a $(9m(m+1)L+12(m+1))$-Lipschitz function on $X_2$. This proves the claim. 

By Proposition \ref{prop:pullbacksixterm1}, we obtain a six-term asymptotically exact sequence
\begin{equation}\label{cd:aes}
\begin{split}
\scalebox{0.85}{\xymatrixcolsep{.7pc}
	\xymatrix{
		\{K_0^L(C(X))\}_{L\geqslant 0}\ar[r]&\{K_0^L(C(X_1))\oplus K_0^L(C(X_2))\}_{L\geqslant 0}\ar[r]&\{K_0^L(C(X_1\cap X_2))\}_{L\geqslant 0}\ar[d]\\
		\{K_1^L(C(X_1\cap X_2))\}_{L\geqslant 0}\ar[u]&\{K_1^L(C(X_1))\oplus K_1^L(C(X_2))\}_{L\geqslant 0}\ar[l]&\{K_1^L(C(X))\}_{L\geqslant 0}\ar[l]
}}	
\end{split}
\end{equation}
where all control functions for the controlled homomorphisms and asymptotic exactness only depend on $m$, the dimension of $X$. In the following, we shall verify that the inductive systems $\{K_*^L(C(X_1\cap X_2))\}_{L\geqslant 0}$,  $\{K_*^L(C(X_1))\}_{L\geqslant 0}$ and   $\{K_*^L(C(X_2))\}_{L\geqslant 0}$ are all uniformly controlled in the sense of Definition $\ref{def:fullyfaithful1}$, where the uniform control pairs only depend on $m$. It then follows from Lemma   $\ref{lemma:fivelemma1}$ that $\{K_*^L(C(X))\}_{L\geqslant 0}$ is uniformly controlled, hence proves the theorem.

Let us first show that $\{K_*^L(C(X_1\cap X_2))\}_{L\geqslant 0}$ is uniformly controlled. Note that $X_1\cap X_2$ deformation retracts to $S_{\frac 1 3}$ with some Lipschitz homotopy, of which the Lipschitz constant only depends on $m$. By construction, the distance of each connected component of $S_{\frac 1 3}$ is uniformly positive and only depends on $m$, and each component is bi-Lipschitz equivalent to the $(m-1)$-dimensional simplicial sphere 
$$\mathbb S^{m-1}=\big\{(t_0,t_1,\cdots,t_m):\sum_{i=0}^{m}t_i=1 \textup{ and } \min_{i=0,1,\cdots,m} t_i=0 \big\},$$
 which is simply the boundary of a standard $m$-simplex. We equip $\mathbb S^{m-1}$
with the standard simplicial metric. 
By the inductive hypothesis, $\{K_*^L(\mathbb S^{m-1})\}_{L\geqslant 0}$ is uniformly controlled in the sense of Definition $\ref{def:fullyfaithful1}$,  since $\mathbb S^{m-1}$ is an $(m-1)$-dimensional simplicial complex, where the uniform control pair only depends on $m$. Therefore $\{K_*^L(C(X_1\cap X_2))\}_{L\geqslant 0}$ is also uniformly controlled by Lemma \ref{lemma:fullyfaithful-iso1}.

Now let us show that $\{K_*^L(C(X_2))\}_{L\geqslant 0}$  is uniformly controlled. Recall that $S_1$ from line \eqref{eq:S_alpha} is the set of  barycenters of $m$-simplices in $X$ equipped with the metric inherited from $X$.
Similar to the construction of $H$ in line \eqref{eq:homotopy-simplicial}, we can show that $S_1$ is a deformation retract of $X_2$ with a homotopy that is $(1+\frac{1}{3(m+1)})$-Lipschitz. 
As $S_1$ is a disjoint union of $\frac{2}{m+1}$-discrete points , $\{K_*^L(C(S_1)\}_{L\geqslant 0}$ is uniformly controlled in the sense of Definition $\ref{def:fullyfaithful1}$. Therefore $\{K_*^L(C(X_2))\}_{L\geqslant 0}$ is also uniformly controlled by Lemma \ref{lemma:fullyfaithful-iso1}.

Finally we show that $\{K_*^L(C(X_1))\}_{L\geqslant 0}$  is uniformly controlled. Note that $X_1$ deformation retracts to the $(m-1)$-skeleton $X^{(m-1)}$ of $X$ with a Lipschitz homotopy, of which the Lipschitz constant only depends on $m$. There are two natural metrics on $X^{(m-1)}$: one is the metric inherited from that of $X$ and the other is the standard simplicial metric on $X^{(m-1)}$ induced by the simplicial structure of $X^{(m-1)}$ itself. We denote the two metrics by $d_X$ and $d_{X^{(m-1)}}$ respectively. Although  $d_X$ and $d_{X^{(m-1)}}$ are different metrics in general, we show that they are bi-Lipschitz equivalent with Lipschitz constants that only depend on $m$. Indeed, by the definition of standard simplicial metrics (cf. the discussion at the beginning of this subsection), it suffices to prove there exist universal positive constants $c_1$ and $c_2$ such that 
\[  c_1\cdot  d_X(x,x') \leqslant d_{X^{(m-1)}}(x,x') \leqslant c_2 \cdot  d_X(x,x')   \]
for  those $x, x'\in X^{(m-1)}$ that lie in the same $m$-simplex of $X$, but in different $(m-1)$-simplices. Say both $x, x' \in X^{(m-1)}$ are in the simplex $[v_0, v_1, \cdots, v_m]$ of $X$. Without loss of generality, let us write 
$$x=t_1v_1+\sum_{i=2}^m t_iv_i\text{ and } x'=t_0'v_0+\sum_{i=2}^m t_i'v_i.$$
Obviously we have \[ d_X(x,x')\leqslant d_{X^{(m-1)}}(x,x'). \] On the other hand,  a direct computation shows that
$$d_X(x,x')=t_0'+t_1+\sum_{i=2}^m|t_i-t_i'|$$
and
$$d_{X^{(m-1)}}(x,x')=d_{X}(x,x')+\min\{0,1-\sum_{i=2}^m\max\{t_i,t_i'\}\}.$$
It follows that 
\begin{align*}
d_{X^{(m-1)}}(x,x')&\leqslant d_{X}(x,x')+\Big|\frac{2-\sum_{i=2}^mt_i-\sum_{i=2}^mt_i'-\sum_{i=2}^m|t_i-t_i'|
}{2}\Big|\\
&=d_{X}(x,x')+\Big|\frac{t_1+t_0'-\sum_{i=2}^m|t_i-t_i'|}{2}\Big|\\
&\leqslant\frac 3 2 d_{X}(x,x').
\end{align*}
By setting $c_1 = 1$ and $c_2 = \frac{3}{2}$, we have proved that $d_X$ and $d_{X^{(m-1)}}$ are bi-Lipschitz equivalent. Now by the inductive hypothesis,  $\{K_*^L(C(X^{(m-1)}))\}_{L\geqslant 0}$ is uniformly controlled in the sense of Definition $\ref{def:fullyfaithful1}$, where the uniform control pair only depends on $m$. By Lemma \ref{lemma:fullyfaithful-iso1}, $\{K_*^L(C(X^{(m-1)}))\}_{L\geqslant 0}$ is also uniformly controlled with the Lipschitz filtration on  $C(X^{(m-1)})$ given with respect to the metric $d_X$, where the new uniform control pair also depends only on $m$,  since $d_X$ and $d_{X^{(m-1)}}$ are bi-Lipschitz equivalent with Lipschitz constants that only depend on $m$.  Therefore $\{K_*^L(C(X_1))\}_{L\geqslant 0}$ is also uniformly controlled by Lemma \ref{lemma:fullyfaithful-iso1}. 

Now apply Lemma \ref{lemma:fivelemma1} to the six-term asymptotically exact sequence in $\eqref{cd:aes}$, and we conclude that $\{K_*^L(C(X))\}_{L\geqslant 0}$ is uniformly controlled, where the uniform control pair only depends on $m$. 
\end{proof}
\subsection{Lipschitz control for $K$-theory: general case}
In this subsection, we prove Theorem \ref{thm:introLipschitzC(X)} for all locally compact finite dimensional simplicial complexes.

We first prove a relative version of Theorem \ref{prop:fullyfaithfulofC(X)}.
Let $X$ be a compact metric space and $Y$ a compact subspace of $X$. Denote 
$$C(X,Y)=\{f\in C(X):f(y)=0,\ \forall y\in Y \}.$$
The algebra $C(X,Y)$ inherits  a natural Lipschitz filtration from that of $C(X)$, i.e., $C(X,Y)_L=C(X,Y)\cap C(X)_L$. 

For any pair of compact metric spaces $Y \subset X$, we have the following six-term asymptotically exact sequence of the Lipschitz controlled $K$-theory (compare to Proposition $\ref{prop:pullbacksixterm1}$). 
\begin{proposition}\label{prop:sixterm1}
	Let $X$ be a compact metric space and $Y$ a compact subspace of $X$. If the restriction map $r\colon C(X)\to C(Y)$ is a controlled surjection in the sense of Definition $\ref{def:controlled surj}$ with control function $F$, then we have a six-term asymptotically exact sequence
	$$\xymatrix{
		\{K_0^L(C(X,Y))\}_{L\geqslant 0}\ar[r]&\{K_0^L(C(X))\}_{L\geqslant 0}\ar[r]&\{K_0^L(C(Y))\}_{L\geqslant 0}\ar[d]\\
		\{K_1^L(C(Y))\}_{L\geqslant 0}\ar[u]&\{K_1^L(C(X))\}_{L\geqslant 0}\ar[l]&\{K_1^L(C(X,Y))\}_{L\geqslant 0}\ar[l]
	}	
	$$
	that is, the above sequence is asymptotically exact at each term with respect to certain control functions. Moreover, these control functions are solely determined by $F$.
\end{proposition}
\begin{proof}
	This is a  special case of Theorem \ref{thm:six-term}. We refer the reader to Appendix \ref{app:QKT-lipschitzfiltration} for a detailed proof.
\end{proof}
\begin{theorem}\label{prop:relativeK-group}
	Let $X$ be a compact $m$-dimensional simplicial complex and $Y$ a subcomplex of $X$.
	 Then $\{K_*^L(C(X,Y))\}_{L\geqslant 0}$ is uniformly controlled in the sense of Definition $\ref{def:fullyfaithful1}$, where the uniform control pair only depends on $m$.
\end{theorem}
\begin{proof}
	Let $X_1$ be the barycentric subdivision of $X$. Equip $X_1$ with its standard simplicial metric. Let $Y_1$ be the subcomplex of $X_1$ generated by $Y$ and equip $Y_1$ with the metric inherited from that of $X_1$. Since the standard simplicial metric on  $X$ and the standard simplicial metric on $X_1$ are bi-Lipschitz equivalent with Lipschitz constants depending only on $m$, it suffices to show that
	$\{K_*^L(C(X_1,Y_1))\}_{L\geqslant 0}$ is uniformly controlled. The proof below is similar to that of Theorem $\ref{prop:fullyfaithfulofC(X)}$. 
	
	\begin{claim*}
		The restriction map $r\colon C(X_1)\to C(Y_1)$ is a controlled surjection with control function $F(L)=6(m+1)L+2$ in the sense of Definition \ref{def:controlled surj}.
	\end{claim*}

	Let us prove the claim. Suppose we have $f\in M_n(C(Y_1))_L$ with $\|f\|=1$. Let $Z_1$ be the subcomplex of $X_1$ generated by the vertices that are not in $Y_1$. As $X_1$ is the barycentric subdivision of $X$, $Z_1$ is non-empty. Furthermore, we see that a simplex lies in $Y_1$ if and only if all of its vertices lie in $Y_1$. Therefore, for any point $x\in X_1$, there exist uniquely $y\in Y_1$, $z\in Z_1$ and $\lambda\in[0,1]$ such that $y=\lambda x+(1-\lambda)z$. Let $g$ be a function in $M_n(C(X_1))$ given by
	\begin{equation}\label{eq:g(x)2}
	    g(x)=\begin{cases}
	f(x),&x\in Y_1;\\
	(2\lambda-1)f(y),&x=\lambda y+(1-\lambda )z,~\lambda\in[\frac 1 2,1],~y\in Y_1,~z\in Z_1;\\
	0,&\text{otherwise.}
	\end{cases}
	\end{equation}
	
	Obviously we have $\|g\|=\|f\|$ and the restriction of $g$ to $Y_1$ is $f$. It remains to estimate the Lipschitz constant of $g$. We shall show that if $x=\lambda y+(1-\lambda )z$ and $x'=\lambda' y'+(1-\lambda')z'$ are two points in $X$ with $y,y'\in Y_1$, $z,z'\in Z_1$ and $\lambda,\lambda'\in[\frac 1 2,1]$, then
	\begin{equation}\label{eq:dxx'>=lam-lam'}
	d(x,x')\geqslant |\lambda-\lambda'|
	\end{equation}
	and
	\begin{equation}\label{eq:dxx'>=dyy'}
	d(x,x')\geqslant \frac{d(y,y')}{6(m+1)}.
	\end{equation}
	
	Let us first assume that $x$ and $x'$ are in the same simplex of $X_1$. More precisely, assume that
	$$x= \lambda\big(\sum_{i=1}^{m_1} t_i v_i\big)+(1-\lambda)\big(\sum_{j=1}^{m_2} s_j w_j\big)\text{ and }x'=\lambda'\big(\sum_{i=1}^{m_1}  t_i' v_i'\big)+(1-\lambda')\big(\sum_{j=1}^{m_2} s_j' w_j'\big),
	$$
	where $\{v_i,w_j\}_{1\leqslant i\leqslant m_1,1\leqslant j\leqslant m_2}$ generates a simplex in $X_1$. Then $d(x,x')$ equals the $\ell^1$-norm of the following vector
	$$x-x'=\sum_{i=1}^{m_1} (\lambda t_i-\lambda' t_i') v_i+\sum_{j=1}^{m_2}\Big((1-\lambda) s_j-(1-\lambda')s_j'\Big) w_j.$$
	
	As $\sum_{i=1}^{m_1}(\lambda t_i-\lambda' t_i')=\lambda-\lambda'$, we have
	$$d(x,x')\geqslant \sum_{i=1}^{m_1}|\lambda t_i-\lambda' t_i'|\geqslant |\lambda-\lambda'|,$$
    which implies line \eqref{eq:dxx'>=lam-lam'}. Now observe that, if $|\lambda-\lambda'|\geqslant\frac{d(x,x')}{3m_1}$, then
    $$d(x,x')\geqslant |\lambda-\lambda'|\geqslant \frac{d(y,y')}{3m_1}\geqslant \frac{d(y,y')}{6(m+1)}$$
    as $m_1\leqslant m+1$, which implies \eqref{eq:dxx'>=dyy'}.
    It remains to consider the case where $|\lambda-\lambda'|\leqslant\frac{d(y,y')}{3m_1}$.
	Without loss of generality, we assume that
	$$|t_1-t_1'|=\max\{|t_i-t_i'|:i=1,2,\cdots,m_1 \}$$
	It follows that 
	$$|t_1-t_1'|\geqslant \frac{1}{m_1}\sum_{i=1}^{m_1}|t_i-t_i'|=\frac{d(y,y')}{m_1}.$$
	We conclude that
	$$d(x,x')\geqslant \lambda|t_1-t_1'|-|\lambda-\lambda'||t_1'|\geqslant
	\frac{d(y,y')}{6m_1}\geqslant \frac{d(y,y')}{6(m+1)},
	$$
which proves inequality \eqref{eq:dxx'>=dyy'} in this case.

Note that the distance between two general points of $X$ is the length of the shortest piecewise linear path between them. Hence applying the above special case to the end points of the line segments appearing in such a shortest piecewise linear path, the general case of the inequalities  \eqref{eq:dxx'>=lam-lam'} and \eqref{eq:dxx'>=dyy'} follows from the standard triangle inequality.  Therefore, we see that the function $g$ in line \eqref{eq:g(x)2} is a $(6(m+1)L+2)$-Lipschitz function on $X_1$ with $r(g)=f$ and $\|g\|=\|f\|=1$. This proves the claim.
	
	By Proposition \ref{prop:sixterm1}, we obtain a six-term asymptotically exact sequence
	\begin{equation}\label{cd:ase-rel}
	\begin{split}
	     \xymatrix{
		\{K_0^L(C(X_1,Y_1))\}_{L\geqslant 0}
		\ar[r]&\{K_0^L(C(X_1))\}_{L\geqslant 0}
		\ar[r]&\{K_0^L(C(Y_1))\}_{L\geqslant 0}
		\ar[d]\\
		\{K_1^L(C(Y_1))\}_{L\geqslant 0}
		\ar[u]&\{K_1^L(C(X_1))\}_{L\geqslant 0}
		\ar[l]&\{K_1^L(C(X_1,Y_1))\}_{L\geqslant 0}\ar[l]
	}
	\end{split}
	\end{equation}
	where all control functions of the controlled homomorphisms and exactness only depend on $m$. It follows from Theorem \ref{prop:fullyfaithfulofC(X)} that $\{K_*^L(C(X_1))\}_{L\geqslant 0}$ is uniformly controlled in the sense of Definition \ref{def:fullyfaithful1}.
	Therefore, it suffices to verify that $\{K_*^L(C(Y_1))\}_{L\geqslant 0}$ is uniformly controlled.
	It then follows from Lemma $\ref{lemma:fivelemma1}$ that $\{K_*^L(C(X_1,Y_1))\}_{L\geqslant 0}$ is uniformly controlled, hence proves the theorem. 
	
	There are two natural metrics on $Y_1$: one is the metric inherited from $X_1$, and the other is the standard simplicial metric on $Y_1$ given by the simplicial structure of $Y_1$. 
	We denote the two metrics by $d_{X_1}$ and $d_{Y_1}$ respectively.
	By Theorem \ref{prop:fullyfaithfulofC(X)}, the inductive system
	$\{K_*^L(C(Y_1))\}_{L\geqslant 0}$ is uniformly controlled with the Lipschitz filtration on  $C(Y_1)$ given with respect to the metric $d_{Y_1}$, where the uniform control pair  depends only on $m$.
	\begin{claim*}
		The inductive system
$\{K_*^L(C(Y_1))\}_{L\geqslant 0}$ is also uniformly controlled with the Lipschitz filtration on  $C(Y_1)$ given with respect to the metric $d_{X_1}$, where the uniform control pair also depends only on $m$.
	\end{claim*}
Note that\footnote{Although $d_{Y_1}$ and $d_{X_1}$ are bi-Lipschitz equivalent,  there are no uniform bounds on the Lipschitz constants for  bi-Lipschitz equivalences between $d_{Y_1}$ and $d_{X_1}$,  when $Y_1$ and $X_1$ vary. } for any $y,y'\in Y_1$, 
	 $$d_{Y_1}(y,y')=d_{X_1}(y,y')$$
	 if either is no more than $1$. As both projections and unitaries have norm $\leq 1$, the inductive systems $\{K_*^L(C(Y_1))\}_{L\geqslant 0}$ with respect to the metric $d_{Y_1}$ and $d_{X_1}$ coincide as long as  $L\geqslant 2$, cf. line $\eqref{eq:lipschitz}$. Therefore $\{K_*^L(C(Y_1))\}_{L\geqslant 0}$ is also uniformly controlled with respect to the metric $d_{X_1}$, where the uniform control pair only depends on $m$. 
	 
	 Now apply Lemma \ref{lemma:fivelemma1} to the six-term asymptotically exact sequence in $\eqref{cd:ase-rel}$, and we conclude that $\{K_*^L(C(X_1,Y_1))\}_{L\geqslant 0}$ is uniformly controlled, where the uniform control pair only depends on $m$. 
\end{proof}

Now we are ready  to prove the general case of Theorem \ref{thm:introLipschitzC(X)}.

\begin{proof}[Proof of Theorem $\ref{thm:introLipschitzC(X)}$]

By assumption,  $X$ is a locally compact $m$-dimensional simplicial complex equipped with the standard simplicial metric.	Let $Y$ a compact simplicial subcomplex of $X$ and  $\partial Y$ is the boundary of $Y$ in $X$.  By Theorem \ref{prop:relativeK-group},  $\{K_*^L(C(Y, \partial Y))\}_{L\geqslant 0}$ is uniformly controlled in the sense of Definition $\ref{def:fullyfaithful1}$, where the uniform control pair only depends on $m$, for all compact simplicial subcomplexes $Y$ of $X$. A priori, the Lipschitz filtration on $C(Y, \partial Y)$ is given with respect to the standard simplicial metric on $Y$.  However, similar to the proof of Theorem \ref{prop:relativeK-group}, by working with the first barycentric subdivision of $X$ and $Y$, the metric on $Y_1$ inherited from $X_1$ and the standard simplicial metric on $Y_1$ itself, coincide for pairs of points with distance $\leqslant 1$. Now note that the $C^*$-algebra $C_0(X)$ is the inductive limit of $C(Y,\partial Y)$, where $Y$ runs through all compact simplicial subcomplexes of $X$. It follows that there exists a constant $L_m$ \textup{(}depending only on $m$\textup{)} such that every class in $K_*(C_0(X))$ admits an $L_m$-Lipschitz representative.
	
	 Furthermore, if  $\alpha \in K_*(C_0(X))$ can be represented by an element that is constant outside a compact set $K$, then $\alpha$ is in the image of the natural map $K_\ast(C(Y, \partial Y)) \to K_*(C_0(X)) $ for some compact simplicial subcomplex $Y$ of $X$ that is contained in the  $1$-neighborhood of $K$. This finishes the proof. 
	
\end{proof}
\section{Decay of scalar curvature on manifolds with finite asymptotic dimension}\label{sec:psc-fad}
In this section, we prove the main theorem of our article (Theorem \ref{thm:main}).
\subsection{Spaces with finite asymptotic dimension}\label{subsec:asymdim}
In this subsection, we recall the definition and some properties of spaces with finite asymptotic dimension.
\begin{definition}[Gromov, \cite{Gromov1993asymptotic}]\label{def:asymdim}
	The asymptotic dimension of a metric space $X$ is the smallest integer $m$ such that for any $r>0$, there exists a uniformly bounded cover $\sC_r=\{U_i\}_{i\in I}$ of $X$ for which the $r$-multiplicity of $\sC_r$ is at most $(m+1)$,  i.e., no ball of radius $r$ in $X$ intersects more than $(m+1)$ members of $\sC_r$.
\end{definition}

Let $X$ be a metric space with asymptotic dimension $m$. Denote by $d_X$ the metric on $X$. For any $r>0$, let $\sC_r = \{U_i\}_{i\in I}$ be a uniformly bounded cover of $X$ with $r$-multiplicity at most $(m+1)$ and the diameter of $U_i$ uniformly bounded by $R$.
Let $\sC'_r$ be the collection of $r$-neighborhoods $U_i':=B(U_i,r)$ of $U_i$. The cover $\sC'_r$ is also uniformly bounded with diameter $R+2r$, and has Lebesgue number at least $r$, i.e., every $r$-ball in $X$ is contained in some of the members of $\sC'_r$. Let $\sN_r$ be the \emph{nerve complex} of $\sC'_r$, that is, $\{U_{i_1}',\cdots,U_{i_k}' \}$ spans a simplex if they have non-empty intersection. Note that the dimension of $\sN_r$ is at most $m$. 

Equip $\sN_r$ with the standard simplicial metric, denoted by  $d_{\sN_r}$. We define a map  $f_r\colon X \to \sN_r$ by setting 
\begin{equation}\label{eq:f_r}
f_r(x)=\frac{\sum_{i\in I}d_X(x,X-U_i')U_i'}{\sum_{i\in I}d_X(x,X-U_i')},
\end{equation}
where it is not difficult to verify that $f_r(x)$ is indeed a point in the nerve complex $\sN_r$.
\begin{proposition}\label{prop:Lip&cobounded}
	With the same notation as above, we have the following. 
	\begin{enumerate}[$(1)$]
		\item $f_r$ is $\frac{(m+1)^2}{r}$-Lipschitz.
		\item $f_r$ is uniformly cobounded. More precisely,  if a bounded subset $K\subset \sN_r$ has diameter $d$, then $f^{-1}_r(K)$ has diameter no more than $Rd+2(R+2r)$.
	\end{enumerate}
\end{proposition}
\begin{proof}
The proof is elementary. We shall only provide the details for part (1). It suffices to show that
	$$d_{\sN_r}(f_r(x),f_r(y))\leqslant \frac{(m+1)^2}{r} d_X(x,y)$$
	when $f_r(x)$ and $f_r(y)$ lie in the same simplex, say $[U_{i_0},U_{i_1},\cdots,U_{i_{m'}}]$,  of $\sN_r$, where  $m'\leqslant m$. From the definition of standard simplicial metrics, we only need to show that the function on $X$ given by
	$$x\mapsto \frac{d_X(x,X-U_{i_k}')}{\sum_{j=0}^{m'}d_X(x,X-U_{i_j}')}$$
	is $\frac{(m+1)}{r}$-Lipschitz for each $k=0,1,\cdots,m'$. 
	Obviously, each term $d_X(x,X-U_{i_k}')$ is $1$-Lipschitz.
	By the definition of $f_r$, we see that $x$ and $y$ lie in the intersection $U_{i_0}'\cap U_{i_1}'\cap\cdots \cap U_{i_{m'}}'$. Since the cover $\{U_i'\}$ has Lebesgue number at least $r$, the denominator $\sum_{j=0}^{m'}d_X(x,X-U_{i_j}')$ is greater than $r$. Hence the proposition follows.
\end{proof}
\begin{remark}
	In fact, each of the following is equivalent to the definition of finite asymptotic dimension given in Definition $\ref{def:asymdim}$, cf.  \cite[Theorem 1]{MR2725304}.
	\begin{enumerate}[(a)]
		\item For every $r>0$, there exists a $m$-dimensional simplicial complex $\sN_r$ and a map $f_r$ as in in Proposition $\ref{prop:Lip&cobounded}$ satisfying properties (1) and (2). 
		\item For every $r>0$, there exists a uniformly bounded cover such as $\sC'_r = \{U'_i\}_{i\in I}$ above with Lebesgue number $r$ and multiplicity $\leq (m+1)$. 
	\end{enumerate}	
\end{remark}
\begin{definition}\label{def:F-asymdim}
	Let $\fD\colon\R_{\geqslant 0}\to\R_{\geqslant 0}$ be a non-decreasing function with $\fD(x)\geqslant x$. We say that a metric space $X$ has asymptotic dimension $\leqslant m$ with diameter control $\fD$ if for any $r>0$,  there exists a uniformly bounded cover $\sC_r=\{U_i\}_{i\in I}$ of $X$ such that the $r$-multiplicity of $\sC_r$ is at most $(m+1)$  and
	$$\sup_{i\in I}\diam(U_i)\leqslant \fD(r).$$
\end{definition}
\subsection{Uniformly contractible spaces with finite asymptotic dimension}
In this subsection, we show that if a metric space is uniformly contractible and has finite asymptotic dimension, then the Lipschitz map $f_r$ in Proposition \ref{prop:Lip&cobounded} admits a left homotopy inverse.

Let us first recall the definition of uniform contractibility.
\begin{definition}
	Let $\fR\colon\R_{\geqslant 0}\to\R_{\geqslant 0}$ be a non-decreasing function with $\fR(x)\geqslant x$. We say a metric space $X$ is uniformly contractible with contractibility radius $\fR$ if the metric ball $B(x, r)$ of radius $r$ centered at any point $x\in X$ becomes contractible in $B(x, \fR(r))$. More precisely,  for any $x\in X$, there exists a continuous map $H\colon B(x,r)\times [0,1]\to B(x,\fR(r))$ such that  $H(y,0)=y$ and $H(y,1)=x$ for all $y\in B(x,r)$.
\end{definition}

Assume that a metric space $X$ is uniformly contractible with contractbility radius $\fR$, and has asymptotic dimension $\leqslant m$ with diameter control $\fD$ (cf. Definition \ref{def:F-asymdim}). We shall construct a map $g_r\colon\sN_r\to X$ that satisfies the following condition:  for any $1\leqslant k\leqslant m$ and any $k$-simplex $[U_{i_0}',U_{i_1}',\cdots, U_{i_{k}}']$ of $\sN_r$, there exists a point $x\in U_{i_0}'\cap U_{i_1}'\cap \cdots\cap U_{i_{k}}'$ such that $g_r([U_{i_0}',U_{i_1}',\cdots, U_{i_{k}}'])\subset B(x,F_{k}(r))$, where
$F_{k}$ is defined  inductively with 
\begin{equation}\label{eq:F_k}
\begin{cases}
F_{1}(y)=\fR(\fD(r)+r);\\
F_{k}(y)=\fR(F_{k-1}(y)+\fD(r)+2r).
\end{cases}
\end{equation}
In particular, if the function $\fR$ is convex, i.e., 
$$\fR(x+y)\geqslant \fR(x)+\fR(y),\ \forall x,y\geqslant 0,$$
then we have $F_{k}(x)\leqslant \fR(2^{2k-1}\fD(x)).$

Let us construct $g_r$ by induction on the skeletons of $\sN_r$. For each vertex $U_i'$ of $\sN_r$, we choose a point $x_i$ in $U_i$ and define $g_r(U_i)=x_i$.

 Now if $U_i'$ and $U_j'$ span an $1$-simplex in $\sN_r$, then the intersection of $U_i'$ and $U_j'$ is non-empty in $X$. For any point $x\in U_i'\cap U_j'$,  $g_r(U_i')$ and $g_r(U_j')$ are contained in the $(\fD(r)+r)$-neighborhood of $x$. Since $X$ is uniformly contractible with contractbility radius $\fR$, we can connect $g_r(U_i')$ and $g_r(U_j')$ by a path within the $\fR(\fD(r)+r)$-neighborhood of $x$.

Assume that we have already defined $g_r$ on the $k$-skeleton of $\sN_r$ with the required properties. In particular, for each  $(k+1)$-simplex $[U_{i_0}',U_{i_1}',\cdots, U_{i_{k+1}}']$, $g_r$ has already been defined on its boundary and  there exists $x\in U_{i_0}'\cap U_{i_1}'\cap \cdots\cap U_{i_{k+1}}'$ such that
$$g_r(\partial [U_{i_0}',U_{i_1}',\cdots, U_{i_{k+1}}'])\subset B(x,F_{k}(r)+\fD(r)+2r).$$
Since $X$ is uniformly contractible with contractbility radius $\fR$, we can extend the map $g_r$ from $\partial [U_{i_0}',U_{i_1}',\cdots, U_{i_{k+1}}']$ to $[U_{i_0}',U_{i_1}',\cdots, U_{i_{k+1}}']$  within the $F_{k+1}(r)$-neighborhood of $x$.

By construction, we see that $g_r$ is also a uniformly cobounded map, i.e., for any $d>0$ and $x\in X$,
$$\diam(g_r^{-1}(B(x,d)))\leqslant \frac{1}{R}(d+2F_{m}(r)),$$
where $R=\sup_{i\in I}\diam(U_i)\leqslant \fD(r).$
\begin{lemma}\label{lemma:homotopy-id}
	Let $X$ be a Riemannian manifold with bounded geometry. Assume that under this Riemannian metric, $X$ is uniformly contractible and has finite asymptotic dimension, then $g_r f_r\colon X\to X$ is homotopic to the identity map.
\end{lemma}
\begin{proof}
As $X$ has bounded geometry, there exists $\varepsilon>0$ such that for any $x\in X$, $B(x,\varepsilon)$ is contained in a geodesically convex neighborhood of $x$. As $X$ is finite dimensional, there is a refinement $\{V_j \}_{j\in J}$ of the open cover $\{B(x,\varepsilon) \}$ such that each $V_j$ is still geodesically convex and  the multiplicity of $\{V_j \}_{j\in J}$ is at most $(\dim X+1)$. We may further assume that the open cover $\{V_j \}_{j\in J}$ has positive Lebesgue number, since $X$ has bounded geometry. 

 Consider the nerve complex $\sN$ of this open cover $\{V_j \}_{j\in J}$, which is a finite dimensional simplicial complex. We define maps
$$f\colon X\to \sN\text{ and }g\colon\sN\to X$$
similar to the construction of $f_r$ and $g_r$ above. As $\sN$ is the nerve complex of an open cover consisting of geodesically convex neighborhoods, it is clear that $gf$ is homotopic to the identity map.
 
Now we define the following map 
\begin{equation}
h_r\colon\sN\to \sN_r,\ \sum_{j\in J}t_jV_j\mapsto \sum_{j\in J}t_j\frac{\sum_{i\in I}d_X(V_j,X-U_i')U_i'}{\sum_{i\in I}d_X(V_j,X-U_i')}.
\end{equation}
When $r>\varepsilon$, $f_r(x)$ and $h_r(f(x))$ always lie in the same $m$-simplex in $\sN_r$ for every $x\in X$. Therefore there is a linear homotopy connecting $f_r$ and $h_r f$.

Furthermore, as $\sN$ is a finite dimensional complex, we also obtain a homotopy of $g\colon \sN\to X$ and $g_rh_r\colon\sN\to X$ constructed by induction on the skeletons of $\sN$. To summarize,  we have the following diagram
$$\begin{tikzcd}
X \arrow[rr, "\id"] \arrow[rd, "f"] \arrow[rdd, "f_r"'] &                                               & X \\
& \sN \arrow[d, "h_r"] \arrow[ru, "g"] &   \\
& \sN_r \arrow[ruu, "g_r"']           &  
\end{tikzcd}$$
where each of the three small triangles commutes up to homotopy. Therefore the large triangle also commutes up to homotopy, that is, $g_rf_r$ is homotopic to the identity map.
\end{proof}
\begin{remark}
	In general, $f_rg_r\colon \sN_r\to \sN_r$ is not necessarily homotopic to the identity map, and the simplicial complex $\sN_r$ may not be uniformly contractible.
\end{remark}
\subsection{Decay of scalar curvature}

In this subsection, we  prove the main theorem of this article (Theorem \ref{thm:main}) and discuss its applications.

\begin{proof}[Proof of Theorem $\ref{thm:main}$]
	If $X$ is odd dimensional, we consider $X\times\R$ instead. Thus we assume that $X$ is even dimensional. Fix $x_0\in X$. For any $\varepsilon>0$ small enough, let $B(x_0,\varepsilon)$ be the geodesic open neighborhood of $x_0$. Let $\beta$ be the Bott element that generates $K_0(C_0(B(x_0,\varepsilon)))$, then  we have
	$$\langle[D],\beta\rangle\ne 0,$$
cf. \cite[Proposition 3.33]{Roecoarse} and \cite[Lemma 9.6.9]{willett2020higher}. 
Given any $r>0$, let $f_r\colon X\to\sN_r$ be the map defined in line \eqref{eq:f_r}. By Proposition \ref{prop:Lip&cobounded}, $f_r$ is  $\frac{(m+1)^2}{r}$-Lipschitz. By Lemma \ref{lemma:homotopy-id},  $f_r$ admits a left homotopy inverse $g_r\colon \sN_r\to X$. Therefore, we have
	$$\langle[D],\beta\rangle=\langle[D],f_r^*g_r^*\beta\rangle
	=\langle (f_r)_*[D],g_r^*\beta\rangle.
	$$
	
	Since $\beta$ is constant outside $B(x_0,\varepsilon)$, $g_r^*\beta$ is constant outside $g_r^{-1}(B(x_0,\varepsilon))$. By Theorem \ref{thm:introLipschitzC(X)}, 
	there are projections $p,q$ in $M_n(C_0(\sN_r)^+)$ such that $p,q$ are $L_m$-Lipschitz, $p-q$ are supported in the $1$-neighborhood of $g_r^{-1}(B(x_0,\varepsilon))$, and $g_r^*\beta=[p]-[q]$.
	Therefore we have 
	$$\langle[D],[f_r^*p]-[f_r^*q]\rangle\ne 0.$$
	Note that
	\begin{enumerate}
		\item $f_r^*p$ and $f_r^*q$ are matrix-valued Lipschitz functions on $X$ with Lipschitz constant $\frac{(m+1)^2L_m}{r}$;
		\item $f_r^*p-f_r^*q$ is supported in $B(x_0,\varepsilon+2F_{m}(r)+2\fD(r))$, where $F_m$ is constructed by induction in line \eqref{eq:F_k}.
	\end{enumerate}
	Let $\varepsilon$ go to zero. By Proposition \ref{prop:pairing-scalarcurvature}, we obtain that 
	$$ \textstyle k_{x_0}(2F_{m}(r)+2\fD(r)+\frac{C_1(m+1)^2r}{L_m})\leqslant \frac{C_2(m+1)^2L_m}{r^2},$$
	where $k_{x_0}(r)$ be the infimum of the scalar curvature function on the ball $B(x_0,r)$.
	If we set 
	$$\textstyle G(r)=2F_{m}(\frac{\sqrt r}{(m+1)\sqrt{C_2L_m}})+
	2\fD(\frac{\sqrt r}{(m+1)\sqrt{C_2L_m}})+
	\frac{C_1(m+1)\sqrt r}{L_m\sqrt{C_2L_m}},$$
	then we have
	$k_{x_0}(G(r))\leqslant \frac{1}{r}.$
	The theorem now follows by setting
	$$F(x)\coloneqq \sup\{y:G(y)\leqslant x \},$$
	where we have adopted the convention $\sup\emptyset=0$.
\end{proof}

For a metric space, a subspace is called a net (or more precisely, $C$-net), if there exists $C>0$ such that its $C$-neighborhood covers the total space. As the function $F$ in Theorem \ref{thm:main} is independent of the base point $x_0$, we obtain the following corollary.
\begin{corollary}\label{coro:net2}
	Let $X$ be a complete Riemannian manifold with bounded geometry. If $X$ is uniformly contractible and has finite asymptotic dimension, then for any $\varepsilon>0$, the points of which the scalar curvature is smaller than $\varepsilon$ form a net in $X$.
\end{corollary}

 We remark that the conclusion is not true if we drop the uniform contractibility condition on $X$. See Section \ref{subsec:counterexample} for more details.

From the above proof of Theorem $\ref{thm:main}$, we have seen the function $F$ only depends on the contractibility radius function $\mathfrak R$, the diameter control function $\mathfrak D$ and the dimension of $X$. If both $\mathfrak R$ and $\mathfrak D$ are given explicitly, then $F$ can also be explicitly computed. As an example, we have the following corollary. 
\begin{corollary}
Let $X$ be a complete Riemannian manifold with bounded geometry. Let $\fR,\fD$ be non-decreasing functions on $\R_{\geqslant 0}$. Assume that $X$ is uniformly contractible with contractibility radius $\mathfrak R$, and has asymptotic dimension $\leqslant m$ with diameter control $\mathfrak D$. 
	\begin{itemize}
		\item If both $\fD$ and $\fR$ are proportional functions, i.e., $\fD(x)=C_d x$ and $\fR(x)=C_rx$, then there exists $C>0$ that only depends on $(m,C_d,C_r)$ such that
		$$\inf_{x\in B(x_0,r)}k(x)\leqslant \frac{C}{r^2}$$
		 for any $x_0\in X$ and $r>0$.
		\item If both $\fD$ and $\fR$ are linear functions, i.e., $\fR(x)=A_rx+B_r$ and $\fD(x)=A_dx+B_d$, then there exists $(A,B)$ that only depends on $(m,A_d,B_d,A_r,B_r)$ such that 
		$$\inf_{x\in B(x_0,r)}k(x)\leqslant \frac{1}{(\max\{Ar-B,0\})^2}$$
		for any $x_0\in X$ and $r>0$.
	\end{itemize}
\end{corollary}

More specifically,  let us consider metrics on $\R^n$ (reps. $\bH^n$) with $n\geqslant 2$  that are quasi-isometric to  the standard Euclidean metric (resp. hyperbolic metric) 
Recall that two metrics $d_1$ and $d_2$ on $X$ are quasi-isometric if there are positive constants $L$ and $C$ such that
\begin{equation}\label{eq:quasi-isometry}
L^{-1}d_1(x,y)-C\leqslant d_2(x,y)\leqslant Ld_1(x,y)+C,\ \forall x,y\in X.
\end{equation}
The growth rates of $\fD$ and $\fR$ essentially remain unchanged for  quasi-isometric metrics. In particular, if a metric is quasi-isometric to the standard Euclidean metric (resp. hyperbolic metric) on $\R^n$  (resp. $\bH^n$), then both $\fD$ and $\fR$ (for this metric) are linear functions.  As a consequence, we have the following corollary. 
\begin{corollary}
	If a complete Riemannian metric $d$ on $\R^n$ \textup{(resp. $\bH^n$)} with $n\geqslant 2$ is quasi-isometric to the standard metric on  $\R^n$ \textup{(resp. $\bH^n$)}, then the scalar curvature of $d$ has quadratic decay, i.e., there are positive constants $A$ and $B$ such that 
$$\inf_{x\in B(x_0,r)}k(x)\leqslant \frac{1}{(\max\{Ar-B,0\})^2}.$$
	for any $x_0\in \R^n$ \textup{(resp. $x_0\in \bH^n$)} and $r>0$.
\end{corollary}

We remark that a result of Gromov shows  that hyperbolic groups, or more generally hyperbolic spaces with bounded growth, have finite asymptotic dimension \cite{Gromov1993asymptotic}. See also Roe's paper \cite{Roehyperbolic} for an alternative proof.   Moreover, Roe's proof shows that the diameter control function can be chosen to be $\fD(x)=4x+4\delta$ for $\delta$-hyperbolic spaces.

Note that  CAT(0) spaces are uniformly contractible with contractibility radius $\fR(x)=x$ (cf. the proof of Theorem \ref{thm:cat0}). It is an open question whether every finite dimensional CAT(0) manifold  has finite asymptotic dimension. On the other hand, if $X$ is a finite dimensional CAT(0) cube complex, then $X$ has finite asymptotic dimension (cf. \cite{MR3748250,MR2916293}). Moreover, the diameter control function can be chosen to be $\fD(x)=C\cdot x^p$ in this case, where $C$ and $p$ only depend on the dimension of $X$, cf.\cite{MR2916293}.

\subsection{Decay rate of scalar curvature and contractibility radius function}\label{subsec:counterexample}
In this subsection, we construct explicit examples of uniformly contractible manifolds with finite asymptotic dimension such that their scalar curvature functions decay very slowly. The construction of these examples also gives information on how the decay rate of scalar curvature depends explicitly on the contractibility radius function $\mathfrak R$. On the other hand, although the decay of scalar curvature can be very slow in these examples, it follows from Theorem \ref{thm:main} that the decay is uniform with respect to choices of based points in the sense that there exists a control function $F\colon \R_{\geqslant 0} \to \R_{\geqslant 0}$ such that   	$$\inf_{x\in B(x_0,r)}k(x)\leqslant \frac{1}{F(r)}$$
for all $x_0\in X$. A direct consequence of this uniformity phenomenon is Corollary $\ref{coro:net}$.  Here we point out that the uniform contractibility condition is essential for this type of uniformity to hold.  More precisely,  we construct  examples of noncompact complete manifolds that are \emph{not} uniformly contractible such that the decay function for scalar curvature depends on the base point of the manifold. In particular,  these examples show that Corollary $\ref{coro:net}$ fails in general,  without the uniform contractibility condition.

We have seen in Theorem $\ref{thm:main}$ that the decay rate of scalar curvature on a uniformly contractible manifold with finite asymptotic dimension depends explicitly on the contractibility radius function $\mathfrak R$ and the diameter control function $\mathfrak D$. For many examples such as CAT(0)-like spaces (cf. Theorem  $\ref{thm:cat0}$) and those studied by Gromov \cite{Gromovinequalities2018} and Zeidler \cite{Rudolf2019}, the decay rate is known to be quadratic. In general, the decay rate can be much slower than quadratic. More precisely,  we shall construct explicit examples to show that  for any prescribed non-decreasing function $F\colon \R_{\geqslant 0}\to \R_{\geqslant 0}$ such that $F(r) \to \infty$ as $r\to \infty$,  there exists uniformly contractible manifold $X$ with finite asymptotic dimension such that $k_{x_0}(r) \coloneqq \inf_{x\in B(x_0,r)}k(x)$ decays slower than $\frac{1}{F(r)}$ as $r\to \infty$, or equivalently
$$\lim_{r\to+\infty} F(r)\cdot k_{x_0}(r) =+\infty,$$
for every $x_0\in X$.  

Recall that a function $F\colon\R_{\geqslant 0}\to \R_{\geqslant0}$ is called a control function if $F$ is non-decreasing and $F(r)\to \infty$ as $r\to\infty$. The following proposition shows that in general the decay of scalar curvature on a uniformly contractible manifold with finite asymptotic dimension can be very slow.

\begin{proposition}\label{prop:example1}
	Given $m\geqslant 3$ and any control function $F$, there exists a complete Riemannian metric with bounded geometry on $\R^m$ satisfying the following: 
	\begin{enumerate}[$(a)$]
		\item $\R^m$ with this metric is uniformly contractible;
		\item $\R^m$ with this metric has asymptotic dimension $\leq m$ whose diameter control function is linear and independent of $F$;
		\item the scalar curvature function $k$ of this metric is positive everywhere and satisfies
		$$\lim_{r\to+\infty}\big( F(r)\cdot\inf_{x\in B(x_0,r)}k(x)\big)=+\infty,$$
		for any $x_0\in\R^m$.
	\end{enumerate}
\end{proposition}
\begin{proof}
	We prove the case where $m=3$. The general case can for example be proved by taking the direct product with the standard Euclidean spaces.
	
	 Given a smooth positive function $\varphi$, we consider the warped metric on $\R^3$ given by
	 \begin{equation}\label{eq:warped.metric}
	 g^\varphi=dt^2+\varphi^2(t)g^{\mathbb S^2}
	 \end{equation}
	where $g^{\mathbb S^2}$ is the standard metric on the unit sphere $\mathbb S^2$. In other word, the metric is inherited from the hypersurface in $\R^4$ given by
	$$(t,\theta,\rho)\mapsto (\psi(t),\varphi(t)\cos\theta,\varphi(t)\sin\theta\cos\rho,\varphi(t)\sin\theta\sin\rho)$$
	where $\varphi'^2+\psi'^2=1$.
	
	A direct computation shows that the scalar curvature of the warped metric in line \eqref{eq:warped.metric} is given by (cf. \cite[Proposition 7.33]{GromovLawson}, \cite[\S 2.4]{Gromov4lectures2019})
	\begin{equation}\label{eq:sc.example}
	\frac{2-2(\varphi')^2-4\varphi\varphi''}{\varphi^2}.
	\end{equation}
	Therefore, if $2-2(\varphi'(t))^2-4\varphi(t)\varphi''(t)>0$ for all $t$, then the scalar curvature is positive everywhere on $\mathbb R^3$.
	
	Given any control function $F\colon \R_{\geqslant 0}\to \R_{\geqslant 0}$, we claim that there exists a smooth function $$\varphi\colon \R_{\geqslant 0}\to \R_{\geqslant 0}$$
	with the following properties:
	\begin{enumerate}
		\item $\lim_{t\to+\infty}\varphi(t)=+\infty$;
		\item $\varphi(t)=\sin t$ for $0\leqslant t\leqslant \frac{\pi}{3}$;
		\item $\varphi'(t)>0$ and $\varphi''(t)<0$  for all $t>0$;
		\item  $\varphi'(t)\leqslant \frac 2 3$ for all  $t\geqslant \frac \pi 3$; 
		\item and 
		\[ \lim_{t\to+\infty}\frac{F(t)}{\varphi^2(t)}=+\infty.\]
	\end{enumerate}

	Let us sketch the construction of $\varphi$. We set $\varphi(t)=\sin t$ for $0\leqslant t\leqslant \frac{\pi}{3}$ as in (2). Note that for $0\leqslant t\leqslant \frac\pi 3$, we have
	$$\varphi'(t)=\cos t \textup{ and } \varphi''(t)=-\sin t.$$
	In particular, (3) and (4) are satisfied if $t$ is sufficiently close to $\frac\pi 3$. Now  it is not difficult to see that we can extend the function $\varphi''$ on  $[0, \frac{\pi}{3}]$ to a function on $[0,+\infty)$ so that $\varphi$ satisfies (1), (3), (4) and (5).
	
	Consider the metric $g^{\varphi}$ on $\R^3$ given by line \eqref{eq:warped.metric} with the function $\varphi$ above. We show that $(\R^3,g^\varphi)$ satisfies the required conditions $(a), (b)$ and $(c)$.
	
	 For $0\leqslant t\leqslant\frac \pi 3$, the metric is exactly the standard metric on the unit $3$-sphere, hence has positive scalar curvature. When $t\geqslant \frac \pi 3$, it follows from line \eqref{eq:sc.example} that
	$$\textstyle
	2-2(\varphi')^2-4\varphi\varphi''\geqslant 2-2\cdot (\frac 2 3)^2\geqslant \frac{10}{9}>0.$$
	Therefore, the scalar curvature is positive everywhere on $(\R^3,g^\varphi)$.

	Obviously $(\R^3,g^\varphi)$ is complete and has bounded geometry. Let us first prove part $(c)$. Without loss of generality, let us $x_0$ is the origin of $\mathbb R^3$, that is, $x_0$ is the point with $t=0$. A direct computation shows that for fixed $\theta_0$ and $\rho_0$,  the curve
	$$(t,\theta_0,\rho_0)\mapsto (\psi(t),\varphi(t)\cos\theta_0,\varphi(t)\sin\theta_0\cos\rho_0,\varphi(t)\sin\theta_0\sin\rho_0)$$ parameterized by $t$ is a geodesic. Therefore the geodesic ball centered at the origin $x_0$ with radius $r$ contains exactly the points with $0\leqslant t\leqslant r$. When $r$ is sufficiently large, we have
	$$
	\inf_{x\in B(x_0,r)}k(x)
	=\inf_{0\leqslant t\leqslant r}\frac{2-2(\varphi')^2(t)-4\varphi(t)\varphi''(t)}{\varphi^2(t)}
	\geqslant \inf_{0\leqslant t\leqslant r}\frac{10}{9\varphi^2(r)}
	$$
	It follows from property (5) of  $\varphi$ that
	$$F(r)\cdot\inf_{x\in B(x_0,r)}k(x)\geqslant\frac{10F(r)}{9\varphi^2(r)}\to +\infty$$
	as $r$ goes to infinity.
	
	Now we show that  $(\R^3,g^\varphi)$ is uniformly contractible with contractibility radius $\fR(r)=\varphi^{-1}(r)$, where $\varphi^{-1}\colon \R_{\geqslant 0} \to \R_{\geqslant 0}$ is the inverse function of $\varphi$. Given $r>0$, let $B(x,r)$ be the closed ball in $(\R^3,g^\varphi)$ with radius $r$ centered at a point $x\in\R^3$. If $\varphi(d(x,x_0))\geqslant r$, then $B(x,r)$ does not contain any pair of antipodal points of the spheres  $\mathbb S^2_t \coloneqq  \{ t=\text{constant}\}$. Hence $B(x,r)$ itself is contractible. It remains to consider the case where $\varphi(d(x,x_0))\leqslant r$. In this case, $B(x,r)$ is homotopic to the origin $x_0$ via the geodesics parameterized by $t$ from above. In other words, $B(x,r)$ is contractible within the closed ball $B(x,d(x,x_0))$. In conclusion, $B(x,r)$ is contractible within the closed ball centered at $x_0$ with radius
	$$\max\{r,\varphi^{-1}(r) \}=\varphi^{-1}(r)$$
	as $0<\varphi'\leqslant 1$.
	
	Let us now show that $(\R^3,g^\varphi)$ has finite asymptotic dimension $\leqslant 3$ with a linear diameter control $\fD(r)=100(r+10)$. In particular, $\fD$ is independent of the choice of $\varphi$. 
	
	For each $r\geqslant 10$, we shall construct a cover of $(\R^3,g^\varphi)$ with $r$-multiplicity at most $4$ and diameter $\leqslant 100r$. Set $s_0=0$ and $s_n=\varphi^{-1}(10\cdot 2^n r)$.
	For each $n\in\N$, we choose a partition of $[s_n,s_{n+1}]$ 
	$$s_n=t_{n,0}\leqslant t_{n,1}\leqslant \cdots\leqslant t_{n,m_n}=s_{n+1}$$
	 such that $r\leqslant |t_{n,i+1}-t_{n,i}|\leqslant 2r$. When $n=0$, the subsets $[t_{0,i},t_{0,i+1}]\times \mathbb S^2$ form a cover of $[s_0,s_1]\times \mathbb S^2$ with $r$-multiplicity at most $2$, each member of which has diameter no more than $100r$.
	
	Now we turn to $n\geqslant 1$. Since $\mathbb S^2$ has Lebesgue covering dimension $2$, there is a cover $\{ V_{n,j}\}_{j\in J_n}$ of $[s_n,s_{n+1}]\times \mathbb S^2$ with $r$-multiplicity $4$ and diameter $\leqslant 100r$, each of the members is of the form $[t_{n,i},t_{n,i+1}]\times V$ with $V\subset \mathbb S^2$.
	Furthermore, when $n\geqslant 2$, we can choose carefully the subsets 
	$$\{V_{n,j},~j\in J_n: V_{n,j}=[t_{n,0},t_{n,1}]\times V \text{ for some }V\subset \mathbb S^2 \}$$
	  so that the union $\{V_{n-1,j} \}_{j\in J_{n-1}}\cup \{V_{n,j} \}_{j\in J_{n}}$ which covers $[s_{n-1},s_{n+1}]\times \mathbb S^2$ has $r$-multiplicity at most $4$. Therefore, the following collection
	  $$\{[t_{0,i},t_{0,i+1}]\times \mathbb S^2 \}_{0\leqslant i\leqslant m_0}\cup 
	  \{V_{1,j} \}_{j\in J_1}\cup \{V_{2,j}\}_{j\in J_2}\cup\cdots
	  $$
	  covers $(\R^3,g^\varphi)$ with $r$-multiplicity $4$ and with diameters $\leqslant 100r$.
	  
	  For $m> 3$, we can consider the metric $g^\varphi+g^{\R^{m-3}}$ on $\R^m$, where $g^{\R^{m-3}}$ is the standard Euclidean metric on $\R^{m-3}$.
\end{proof}

We know that given a complete Riemannian spin manifold $(X, g)$, if the higher index of its Dirac operator (cf. Definition \ref{def:lind} and \ref{def:ind-odd}) is non-zero in $K_*(C^*(X,g))$ (cf. Definition \ref{def roe and localization}), then any Riemannian metric on $X$ that is  quasi-isometric (cf. line \eqref{eq:quasi-isometry}) or more generally coarsely equivalent to $g$ (cf. \cite[\S 1.4]{MR2986138}) cannot have uniformly positive scalar curvature. In other words, for any $x_0\in X$, there exists some control function $F_{x_0}\colon \mathbb R_{\geqslant 0 } \to \R_{\geqslant 0}$ such that 
$$\inf_{x\in B(x_0,r)}k(x)\leqslant \frac{1}{F_{x_0}(r)}$$
in this case. However, how fast the scalar curvature decays may depend on the base point $x_0$. This is in contrast with Theorem $\ref{thm:main}$, which states that the decay is uniform with respect to choices of based points in the sense that there exists a control function $F\colon \R_{\geqslant 0} \to \R_{\geqslant 0}$ such that   	$$\inf_{x\in B(x_0,r)}k(x)\leqslant \frac{1}{F(r)}$$
for all $x_0\in X$. A direct consequence of this uniformity phenomenon is Corollary $\ref{coro:net}$. In the following proposition, we show that if we drop the uniform contractibility condition in Theorem $\ref{thm:main}$ but only assume the nonvanishing of the higher index of the Dirac operator instead, then 
Corollary \ref{coro:net} fails in general.

\begin{proposition}\label{prop:example2}
    There exists a complete Riemannian metric $g$ with bounded geometry on $\R^m$, $m\geqslant 3$ such that the metric space $(\R^m,g)$ has asymptotic dimension no more than $m$ with a linear diameter control and the associated Dirac operator has non-zero higher index \textup{(}cf. Definition $\ref{def:lind},\ref{def:ind-odd}$\textup{)}. However, the conclusion of Corollary $\ref{coro:net}$ fails, i.e. there exists some $\varepsilon_m >0$ such that the set of points in $(\R^m,g)$  with scalar curvature $\leqslant \varepsilon_m$ is not a net in $(\R^m,g)$. 
\end{proposition}

\begin{proof}
	We prove the case where $m=3$. The general case can for example be proved by taking the direct product with the standard Euclidean spaces.
		
	Similar to the proof of Proposition \ref{prop:example1}, we consider a warped metric
	 \begin{equation*}
		g^\varphi=dt^2+\varphi(t)^2g^{\mathbb S^2}
	\end{equation*}
 on $\R^3$. We claim that there exists a smooth function
    $$\varphi\colon \R_{\geqslant 0}\to \R_{\geqslant 0}$$
    as well as an increasing sequence of positive numbers $\{a_n\}_{n\in\N^+}$ with the following properties:
    \begin{enumerate}
        \item $\varphi(t)=\sin t$ for $0\leqslant t\leqslant \frac \pi 3$,
        \item  $|\varphi'(t)|\leqslant \frac 2 3$ for all $t\geqslant \frac \pi 3$;
        \item for each $n\in\N^+$, we have
        $$ \textstyle \varphi(t)=2-\cos\big(\frac{t-a_n}{10n}\big),~\forall t\in[a_n-15n\pi,a_n+15n\pi];$$
        \item $\varphi''(t)<0$ for all         $t\in \R_{>0}-\bigcup_{n\in\N^+}[a_n-15n\pi,a_n+15n\pi];$
        \item and  $$\lim_{t\to+\infty}\varphi(\frac{a_n+a_{n+1}}{2})=+\infty.$$
    \end{enumerate}
    
    Let us sketch the construction of the function $\varphi$ above. We first choose an increasing sequence of positive numbers $\{a_n\}_{n\in\N^+}$ such that $|a_{n}-a_{n-1}|\gg 30n\pi$. Define $\varphi$ near $t =0$ and $t = a_n$ as in conditions (1) and (3) respectively. It follows that on the intervals $[a_n-15n\pi,a_n+15n\pi]$, we have
    \[ \textstyle |\varphi'(t)|=
    \frac{1}{10n}|\sin\big(\frac{t-a_n}{10n}\big)|\leqslant \frac{1}{10n}<\frac 2 3, \]
    and 
    \[  \textstyle \varphi''(t)=\frac{1}{100n^2}\cos\big(\frac{t-a_n}{10n}\big). \]
    In particular, conditions (2) and (4) are  satisfied when $t$ is near $\frac\pi 3$ or $a_n\pm 15n\pi$. Now it is not difficult to see one can extend  $\varphi$ on $[\frac\pi 3,a_1-15\pi]$ and $[a_n+15n\pi,a_{n+1}-15(n+1)\pi]$ to a smooth function on $[0, \infty)$  such that conditions (2), (4), and (5) are satisfied.
    
	Consider the metric $g^\varphi$ on $\R^3$ given by 
	$$dt^2+\varphi(t)^2g^{\mathbb S^2}$$
	with the function $\varphi$ above. We will show that $(\R^3,g^\varphi)$ satisfies the required properties. 
    
    On the intervals $[a_n+15n\pi,a_{n+1}-15n\pi]$, we have 
    $$\textstyle |\varphi(t)\varphi''(t)|=\frac{1}{100n^2}|\big(2-\cos(\frac{t-a_n}{10n})\big)\cos(\frac{t-a_n}{10n})|\leqslant \frac{1}{100}.$$
    Therefore if $t$ lies in $[a_n+15n\pi,a_{n+1}-15n\pi]$, then the scalar curvature of $g$ satisfies 
    \begin{align*}
 \frac{2-2(\varphi')^2-4\varphi\varphi''}{\varphi^2}\geqslant 
    \frac{2-2\cdot (\frac 2 3)^2-4\cdot \frac{1}{100}}{3^2}\geqslant \frac 1 9.
    \end{align*}
     On the other hand, when $t$ is not in $\bigcup_{n\in\N^+}[a_n-15n\pi,a_n+15n\pi]$, we see from $\varphi''<0$ and $|\varphi'|\leqslant \frac 2 3$ that the scalar curvature is positive. Therefore, the scalar curvature function is positive everywhere on $(\R^3,g^\varphi)$.
     Furthermore, for any point $x$ in the sphere at $t=a_n$, the scalar curvature on the $15n\pi$-neighborhood of $x$ is uniformly bounded below by $\frac 1 9$. This shows that the conclusion of Corollary \ref{coro:net} fails for $(\R^3,g^\varphi)$. More precisely, the set of points in $(\R^3,g^\varphi)$  with scalar curvature $\leq \frac 1 9$  is not a net in $(\R^3,g^\varphi)$.  
    
    Similar to the proof of Proposition $\ref{prop:example1}$, one can also verify that the metric space $(\R^3,g^\varphi)$ has asymptotic dimension $\leqslant 3$ with a linear diameter control $\fD(r)=100(r+10)$. We omit the details.   We point out that although $(\R^3, g^\varphi)$ is contractible, it is not uniformly contractible, since the diameters of the spheres at
    $t=a_n$ are uniformly bounded for $n\in\N^+$.
    
    Now we show that the higher index of the Dirac operator $D$ on $(\R^3,g^\varphi)$ is non-zero.
Let $\sN_r$ be the nerve complex of $(\R^3,g^\varphi)$ constructed in Section \ref{subsec:asymdim}, as $(\R^3,g^\varphi)$ has finite asymptotic dimension.
Consider the coarse Baum--Connes assembly map for $(\R^3,g^\varphi)$:
$$\mu\colon \lim_{r\to\infty}K_\bullet(\sN_r)\to  \lim_{r\to\infty}K_\bullet (C^\ast(\sN_r)) \cong  K_\bullet(C^*(\R^3,g^\varphi)),$$
which is an isomorphism because $(\R^3,g^\varphi)$ has finite asymptotic dimension \cite{Yu}. Let us denote by $[D]$ the $K$-homology class of $D$ in $(\R^3,g^\varphi)$. Let $f_r\colon \R^3\to \sN_r$ be the map defined in line $\eqref{eq:f_r}$. Then  the family  $\{(f_{r})_*[D]\}_{r>0}$ determines a $K$-homology class  in $\lim_{r\to\infty}K_1(\sN_r)$, which we will denote by $[D]_{\sN}$. The assembly map $\mu$ maps $[D]_{\sN}$ to  the higher index $\ind(D)$ of $D$, an element in $K_1(C^*(\R^3,g^\varphi))$.

By the condition (5) of $\varphi$, the radii of the spheres at $t=\frac{a_n+a_{n+1}}{2}$ go to infinity as $n$ goes to infinity.
Hence for each $r>0$, there exists $n$ such that the sphere at $t=\frac{a_n+a_{n+1}}{2}$ is homotopy equivalent to its image in the nerve complex $\sN_r$. Recall that the $K$-homology group $K_0(\mathbb S^2)\cong \mathbb Z$ and  is generated by the Dirac operator on $\mathbb S^2$. Therefore,  by applying a standard Mayer-Vietoris sequence argument on the $K$-homology of $\sN_r$, we conclude that  $(f_{r})_{*}[D]$ does not vanish in  $K_1(\sN_r)$ for all $r>0$. Consequently, the injectivity of $\mu$ implies that the higher index $\ind(D)$ of $D$ is nonzero in $K_1(C^*(\R^3,g^\varphi))$. 

	  For $m> 3$, we can consider the metric $g^\varphi+g^{\R^{m-3}}$ on $\R^m$, where $g^{\R^{m-3}}$ is the standard Euclidean metric on $\R^{m-3}$. This finishes the proof. 
    \end{proof}

\appendix
	\section{Controlled $K$-theory for Lipschitz filtration}\label{app:QKT-lipschitzfiltration}
	In the main body of the article, one of the key ingredients for the proof of our main theorem is a Lipschitz controlled $K$-theory for locally compact metric spaces.  In this appendix, we shall provide detailed proofs for various properties of the Lipschitz controlled $K$-theory used in Section $\ref{sec:Lipschitzcontrol}$. In fact, since the exact same proofs also work for general $C^\ast$-algebras, we shall develop a Lipschitz controlled $K$-theory for general $C^\ast$-algebras, which we believe is of independent interest.

	\subsection{Lipschitz filtered $K$-groups}
Recall that a $C^\ast$-algebra $A$ is a complete normed $*$-algebra such that  $\|ab\|\leqslant \|a\|\|b\|$ and $\|aa^*\|=\|a\|^2$. For example, for every locally compact metric space $X$, the algebra $C_0(X)$ of continuous functions on $X$ vanishing at the infinity is a $C^\ast$-algebra with its norm given by the sup-norm and the $\ast$-operator given by the complex conjugation. 

\begin{definition}\label{def:Lipfiltration}
	Let $A$ be a unital $C^*$-algebra. Assume that $\{A_L\}_{L\geqslant 0}$ is an increasing sequence of subsets in $A$ and $\bigcup_{L\geqslant 0} A_L$ is dense in $A$. We say $\{A_L\}_{L\geqslant 0}$ is a Lipschitz filtration of $A$ if the following are satisfied:
	\begin{enumerate}[(L1)]
		\item $A_L$ is $*$-closed and norm-closed, and contains $\lambda\cdot 1_A$ for any $\lambda\in\mathbb C$;
		\item if $a\in A_{L}$ and $\lambda\in\C$, then $\lambda a\in A_{|\lambda|L}$;
		\item if $a_1\in A_{L_1}$, $a_2\in A_{L_2}$, then $a_1+a_2\in A_{L_1+L_2}$;
		\item if $a_1\in A_{L_1}$, $a_2\in A_{L_2}$, then $a_1a_2\in A_{\|a_1\|L_2+\|a_2\|L_1}$;
		\item for any $n\geqslant 1$, the $C^*$-algebra $M_n(A)$ filtered by 
		$$M_n(A)_L :=\{x\in M_n(A):\xi^Tx\eta\in A_L,\forall\xi,\eta\in\C^n,\|\xi\|,\|\eta\|\leqslant 1 \}$$ 
		also satisfies (L1)-(L4) above. In particular, $\lambda\cdot 1_A\in M_n(A)_L$ for all $L\geqslant 0$ and $\lambda\in M_n(\mathbb C)$.
	\end{enumerate}
\end{definition}

If $A$ is non-unital $C^\ast$-algebra, we say an increasing sequence of subsets $\{A_L\}_{L\geqslant 0}$ in $A$  is a Lipschitz filtration on $A$ if $\{A_L^+\}_{L\geqslant 0}$ is a Lipschitz filtration of the unitalization $A^+$ of $A$, where $A_L^+\coloneqq A_L+\C$.
\begin{example}\label{example:Lipfil}
	Let $X$ be a locally compact metric space and  $C_0(X)$ the algebra of continuous functions on $X$ vanishing at the infinity. Let $C_0(X)_L$ be the set of $L$-Lipschitz functions in $C_0(X)$. It is easy to verify that  $\{C_0(X)_L\}_{L\geqslant0}$  satisfies (L1)-(L5), hence a Lipschitz filtration on $C_0(X)$. This is the reason for our choice of the terminology ``Lipschitz filtration" in the general case. 
\end{example}
\begin{lemma}\label{lemma:powerseries}
	Let $\{A_L\}_{L\geqslant 0}$ be a Lipschitz filtration on $A$. Let
	$$f(x)=c_1x+c_2x^2+c_3x^3+\cdots$$
	be a power series. If $a\in M_n(A)_L$ and $\|a\|$ is strictly smaller than the convergence radius of $f$, then
	$f(a)\in M_n(A)_{\widetilde f(\|a\|)L}$, where
	$$
	\widetilde f(x)=|c_1|+2|c_2|x+3|c_3|x^2+\cdots.
	$$
	In particular, if $c_i\geqslant 0$ for all $i\geqslant 1$, then $\tilde f=f'$.
\end{lemma}
\begin{proof}
	By induction, we have $a^k\in M_n(A)_{k\|a\|^{k-1}L}$ for any $k\in \N$. Therefore
	if we set $f_N(x)=c_1x+c_2x^2+c_3x^3+\cdots+c_Nx^N$, then
	$$f_N(a)\in M_n(A)_{\widetilde{f}_N(\|a\|)L}\subset M_n(A)_{\widetilde f(\|a\|)L}.$$
	Now the lemma follows from the fact that $M_n(A)_{\widetilde f(\|a\|)L}$ is norm-closed.
\end{proof}

\begin{remark}\label{rk:extra}
	Lemma \ref{lemma:powerseries} suggests that a Lipschitz filtration is ``closed under functional calculus" in some sense. In the rest of the appendix, we will assume that a Lipschitz filtration should always satisfy the following extra condition:
	\begin{enumerate}[(L6)]
		\item If $a\in M_n(A)_L$ and $\lambda$ is not in the spectrum of $a$, then 
		$$(a-\lambda)^{-1}\in M_n(A)_{\|(a-\lambda)^{-1}\|^2L}.$$
	\end{enumerate}
	To be precise, (L6) is only assumed here for convenience. We can develop the Lipschitz controlled $K$-theory in this appendix without (L6) by using almost projections and almost unitaries, cf.  \cite{Oyono-OyonoYu} and \cite{Yu}.  On the other hand, the main examples (Example \ref{example:Lipfil}) of this article  do satisfy the condition (L6). And assuming (L6) will help streamline our proofs. \end{remark}

Assume that $A$ is a unital $C^*$-algebra equipped with a Lipschitz filtration.
	\begin{itemize}
	\item For any $L\geqslant 0$, let $P^L_n(A)$ be the set of projections in $M_n(A)_L$, with the following natural inclusion given by
	$$P_n^L(A)\to P_{n+1}^L(A),\ p\mapsto \begin{pmatrix}
	p &0\\0&0
	\end{pmatrix}.$$
	\item For any $L\geqslant 0$, let $U_n^L(A)$ be the set of unitaries in $M_n(A)_L$, with the following natural inclusion given by
	$$U_n^L(A)\to U_{n+1}^L(A),\ u\mapsto \begin{pmatrix}
	u &0\\0&1
	\end{pmatrix}.$$
\end{itemize}

	Let $P^L(A)$ (resp. $U^L(A)$) be the union of $P^L_n(A)$ (resp. $U^L_n(A)$). 
	In the following, we will denote by $0_n$ the zero matrix in $M_n(A)$, and by $I_n$ the identity matrix in $M_n(A)$ if $A$ is unital.
	We define the following equivalence relations.
	\begin{itemize}
		\item For $(p,k)$ and $(q,k)$ are in $P^L(A)\times\N$, we say $(p,k)\sim (q,k')$ if $p\oplus I_{j+k'}$ and $q\oplus I_{j+k}$ are homotopic in $P^{2L}(A)$ for some $j$, where $\oplus$ means direct sum.
		\item For $u$ and $v$ are in $U^L(A)$, we say $u\sim v$ if $u$ and $v$ are homotopic in $U^{2L}(A)$.
	\end{itemize}
	\begin{definition}\label{def:K*Lgroups}
		Let $A$ be a $C^*$-algebra and $\{A_L\}_{L\geqslant 0}$ be a Lipschitz filtration on $A$. If $A$ is non-unital, we denote by $\pi$ the homomorphism from $A^+$ to $\C$ with kernel $A$. We define the Lipschitz controlled $K$-groups of $A$ as follows.
		\begin{enumerate}
			\item If $A$ is unital, then
			$$K_0^L(A):=P^L(A)\times \N/\sim.$$
			If $A$ is non-unital, then
			$$K_0^L(A):=\{(p,k)\in P^L(A^+)\times \N: \mathrm{rank}(\pi(p))=k\}/\sim.$$
			\item If $A$ is unital, then
			$$K_1^L(A):=U^L(A)/\sim.$$
			If $A$ is non-unital, then
			$$K_1^L(A):=U^L(A^+)/\sim.$$
		\end{enumerate}
	\end{definition}
	\begin{lemma}\label{lemma:K_*^Lisabelian}
		$K_*^L(A)$ are abelian groups under the summation
		$$[p,k]+[p,k']=[p\oplus p',k+k']$$
		and
		$$[u]+[v]=[u\oplus v].$$
	\end{lemma}
	\begin{proof}
		Without loss of generality, we assume that $A$ is unital. Note that if $p$ and $q$ are in $P^L_n(A)$, then the homotopy
		$$
				\begin{pmatrix}
		\cos\theta&\sin\theta\\-\sin\theta&\cos\theta
		\end{pmatrix}
		\begin{pmatrix}
		p&0\\0&q
		\end{pmatrix}
		\begin{pmatrix}
		\cos\theta&-\sin\theta\\\sin\theta&\cos\theta
		\end{pmatrix},\ \theta\in[0,\frac\pi 2]$$
		that connects $p\oplus q$ and $q\oplus p$ lies in $P^L_{2n}(A)$. This shows that the summation is well-defined and commutative.
		
		The identity element of $K_0^L(A)$ is 
		$$[0,0]\sim [I_j,j]$$
		for any $j\in\mathbb N$. The inverse of $[p,k]$ is $[I_j-p,j-k]$ for some large $j$, as the homotopy
		$$
		\begin{pmatrix}
		p&0\\0&0
		\end{pmatrix}
		+\begin{pmatrix}
		I_j-p&0\\0&I_j-p
		\end{pmatrix}\begin{pmatrix}
		\cos^2\theta&\cos\theta\sin\theta\\
		\cos\theta\sin\theta&\sin^2\theta
		\end{pmatrix},\ \theta\in[0,\frac\pi 2]
		$$
		that connects $p\oplus (I_j-p)$ and $I_j\oplus 0_j$ lies in $P^{2L}_{2j}(A)$. Therefore, $K_0^L(A)$ is an abelian group.
		
		As for the $K_1$-group, the homotopy between $u\oplus v$ and $v\oplus u$ is given by
		$$
		\begin{pmatrix}
		\cos\theta&\sin\theta\\-\sin\theta&\cos\theta
		\end{pmatrix}
		\begin{pmatrix}
		u&0\\0&v
		\end{pmatrix}	\begin{pmatrix}
		\cos\theta&-\sin\theta\\\sin\theta&\cos\theta
		\end{pmatrix}, \theta\in[0,\frac\pi 2],
		$$
		which lies in $U^{L}(A)$. This shows that the summation is well-defined and commutative. The identity element is $[1]$. The inverse of $[u]\in K_1^L(A)$ is $[u^*]$, as the homotopy
		$$
		\begin{pmatrix}
		u&0\\0&1
		\end{pmatrix}
		\begin{pmatrix}
		\cos\theta&\sin\theta\\-\sin\theta&\cos\theta
		\end{pmatrix}
		\begin{pmatrix}
		u^*&0\\0&1
		\end{pmatrix}	\begin{pmatrix}
		\cos\theta&-\sin\theta\\\sin\theta&\cos\theta
		\end{pmatrix},\ \theta\in[0,\frac\pi 2]
		$$
		connects $1$ and $u\oplus u^*$ in $U^{2L}(A)$.
		\end{proof}

Since $A_L\subset A_{L'}$ if $L<L'$, there is a natural map
$$i_{L,L'}\colon K_*^L(A)\longrightarrow K_*^{L'}(A).$$
The collection of abelian groups $\{K_*^L(A)\}_{L\geqslant 0}$ together with the homomorphisms $\{i_{L,L'}\}_{0\leqslant L<L'}$ form an inductive system over $\R_{\geqslant 0}$, which we will discuss in further detail in the next subsection.

Clearly, the Lipschitz controlled $K$-groups map naturally to the usual $K$-groups:
\begin{align*}
&K_0^L(A)\to K_0(A),\ [p,k]\mapsto [p]-[I_k];\\
&K_1^L(A)\to K_1(A),\ [u]\mapsto [u].
\end{align*}
Since the union of $\{A_L\}_{L\geqslant 0}$ is dense in $A$, we have that
$$
K_*(A)=\ilim K_*^L(A).
$$

We have the functoriality of $K_*^L(\cdot)$ in the following sense.
\begin{definition}\label{def:filtered hom}
Let $A,B$ be  $C^*$-algebras with Lipschitz filtration $\{A_L\}_{L\geqslant 0}$ and $\{B_L\}_{L\geqslant 0}$ respectively. A homomorphism $f\colon A\to B$ is said to be a \emph{filtered homomorphism} if $f(A_L)\subset B_{L}$ for any $L\geqslant0$.
\end{definition}
\begin{proposition}\label{prop:functoriality+homotopy-invariance}
Let $A,B$ be Lipschitz filtered $C^*$-algebras and $f\colon A\to B$ a filtered homomorphism.Then $f$ naturally induces homomorphisms $(f_{*})_{L}\colon K_*^L(A)\to K_*^L(B)$ \textup{(}more precisely, a controlled homomorphism $f_*\colon \{K_*^L(A)\}_{L\geqslant 0}\xrightarrow{}\{K_*^L(B)\}_{L\geqslant 0}$ with control function $L\mapsto L$ in the sense of Definition $\ref{def:controlledhom}$\textup{)}.
Furthermore, if $f_t\colon A\to B$ is a strongly continuous path of filtered homomorphisms, then $(f_{t,*})_{L}$ is independent of $t$ for any $L\geqslant 0$.
	\end{proposition}
	\begin{proof}
		Without loss of generality, we assume that both $A$ and $B$ are unital, and $f$ is a unital map. In this case, $f$ extends to a filtered homomorphism on the matrix algebra, i.e., $f(M_n(A)_L)\subset M_n(B)_L$. Hence $f$ maps $P^L(A)$ (resp. $U^L(A)$) to $P^L(B)$ (resp. $U^L(B)$).
		We define 
					\begin{itemize}
			\item $(f_{0})_{L}([p,k]):=[f(p),k],$
			\item $(f_{1})_{L}([u]):=[f(u)].$
		\end{itemize}

	It is easy to see that $f_*$ is well defined. The homotopy equivalence follows directly from the equivalence relation of $K_*^L(B)$.

	\end{proof}
	
	We know that if two projections (resp. unitaries) are sufficiently close to each other, then they are homotopic via a path of projections (resp. unitaries). The analogue also holds in the Lipschitz controlled setting. More precisely, we have the following lemma. 
	\begin{lemma}\label{lemma:tech}
		Let $A$ be a unital $C^*$-algebra with a Lipschitz filtration $\{A_L\}_{L\geqslant 0}$. 
		\begin{enumerate}[$(1)$]
			\item If $p,q\in P^{L}_n(A)$ and $\|p-q\|\leqslant1/12$, then there exists a  $3$-Lipschitz homotopy connecting $p$ and $q$, that is, a norm-continuous map $h\colon[0,1]\to P^{3L}_n(A)$ such that $h$ is $3$-Lipschitz with $h(0) = p$ and $h(1) = q$. 
			\item If $u,v\in U^{L}_n(A)$ and $\|u-v\|\leqslant 1/6$, then there exists a  $1$-Lipschitz homotopy $h\colon[0,1]\to U^{3L}_n(A)$ connecting $1$ and $vu^*$. In particular, 
			$h(t)u$ is a $1$-Lipschitz map that connects $u$ and $v$ in $U^{4L}(A)$.
		\end{enumerate}
	\end{lemma}
	\begin{proof}\
		
		\begin{enumerate}
			\item  We define 
			\[ a_t\coloneqq tp+(1-t)q\in M_n(A)_L \] 
			for $t\in [0, 1]$.
			Since $\|p-q\|\leqslant \frac{1}{12} $ by assumption,  we have
			$$\|a_t^2-a_t\|\leqslant (t^2+2t(1-t))\|p-q\|\leqslant \frac{1}{6}.$$
			Thus the spectrum of $a_t$ is contained in the $\frac{1}{12}$-neighborhood of $0$ and $1$.
			
			Let $\Theta$ be the function on $\C$ given by
			$$\Theta(z)=\begin{cases}
				0& \mathrm{Re}(z)<\frac 1 2;\\
				1& \mathrm{Re}(z)\geqslant \frac 1 2.\\
			\end{cases}$$
			then $\Theta(a_t)$ is well defined and is a path of projections such that $\Theta(a_0) = p$ and $\Theta(a_1) = q$. The functional calculus of $a_t$ is defined by
			$$\Theta(a_t)=\frac{1}{2\pi i}\int_\Gamma(a_t-\xi)^{-1}d\xi
			=\frac{1}{2i}\int_\Gamma(a_t-\xi)^{-1}\frac{d\xi}{\pi},$$
			where $\Gamma=\{|z-1|=1/2\}$. As 
			$$\|(a_t-\xi)^{-1}\|\leqslant \big|1-\frac{1}{12}-\frac 1 2\big|^{-1}=\frac{12}{5}$$
			 for any $\xi\in\Gamma$, we have 
			$$\Theta(a_t)\in P^{\frac 1 2(\frac{12}{5})^2L}_n(A)\subset P_n^{3L}(A)$$
			 by (L6) in Remark \ref{rk:extra}. Furthermore, we have 
			\begin{align*}
			\|\Theta(a_t)-\Theta(a_{t'})\|&\leqslant \frac{1}{2}\int_\Gamma\|(a_t-\xi)^{-1}-(a_{t'}-\xi)^{-1}\|\frac{d\xi}{\pi}\\
			&\leqslant \frac{1}{2}\int_\Gamma\|(a_t-\xi)^{-1}\||t-t'|\|(a_{t'}-\xi)^{-1}\|\frac{d\xi}{\pi}\\
			&\leqslant \frac{1}{2}(\frac{12}{5})^2|t-t'|\leqslant 3|t-t'|.
			\end{align*}
		Hence we have proved part (1) by setting $h(t) = \Theta(a_t)$. 
			\item If $u,v\in U^{L}_n(A)$ and $\|u-v\|\leqslant 1/6$, then  $\|vu^*-1\|\leqslant \frac{1}{6}$ and $vu^*\in U^{2L}_n(A)$. We define
			$$u_t=\exp(t\log(vu^*)).$$
			By Lemma \ref{lemma:powerseries}, we see that
			$$\log(vu^*)=\log(1-(1-vu^*))\in M_n(A)_{2L/(1-\frac{1}{6})}$$ and $\|\log(vu^*)\|\leqslant -\log(1-\frac{1}{6})$. Therefore $u_t$ is a path of unitaries in $U_n^{2L/(1-\frac 1 6)^2}(A)\subset U^{3L}_n(A)$ connecting $1$ and $vu^\ast$. Furthermore,  we have
			\begin{align*}
			\|u_t-u_s\|&= \|\exp(t\log(vu^*))-\exp(s\log(vu^*))\|\\
			&\leqslant \frac{1}{1- \frac{1}{6}}\log\big(\frac{1}{1-\frac{1}{6}}\big)|t-s|< |t-s|
			\end{align*} 
					for all $t,s\in [0, 1]$. 
			This proves part (2). 
		\end{enumerate}
	\end{proof}
	
	The following technical lemma shows that, by increasing the size of matrices,  we can turn an arbitrary  homotopy of projections (resp. unitaries) into a Lipschitz homotopy with some universal Lipschitz constant. 
	\begin{lemma}\label{lemma:Liphomotopy}
	Let $A$ be a unital $C^*$-algebra equipped with a Lipschitz filtration $\{A_L\}_{L\geqslant 0}$. 
		\begin{enumerate}
			\item If $p,q\in P^{L}_n(A)$ are homotopic in $P^{L}(A)$, then there exist integers $k,l$ and a $19$-Lipschitz homotopy $h\colon[0,1]\to P^{3L}_{n+k+l}(A)$ that connects $p\oplus I_k\oplus 0_l$ and $q\oplus I_k\oplus 0_l$.
			\item If $u,v\in U^{L}(A)$ are homotopic in $U^{L}(A)$, then there exists a $11$-Lipschitz homotopy $h\colon[0,1]\to U^{4L}(A)$ which connects $u$ and $v$.
		\end{enumerate}
	\end{lemma}
	\begin{proof}\ 
		\begin{enumerate}
			\item By assumption, there is a continuous path of projections in $P^L(A)$ connecting $p$ and $q$. It follows that there exist $p_0,p_1,\cdots,p_m$ such that $p_i\in P^L_j(A)$, $p_0=p\oplus 0_{j-n}$, $p_m=q\oplus 0_{j-n}$, and $\|p_i-p_{i+1}\|\leqslant 1/12$. Write $k=2mj$. We consider the following homotopy:
			\begin{align*}
			& p\oplus I_k\oplus 0_{k+(j-n)} & \\
			\sim & p_0\oplus I_j\oplus 0_j\oplus\cdots\oplus I_j\oplus 0_j
			&(\pi\text{-Lipschitz},P^L(A)) \\
			\sim &p_0\oplus(I_j-p_1)\oplus p_1\oplus\cdots \oplus(I_j-p_m)\oplus p_m
			&(\pi\text{-Lipschitz},P^{2L}(A)) \\
			\sim&p_0\oplus (I_j-p_0)\oplus p_1\oplus\cdots \oplus(I_j-p_{m-1})\oplus p_m
			&(3\text{-Lipschitz},P^{3L}(A))\\
			\sim&I_j\oplus 0_j\oplus\cdots\oplus I_j\oplus 0_j\oplus p_m
			&(\pi\text{-Lipschitz},P^{2L}(A))\\
			\sim& p_m\oplus 0_j\oplus I_j\oplus\cdots\oplus 0_j\oplus I_j
			&(\pi\text{-Lipschitz},P^{L}(A))\\
			\sim& q\oplus I_{k}\oplus 0_{k+(j-n)}&(\pi\text{-Lipschitz},P^{L}(A))
			\end{align*}
			where for example $(\pi\text{-Lipschitz}, P^L(A))$ means the first homotopy is a $\pi$-Lipschitz homotopy taking place in $P^L(A)$. By concatenating these homotopies together, we obtain a $19$-Lipschitz homotopy $h\colon[0,1]\to P^{3L}(A)$ connecting $p\oplus I_k\oplus 0_l$ and $q\oplus I_k\oplus 0_l$. 
			\item By assumption, there is a continuous path of unitaries in $U^L(A)$ connecting $u$ and $v$. It follows that there exist $u_0,u_1,\cdots,u_m$ such that $u_i\in U_n(A)_L $, $u_0=u$, $u_m=v$, and $\|u_i-u_{i+1}\|\leqslant 1/6$.  We consider the following homotopy:
			\begin{align*}
			&u_0\oplus I_{2mn}= u_0\oplus I_{2n}\oplus\cdots\oplus I_{2n}\\
			\sim&u_0\oplus u_1^*\oplus u_1\oplus\cdots \oplus u_m^*\oplus u_m
			&(\pi\text{-Lipschitz},U^{2L}(A))\\
			\sim&u_0\oplus u_0^*\oplus u_1\oplus u_1^*\oplus\cdots\oplus u_{m-1} \oplus u_{m-1}^*\oplus u_m
			&(1\text{-Lipschitz},U^{4L}(A))\\
			\sim&I_{2mn}\oplus u_m=I_{n}\oplus I_{(2m-1)n}\oplus u_m
			&(\pi\text{-Lipschitz},U^{2L}(A))\\
			\sim& u_m\oplus I_{(2m-1)n}\oplus I_n=u_m\oplus I_{2mn}
			&(\pi\text{-Lipschitz},U^{L}(A))
			\end{align*}
			where for example $(\pi\text{-Lipschitz}, U^{2L}(A))$ means the first homotopy is a $\pi$-Lipschitz homotopy taking place in $U^{2L}(A)$. By concatenating these homotopies together, we obtain a $11$-Lipschitz homotopy $h\colon[0,1]\to U^{4L}(A)$ connecting $u$ and $v$.
			
		\end{enumerate}
	\end{proof}
	
	We know if two projections are homotopic via a path of projections, then they are stably unitarily equivalent. The following is an analogue for Lipschitz controlled $K$-theory.
	\begin{lemma}\label{lemma:homotopy->unitaryeq}
	Let $A$ be a unital $C^*$-algebra and $\{A_L\}_{L\geqslant 0}$  a Lipschitz filtration on $A$.
		If $p,q\in P^L_n(A)$ are homotopic in $P^L(A)$, then there exist $k,l$ and a unitary $u\in U_{n+k+l}^{1635L}(A)$ such that $p\oplus I_k\oplus 0_l=u(q\oplus I_k\oplus 0_l)u^*$. Furthermore, $u$ can be chosen so that it is homotopic to the identity in $U_{n+k+l}^{2087L}(A)$.
	\end{lemma}
	\begin{proof}
		By Lemma \ref{lemma:Liphomotopy}, there exists a $19$-Lipschitz homotopy $h$ in $P^{3L}_{n+k+l}(A)$ that connects  $p\oplus I_k\oplus 0_l$ and $q\oplus I_k\oplus 0_l$. Denote $m=n+k+l$.
		
		Write $p_j=h(j/190)$. Set
		$$z_j=1+(p_{j+1}-p_j)(2p_j-1) \textup{ and } u_j=z_j(z_j^*z_j)^{-\frac 1 2}.$$
		Then we have $p_{j+1}=u_jp_ju_j^*$. Hence if we set $u=u_{190}u_{189}\cdots u_1u_0$, then $p_{190}=up_0u^*$.

		Write $\varepsilon=1/10$ for short.
		As $\|p_{j+1}-p_j\|\leqslant \varepsilon$, we have 
		$z_j\in M_m(A)_{(1+\varepsilon)6L}$ and $\|z_j\|\leqslant 1+\varepsilon$. Besides, we have $z_j^*z_j=1-(p_{j+1}-p_j)^2$. Therefore we have $z_j^*z_j\in M_m(A)_{12\varepsilon L}$, $\|1-z_j^*z_j\|\leqslant \varepsilon^2$, and $\|(z_j^*z_j)^{-1/2}\|=\|z_j\|\leqslant1+\varepsilon $. By Lemma \ref{lemma:powerseries}, we have 
		$$(z^*_jz_j)^{-\frac 1 2}=(1-(1-z_j^*z_j))^{-\frac 1 2}\in 
		M_m(A)_{(1-\varepsilon^2)^{-\frac 3 2}\cdot 12\varepsilon L}.$$
		Therefore  
		$$u_j\in U_m(A)_{6(1+\varepsilon)^2L+(1+\varepsilon)(1-\varepsilon^2)^{-\frac 3 2}\cdot 12\varepsilon L}\subset U_m^{\frac{1635}{190}}.$$
		Write $\delta= 6(1+\varepsilon)^2+12\varepsilon(1+\varepsilon)(1-\varepsilon^2)^{-\frac 3 2}$ for short. 
		Thus $u$ is a unitary in $U_m^{\frac{19\delta}{\varepsilon}L}(A)\subset U_m^{1635}(A)$. 
		
		To show that this $u$ is homotopic to the identity in $U_{n+k+l}^{2087L}(A)$,  we first notice that
		$$\|1-(z_j^*z_j)^{-\frac 1 2}\|\leqslant (1-\varepsilon^2)^{-\frac 1 2}-1.$$
		Therefore
		$$\|u_j-1\|\leqslant \|(z_j^*z_j)^{-\frac 1 2}\|\|z_j-1\|+\|(z_j^*z_j)^{-\frac 1 2}-1\|\leqslant (1+\varepsilon)\varepsilon+(1-\varepsilon^2)^{-\frac 1 2}-1.$$
		Set $a_j=\log(u_j)$. Again by Lemma \ref{lemma:powerseries}, we have
		$$\|a_j\|\leqslant -\log(1-\|1-u_j\|)=-\log(2-(1+\varepsilon)\varepsilon-(1-\varepsilon^2)^{-\frac 1 2}),$$ 
		and $a_j\in M_m(A)_{(2-(1+\varepsilon)\varepsilon-(1-\varepsilon^2)^{-\frac 1 2})^{-1}\delta L}$. Now we see from Lemma \ref{lemma:powerseries} that
		$u_j$ is connected to the identity by 
		$$\exp(t a_j) \in U_m^{(2-(1+\varepsilon)\varepsilon-(1-\varepsilon^2)^{-\frac 1 2})^{-2}\delta L}(A)
		\subset U_m^{\frac{2087}{190}L}(A).$$
		Therefore $u$ is connected to the identity by
		$$\exp(t\log a_{190})\exp(t\log a_{189})\cdots\exp(t\log a_0)\in U_m^{2087L}(A).$$
	\end{proof}
	
	\subsection{Inductive systems and controlled homomorphisms}
	
	In this subsection, we discuss some basic properties of inductive systems and  controlled homomorphisms between them.
	\begin{definition}\label{def:inducsys}
		Let $\{M_L\}_{L\geqslant 0}$ be a collection of abelian groups index by the set $\R_{\geqslant 0} = [0, \infty)$. Assume there is a homomorphism
		$$i_{L,L'}\colon M_L\to M_{L'}$$
		for all $L<L'$.
		We call $\{M_L\}_{L\geqslant 0}$ an\emph{ inductive system} if $i_{L,L''}=i_{L',L''}\circ i_{L,L'}$ when $L<L'<L''$. 
	\end{definition}
	
We will be only concerned with inductive systems indexed by the set $\R_{\geqslant 0}$. For notational simplicity, we shall omit the subscript and write $\{M_L\}$ in place of $\{M_L\}_{L\geqslant 0}$ from now on.
	
	Let $\ilim M_L$ be the inductive limit of $\{M_L\}$. We will denote the natural map from $M_L $ to $\ilim M_L$ by $i_L$.
	\begin{definition}\label{def:controlledhom}
			A function $F\colon\R_{\geqslant 0}\to \R_{\geqslant 0}$ is called a \emph{control function} if $F$ is non-decreasing and  $F(x) \to \infty$, as $x\to \infty$. Let $\{M_L\}$ and $\{M_L'\}$ be two inductive systems. We say a collection of group homomorphisms $\{\xi_L\colon M_L\to M_{F(L)}'\}$ is a \emph{controlled homomorphism} from $\{M_L\}$ to $\{M_L'\}$   with control function $F$ if the following diagram commutes:
		$$\xymatrix{M_L\ar[rr]^{\xi_L}\ar[d]_{i_{L,L'}}&&M_{F(L)}'\ar[d]^{i'_{F(L),F(L')}}\\
			M_{L'}\ar[rr]^{\xi_{L'}}&&M_{F(L')}'
		}
		$$
		for all $L< L'$.  From now on, we shall denote  such a controlled homomorphism by  $\xi\colon \{M_L\} \xlongrightarrow[F]{} \{M_L'\}$ or
		${\{M_L\}_{L \geqslant 0}\xlongrightarrow[F]{\xi}\{M_L'\} }$. If the control function $F$ is clear from the context, we will simply write  $\xi\colon \{M_L\}\to \{M_L'\}_{L \geqslant 0}$ or
		${\{M_L\}\xlongrightarrow{\xi}\{M_L'\} }$ instead. 
	\end{definition}
	\begin{definition}\label{def:similar}
		Let $F$, $G$ and $H$ be control functions such that $H(x)\geqslant G(x)$ and $H(x)\geqslant F(x)$ for all $x\in\mathbb R_{\geqslant 0}$. Given two controlled homomorphisms		$\xi\colon \{M_L\}\xlongrightarrow[F]{}\{M_L'\}$ and $\eta\colon \{M_L\} \xlongrightarrow[G]{}\{M_L'\} $, we say $\xi$ is controlled equivalent to $\eta$ with control function $H$ if the following diagram commutes:
		$$\xymatrix{M_L\ar[rr]^{\xi_L}\ar[d]_{\eta_L}&&M_{F(L)}'\ar[d]^{i'_{F(L),H(L)}}\\
			M_{G(L)}'\ar[rr]^{i'_{G(L),H(L)}}&&M_{H(L)}'
		}
		$$
		In this case, we write $\xi\sim_H\eta$ or simply $\xi\sim\eta$.
	\end{definition}
\begin{remark}\label{rk:controlled-iso-com}
		A controlled homomorphism between two inductive systems naturally induces a homomorphism between the inductive limits, and equivalent controlled homomorphisms induces the same map. We say a controlled homomorphism is a \emph{controlled isomorphism} if it admits an inverse up to controlled equivalences in the sense of Definition \ref{def:similar}. In this case, the induced map on the inductive limit is an isomorphism. Similarly, we can define the commutativity of a diagram of controlled homomorphisms up to controlled equivalences.
\end{remark}
	\begin{definition}\label{def:fullyfaithful}
	Given an inductive system  $\{M_L\}$, let us denote by $i_L$ the natural map $M_L\to \varinjlim M_L$. We say $\{M_L\}$ is \emph{uniformly controlled} if there exist $L_0\geqslant 0$ and a control function $F\colon \mathbb R_{\geqslant 0}\to \mathbb R_{\geqslant 0} $ such that 
\begin{itemize}
	\item for any $L\geqslant L_0$, the map $i_L\colon M_L\to \varinjlim M_L$ is surjective;
	\item if $i_L(x)=0$ in $\varinjlim M_L$ for some $x\in M_L$, then $i_{L,F(L)}(x)=0$ in $M_{F(L)}$.
\end{itemize}
We call $(L_0,F)$ a \emph{uniform control pair} of $\{M_L\}$. And we shall say $\{M_L\}$ is $(L_0,F)$-uniformly controlled if we want to specify the uniform control pair $(L_0,F)$.
	\end{definition}

By definition, $\{M_L\}$ is uniformly controlled if and only if the natural map from $\{M_L\}$ to the constant inductive system\footnote{A constant inductive system $\{N_L\}$ is an inductive system where every $N_L$ is equal to   some fixed abelian group $A$.} 
$\{ \ilim M_L \}$  is a controlled isomorphism. Therefore we have the following lemma.

\begin{lemma}\label{lemma:fullyfaithful-iso}
	Let $\{M_L\}$ and $\{M_L'\}$ be inductive systems. Suppose that $\{M_L\}$ and $\{M_L'\}$ are controlled isomorphic, i.e., there are controlled homomorphisms 
	$\xi\colon\{M_L\}\xrightarrow{}\{M_L'\}$ and $\eta\colon\{M_L'\}\xrightarrow{}\{M_L\}$ such that $\xi\circ\eta$ and $\eta\circ\xi$ are controlled equivalent to the identity map respectively. If $\{M_L\}$ is uniformly controlled, then $\{M_L'\}$ is also uniformly controlled. Moreover,  the uniform control pair of $\{M_L'\}$ only depends on the uniform control pair of $\{M_L\}$, the control functions of $\xi,\eta$, and the control functions of $\xi\circ\eta\sim\id$ and $\eta\circ\xi\sim\id$.
\end{lemma}

	\begin{definition}\label{def:asymp}
		Let $F$, $G$, $F_1$ and $F_2$ be control functions.  Given two controlled homomorphisms		$\xi\colon \{M_L\}\xlongrightarrow[F]{}\{M_L'\}$ and $\eta\colon \{M'_L\} \xlongrightarrow[G]{}\{M_L''\} $, we say  the sequence 
		$$\{M_L\}\xlongrightarrow[F]{\xi}\{M_L'\}\xlongrightarrow[G]{\eta}\{M_L''\} $$
		is \emph{asymptotically exact} at $\{M_L'\}$ with control functions $F_1$ and $F_2$ if
		\begin{itemize}
			\item $\eta\circ\xi\sim_{F_1}0$;
			\item for any $m'\in M_L'$ with $\eta_L(m')=0$, there exists $m\in M_{F_2(L)}$ such that 
			\[   \xi_{F_2(L)} (m) = i'_{L, F(F_2(L))}(m')\] 
			in $M_{F(F_2(L))}'$.
		\end{itemize} 
	\end{definition}

We have the following five lemma type result on the uniform controls of inductive systems.   
	\begin{lemma}\label{lemma:fivelemma}
		Assume that the following sequence
		$$\{M_L^1\}\xrightarrow{\xi^1}\{M_L^2\}\xrightarrow{\xi^2}\{M_L^3\}\xrightarrow{\xi^3}\{M_L^4\}\xrightarrow{\xi^4}\{M_L^5\} $$
		is asymptotically exact. If $\{M_L^i\}$ is uniformly controlled for $i=1,2,4,5$, then $\{M_L^3\}$ is uniformly controlled. Moreover, the uniform control pair of $\{M_L^3\}$ only depends on the control functions of $\xi^i$, $i=1,2,3,4$, the control functions of the asymptotic exactness at $\{M_L^i\}$, $i=2,3,4$, and the uniform control pairs of $\{M_L^i\}$, $i=1,2,4,5$.
	\end{lemma}
	\begin{proof}
		Let $F_j$ be the control function of $\xi^j$ (cf. Definition \ref{def:controlledhom}) and $Z_{j+1,j}$ the control function of the controlled equivalence $\xi^{j+1}\xi^j\sim 0$ (cf. Definition \ref{def:similar}). 
		Let $E_{j-1,j}$ be the control function of the exactness at $\{M_L^{j}\}$ (cf. Definition $\ref{def:asymp}$). Suppose the inductive system $\{M_L^i\}$ is $(L_j,U_j)$-uniformly controlled for $j=1,2,4,5$ (cf. Definition \ref{def:fullyfaithful}. Let us write  $M^j=\ilim M^j_L$ and denote the map $M^j \to M^{j+1}$ induced by  $\xi^j\colon \{M_L^j\} \to \{M_L^{j+1}\}$ still by $\xi^j$.
		
		Given $a\in M^3$, as $\{M_L^4\}$ is uniformly controlled, there exists $b\in M^4_{L_4}$ such that $b$ eventually equals $\xi^3(a)$ in $M^4$. Therefore $\xi^4_{L_4}(b)$ is eventually zero in $M^5$, thus zero in $M^5_{U_5(F_4(L_4))}$. Set 
		$$L_{4+}=\inf\{t : F_4(t)\geqslant U_5(F_4(L_4)) \}+1,$$
		 so that $F_4(L_{4+})\geqslant U_5F_4(L_4)$. Therefore $i_{L_{4},L_{4+}}^4(b)$ is in the kernel of the map $\xi^4_{L_{4+}}$. It follows that there exists  $c\in M^3_{E_{3,4}(L_{4+})}$ such that $\xi^3(c) = i_{L_{4},L_{4+}}^4(b)$ due to the exactness at $\{M_L^4\}$.
		
		Set $d=a-i^3_{E_{3,4}(L_{4+})}(c)$. We see that $\xi^3(d)=0$ as $\xi^3i^3(c)=\xi^3(a)$. Therefore $d$ is the image of $\xi^2\colon M^2 \to M^3$, hence there exists $e\in M^2_{L_2}$ such that $\xi^2(e) = d$. Set 
		$$L_3=\max\{F_2(L_2),E_{3,4}(L_{4+})\}.$$
		Then 
		$$i_{L_3}^3(i^3_{E_{3,4}(L_{4+}),L_3}(c)+i^3_{F_2(L_2),L_3}\xi^2_{L_2}(e))=a.$$
		Thus $i^3_{L_3}\colon M^3_{L_3}\to M^3$ is surjective, so is $i_{L}^3$ for all  $L\geqslant L_3$.
		
		On the other hand, if $x\in M_L^3$ is eventually zero in $M^3$, then $\xi^3_L(x)$ is eventually zero in $M^4$, hence zero in $M^4_{U_2F_3(L)}$. Set 
		$$L_+=\inf\{t : F_3(t) \geqslant U_2F_3(L)\}+1$$
		so that $F_3(L_+)\geqslant U_2F_3(L)$.
		Then we have  $\xi^3_{L_+}i^3_{L,L_+}(x)=0$.  So $i^3_{L,L^+}(x)$ is the image of some $y\in M^2_{E_{2,3}(L_+)}$ under the map $\xi^2$. As $i^2_{E_{2,3}(L_+)}(y)\in M^2$ is in the kernel of the map $\xi^2 \colon M^2 \to M^3$, it is in the image of $\xi^1\colon M^1 \to M^2$, hence the image of some $z\in M^1_{L_1}$. Set
		 $$L_{++}=\inf\{t : t \geqslant L_1,\ F_1(t)\geqslant E_{2,3}(L_+) \}+1$$
		 so that $F_1(L_{++})\geqslant E_{2,3}(L_+)$. Set $w=i^1_{L_1,L_{++}}(z)$. Then we have  \[ \xi^1_{L_{++}}(w)=i^2_{E_{2,3}(L_+),F_1(L_{++})}(y), \] hence
		$$i^3_{L,F_2F_1(L_{++})}(x)= \xi^2_{F_1(L_{++})}\xi^1_{L_{++}}(w).$$
		As $\xi^2\xi^1$ is controlled equivalent to zero with control function $Z_{2,1}$, we see that $x$ equals zero in $M^3_{F_3(L)}$, where
		$F_3(L)=Z_{2,1}(L_{++})$.
	\end{proof}

The following two lemmas can be proved in a similar way, by a standard diagram chasing. We omit the proofs. 
	\begin{lemma}\label{lemma:splitexact}
		Assume that the following sequence
		$$\{M_L^1\}\xrightarrow{\xi^1}\{M_L^2\}\xrightarrow{\xi^2}\{M_L^3\}\xrightarrow{\xi^3}\{M_L^4\}$$
		is asymptotically exact. If $\xi^1$ is controlled equivalent to zero and there exists a controlled homomorphism $s\colon\{M_L^4\}\rightarrow\{M_L^3\}$ such that $\xi^3\circ s$ is controlled equivalent to the identity map, then there exists a controlled homomorphism $p\colon\{M_L^3\}\rightarrow \{M_L^2\}$ such that $p\circ\xi^2$ is controlled equivalent to identity. Moreover, the control function of $p$ and the control function for $p\circ\xi^2\sim\mathrm{id}$ only depend on the  control functions of $\xi^i$, $i=1,2,3$, the  control function of the controlled equivalence $\xi^1\sim 0$ and the  control functions of  the asymptotic exactness.
	\end{lemma}

\begin{remark}\label{rk:naturalofsplit}
	Also by diagram chasing, we see that the map $p\colon\{M_L^3\}\rightarrow \{M_L^2\}$ in Lemma \ref{lemma:splitexact} is natural in the following sense. If the following diagram
		$$\xymatrix{\{M_L^1\}\ar[r]^{\xi^1}\ar[d]_{\eta^1}&\{M_L^2\}\ar[r]^{\xi^2}\ar[d]_{\eta^2}&\{M_L^3\}\ar[r]^{\xi^3}\ar[d]_{\eta^3}&\{M_L^4\}\ar[d]_{\eta^4}\\
			\{N_L^1\}\ar[r]^{\zeta^1}&\{N_L^2\}\ar[r]^{\zeta^2}&\{N_L^3\}\ar[r]^{\zeta^3}&\{N_L^4\}
}
$$
commutes up to controlled equivalences and  each row satisfies the conditions in Lemma \ref{lemma:splitexact}, then the diagram
		$$\xymatrix{\{M_L^2\}\ar[d]_{\eta^3}&\{M_L^3\}\ar[l]_{p}\ar[d]_{\eta^2}\\
	\{N_L^2\}&\{N_L^3\}\ar[l]_{q}
}
$$
commutes up to controlled equivalences, of which the control function only depends on the given control functions.
\end{remark}
	\begin{lemma}\label{lemma:epi-mono}
		Assume that the following sequence
		$$\{M_L^1\}\xrightarrow{\xi^1}\{M_L^2\}\xrightarrow{\xi^2}\{M_L^3\}\xrightarrow{\xi^3}\{M_L^4\}$$
		is asymptotically exact. 
		\begin{enumerate}[$(1)$]
			\item If there exists $s\colon\{M_L^3\}\rightarrow\{M_L^2\}$ such that $\xi^2\circ s$ $($resp. $s\circ\xi^2$$)$ is
			 controlled equivalent to the identity map, then $\xi^3$ $($resp. $\xi_1$$)$ is controlled equivalent to zero. Moreover, the control function of the controlled equivalence $\xi^1\sim 0$ $($resp. $\xi^3\sim 0$$)$ only depends on the control functions of $\xi^1,\xi^2,\xi^3,s$, $\xi^2\circ s\sim\id$ $($resp. $s\circ\xi^2\sim\id$$)$ and the control functions of the asymptotic exactness.
			\item If $\xi^1$ and $\xi^3$ are both controlled equivalent to zero, then there exists a controlled homomorphism $s\colon\{M_L^3\}\rightarrow\{M_L^2\}$ such that $\xi^2\circ s$ 
			and $s\circ \xi^2$ are controlled equivalent to the identity map respectively. Moreover, the control functions of $s$ and of the controlled equivalences only depend on the control functions of $\xi^1,\xi^2,\xi^3$, $\xi^1\sim 0$, $\xi^3\sim 0$ and of the asymptotic exactness.
		\end{enumerate}
	\end{lemma}

	\subsection{Asymptotically exact sequence for Lipschitz controlled $K$-theory}
	In this subsection, we prove the asymptotically exact sequence for Lipschitz controlled $K$-theory. 
	
	Let $A$ be a Lipschitz filtered $C^*$-algebra with a Lipschitz filtration $\{A_L\}_{L\geqslant 0}$.
	Let $J$ be an ideal of $A$. Denote $Q=A/J$ and let $r\colon J \to A$ and $\pi\colon A\to Q$ be the inclusion and quotient map respectively. We assume that $\{J\cap A_L\}_{L\geqslant 0}$ is a Lipschitz filtration on $J$. Or equivalently, assume that $J$ is Lipschitz filtered and the inclusion $r\colon J \to A$ is a filtered homomorphism. Also, assume that we have a Lipschitz filtration $\{Q_L\}_{L\geqslant 0}$ on $Q$.
	\begin{definition}\label{def:filtered-exact-sequence}
		Let $F_l\colon \R_{\geqslant 0}\to \R_{\geqslant 0}$ be a control function such that $F_l(x)\geqslant 0$.	
	With the above notation, we say the short exact sequence
		$$0\longrightarrow J\xlongrightarrow{r}A\xlongrightarrow{\pi}Q\longrightarrow0
		$$
		is controlled if the following are satisfied: 
		\begin{enumerate}
			\item $\pi$ is a \emph{filtered homomorphism}, i.e., $\pi(A_L )\subset Q_L$ for any $L\geqslant 0$,
			\item  $\pi$ is a \emph{controlled surjection} with  control function $F_l$, i.e., for any $x\in M_n(Q)_L$ there exists some $y\in M_n(A)_{F_l(L)\cdot \|x\|}$ such that 	$\|y\|\leqslant 2\|x\|$ and  $\pi(y)=x$. 
		\end{enumerate}
	\end{definition}

We need the following technical lemma to define the boundary maps in the long asymptotically exact sequence for  Lipschitz controlled $K$-theory.
	\begin{lemma}\label{lemma:liftunitary}
		Let $A$ and $Q$ be unital $C^\ast$-algebras and  $\pi\colon A\to Q$ a unital controlled surjection with control function $F_l$ in the sense of Definition $\ref{def:filtered-exact-sequence}$.
		If $v\in U^L(Q)$ is homotopic to the identity in $U^{L}(Q)$, then there exists $u\in U^{14F_l(L)}(A)$ which is homotopic to identity in $U^{14F_l(L)}(A)$ such that $\pi(u)=v$.
	\end{lemma}
	\begin{proof}
		By Lemma \ref{lemma:Liphomotopy}, there exists a $11$-Lipschitz homotopy $h:[0,1]\to U^{4L}_n(Q)$ that connects $1$ and $v$ for some $n\in \mathbb N$.
		
		Let $v_j=h(j/55)$, $j=0,1,\cdots,55$. We have $\|v_{j+1}v_j^*-1\|\leqslant 1/5$. Therefore $$\log(v_{j+1}v_j^*)=\log(1-(1-v_{j+1}v_j^*))\in M_n(Q)_{10L}$$
		 and 
		$\|\log(v_{j+1}v_j^*)\|\leqslant \log(5/4)$. Since $\pi\colon A\to Q$ is a controlled surjection with control function $F_l$, there exists 
		$a_j\in M_n(A)_{\log(5/4)F_l(L)}$
		with $\pi(a_j)=\log(v_{j+1}v_j^*)$ and $\|a_j\|\leqslant 2\|\log(v_{j+1}v_j^*)\|$.
		We may assume that $a_j^*=-a_j$ as well.
		
		Let us define
		$$u=\exp(a_{54})\exp(a_{33})\cdots \exp(a_1)\exp(a_0).$$
		We have $u$ is a unitary in $ U^{110\cdot(5/4)\cdot\log(5/4)F_l(L)}(A)$ with $\pi(u)=v$. The homotopy that connects $u$ and the identity is given by
		$$t\mapsto \exp(ta_{54})\exp(ta_{53})\cdots \exp(ta_1)\exp(ta_0),$$
		which lies in $U^{110\cdot(5/4)\cdot\log(5/4)F_l(L)}(A)\subset U^{14F_l(L)}(A)$. This finishes the proof. 
	\end{proof}

	Let 	$$0\longrightarrow J\xlongrightarrow{r}A\xlongrightarrow{\pi}Q\longrightarrow0
	$$ be a controlled short exact sequence of Lipschitz filtered $C^\ast$-algebras as in Definition \ref{def:filtered-exact-sequence}.  
	We shall define the  boundary map
	$$\{K_1^L(Q)\}_{L\geqslant 0}\xrightarrow{\partial } \{K_0^L(J)\}_{L\geqslant 0}$$ with control function $r\mapsto 98F_l(3r)$ as follows.
	Without loss of generality, we assume that $A$ and $Q$ are unital.
	
	Let $u$ be a unitary in $M_n(Q)_L$, and $w\in U_{2n}^{14F_l(L)}(A)$ a unitary that lifts $u\oplus u^{*}$ (that is, $\pi(w) = u\oplus u^{*}$) guaranteed by Lemma \ref{lemma:liftunitary}. Then  $w(I_n\oplus 0_n)w^*$ is a projection in $P^{28F_l(L)}_{2n}(J^+)\subset P^{126F_l(3L)}_{2n}(J^+)$.
	\begin{definition}
We define the boundary map $$\{K_1^L(Q)\}_{L\geqslant 0}\xrightarrow{\partial } \{K_0^L(J)\}_{L\geqslant 0}$$ 
by setting
$$\partial ([u])=[w(I_n\oplus 0_n)w^*,n]\in K_0^{126F_l(3L)}(J).$$
	\end{definition}
	It is well-defined by the following.
	\begin{enumerate}[(a)]
		\item First, we show that $\partial([u]) = \partial([u_1]) $ where $u_1=u\oplus I_{n_1}$ in $U_{n+n_1}^L(Q)$. Note that 
		$$w_1=\begin{pmatrix}
		w_{11}&0&w_{12}&0\\0&I_{n_1}&0&0\\w_{21}&0&w_{22}&0\\0&0&0&I_{n_1}
		\end{pmatrix}$$
		is a unitary in $U_{2(n+n_1)}(A)_{14F_l(L)}$ that lifts  $u_1\oplus u_1^*$, where we have written 
		$$w=\begin{pmatrix}
		w_{11}&w_{12}\\w_{21}&w_{22}
		\end{pmatrix}\in U_{2n}^{14F_l(L)}(A)$$
	in the block-matrix form. 
		Apparently
		$$w_1(I_{n+n_1}\oplus 0_{n+n_1})w_1^*\text{ and } w(I_n\oplus 0_n)w^*\oplus I_{n_1}\oplus 0_{n_1}$$
		are homotopic in $P^{28F_l(L)}(J^+)$.
		\item Second, we show that the class $\partial([u])$ in $K_0^{126F_l(3L)}(J)$ is independent of the choice of $w \in U_n^{14F_l(L)}(A)$ that lifts $u\oplus u^\ast$.  Indeed, if $\bar w$ is another unitary in $U_n^{14F_l(L)}(A)$ that lifts $u\oplus u^{-1}$, then we have $\bar w w^*\in U_{2n}^{28F_l(L)}(J^+)$. Note that  
		\[ \bar w(I_n\oplus 0_n)\bar w^*=(\bar ww^*)w(I_n\oplus 0_n)w^*(\bar ww^*)^*. \] By Lemma \ref{lemma:K_*^Lisabelian}, $\bar ww^*\oplus (\bar ww^*)^*$ is homotopic to the identity in $U_{4n}^{56F_l(L)}(J^+)$. It follows that 
		\begin{align*}
		&\bar w(I_n\oplus 0_n)\bar w^*\oplus 0_{2n}\\
		=&\big(\bar ww^*\oplus (\bar ww^*)^*\big)\big(w(I_n\oplus 0_n)w^*\oplus 0_{2n}\big)\big((\bar ww^*)^*\oplus \bar ww^*\big)
		\end{align*}
		is homotopic to $w(I_n\oplus 0_n)w^*\oplus 0_{2n}$ in $P^{126F_l(L)}(J^+)$.
		\item Finally, we show that $\partial([u]) = \partial([u'])$ in $K_0^{126F_l(3L)}(J)$, if $u'\in U_n^L(Q)$ is equivalent to $u$ in $K_1^L(Q)$, that is, $u'$ is homotopic to $u$ in $U^{2L}(Q)$. Indeed, in this case,  $u^*u'$ and $u(u')^*$ are homotopic to the identity in $U^{3L}(Q)$ respectively. We may assume that the homotopy is realized in $U^{3L}_n(Q)$ due to part (a). Hence there exist unitaries $v_1$ and $v_2$ in $U_{n}^{14F_l(3L)}(A)$ that lift 
		$u'u^*$ and $(u')^*u$ respectively. Set 
		$$v=v_1\oplus v_2\in U_{2n}^{14F_l(3L)}(A).$$
		Then we have
		$$wv(I_n\oplus 0_n)v^*w^*=w(I_n\oplus 0_n)w^*\in P_{2n}^{28F_l(L)}(J^+).$$
		
		Suppose $w'\in U_{2n}^{14F_l(L)}(A)$ is a lift of $u'\oplus (u')^*$. Note that $wv$ is also a lift of $u'\oplus (u')^*$ in $U_{2n}^{14F_l(3L)+14F_l(L)}(A)\subset U_{2n}^{28F_l(3L)}(A)$. Similar to the proof of part (b), we have $wv(w')^*\in U_{2n}^{56F_l(3L)}(J^+)$. Note that  
		\[ wv(I_n\oplus 0_n)v^*w^*=\big(wv(w')^*\big)w'\big(I_n\oplus 0_n\big)(w')^*\big(w'v^*w^*\big). \] 
		By Lemma \ref{lemma:K_*^Lisabelian}, $wv(w')^*\oplus w'v^*w^*$ is homotopic to the identity in $U_{4n}^{112F_l(3L)}(J^+)$. It follows that 
		\begin{align*}
		&wv(I_n\oplus 0_n)v^*w^*\oplus 0_{2n}\\
		=&\big(wv(w')^*\oplus  w'v^*w^*\big)\big(w'(I_n\oplus 0_n)(w')^*\oplus 0_{2n}\big)\big( w'v^*w^*\oplus wv(w')^*\big)
		\end{align*}
		is homotopic to $w'(I_n\oplus 0_n)(w')^*\oplus 0_{2n}$ in $P^{252F_l(3L)}(J^+)=P^{2\times 126F_l(3L)}(J^+)$.

	\end{enumerate}
	
	Let us now prove part of the asymptotically exact sequence of Lipschitz controlled $K$-theory. The full asymptotically exact sequence will be proved after we prove a controlled Bott periodicity for Lipschitz controlled $K$-theory, cf. Theorem \ref{prop:Bottper} and Theorem \ref{thm:six-term} below.  
	\begin{proposition}\label{prop:longexact-v1}
		Let 	$$0\longrightarrow J\xlongrightarrow{r}A\xlongrightarrow{\pi}Q\longrightarrow0
		$$ be a controlled short exact sequence of Lipschitz filtered $C^\ast$-algebras as in Definition $\ref{def:filtered-exact-sequence}$. Then we have the following asymptotically exact sequence
		\begin{align*}
		\{K_1^L(J)\}\xrightarrow{r_1} \{K_1^L(A)\}\xrightarrow{\pi_1}\{K_1^L(Q)\}
		\xrightarrow{\partial}\{K_0^L(J)\}\xrightarrow{r_0} \{K_0^L(A)\}\xrightarrow{\pi_0}\{K_0^L(Q)\}
		\end{align*}
		where the control function for the asymptotic exactness at each term only depends on the control function $F_l$ of the controlled surjection $\pi$. 
	\end{proposition}
	\begin{proof}
		Without loss of generality, we assume that $A$ and $Q$ are unital.
		\begin{itemize}
			\item As $\pi\circ r=0$ and $r,\pi$ are filtered homomorphisms, we have $\pi_0r_0\sim 0$ and $\pi_1r_1\sim 0$ with control function $L\mapsto L$.
			\descitem{Exactness at $\{K_0^L(A)\}_{L\geqslant 0}$ with control function $$L\mapsto 28F_l(4174L)+L.$$}
			Suppose $[p,k]$ is in $K_0^L(A)$ with $p\in P_n^L(A)$. If $\pi_{0,L}([p,k])=0$, then there exist $j\in \mathbb N$ and a unitary $u\in U_{n+2j}^{3270L}(Q)$ such that
			$$\pi(p)\oplus I_j\oplus 0_j=u(I_{k+j}\oplus 0_{j+n-k})u^*$$
			by Lemma \ref{lemma:homotopy->unitaryeq}. As $u$ is homotopic to identity in $U^{4174L}_{n+2j}(Q)$, there exists $v\in U^{14F_l(4174L)}_{n+2j}(A)$ such that $\pi(v)=u$ by Lemma \ref{lemma:liftunitary}. Set $$q=v^*(p\oplus I_j\oplus 0_j)v\in P^{28F_l(4174L)+L}_{n+2j}(A).$$
			 As $\pi(q)=I_{k+j}\oplus 0_{j+n-k}$, we have $q\in P_{n+2j}^{28F_l(4174L)+L}(J^+)$. Moreover, we have $(r_{0})_{28F_l(4174L)+L}([q,k+j])=[p,k]$, as $q$ is homotopic to $p\oplus I_j\oplus 0_j$ in $P_{n+2j}^{28F_l(4174L)+L}(A)$.
			\descitem{Exactness at $\{K_1^L(A)\}_{L\geqslant 0}$ with control function 
				$$L\mapsto 14F_l(L)+L.$$}
			Suppose $[u]\in K_1^L(A)$ with $u\in U_n^L(A)$. If $\pi(u)=0\in K_1^L(Q)$, then by Lemma \ref{lemma:liftunitary}, there exists $v\in U^{14F_l(L)}(A)$ which is homotopic to identity in $U^{14F_l(A)}(A)$ such that $\pi(v)=\pi(u)$. Therefore $\pi(v^*u)$ is the identity, which implies that $v^*u\in U^{14F_l(L)+L}(J^+)$. Since $v^*u$ is homotopic to $u$ in $U^{14F_l(L)+L}(A)$, we have $(r_1)_{14F_l(L)+L}([v^*u])=[u]$.
			\descitem{The map $(r_0)_{126F_l(3L)}\circ \partial_{L}\colon K_1^L(Q)\to K_0^{126F_l(3L)}(A)$ is zero.}
			
			If $u\in U_n^L(Q)$, then $\partial_L(u)=[w(I_n\oplus 0_n)w^*,n]$, where $w\in U^{14F_l(L)}_{2n}(A)$ is a lift of $u\oplus u^*$ given by Lemma \ref{lemma:liftunitary}. Since by construction $w$ is homotopic to identity in $U^{14F_l(L)}_{2n}(A)$, we have $[w(I_n\oplus 0_n)w^*,n]=0$ in $K_0^L(A)$.
			\descitem{The map $\partial_L \circ (\pi_1)_L \colon K_1^L(A)\to K_0^{126F_l(3L)}(J)$ is zero.}
			
			If $u\in U^L(Q)$ is the image of some unitary $v$ in $U^L(A)$, then we can choose $v\oplus v^* \in U^L(A)$ to be the lift of $u\oplus u^*$ in the construction of $\partial_L([u])$. It follows that  $\partial_L([u])=0$.
			\descitem{Exactness at $\{K_0^L(J)\}_{L\geqslant 0}$ with control function 
				$$L\mapsto 4367L.$$}
			If $[p,k]\in K_0^L(J)$ with $p\in P^L_n(J^+)$ is in the kernel of  the map $(r_0)_L$, then there exists $j_1,j_2>0$ such that $p\oplus I_{j_1-k}\oplus 0_{j_2-n+k}$ is homotopic to $I_{j_1}\oplus 0_{j_2}$ in $P^{2L}_{j_1+j_2}(A)$. By the definition of $K_0^L(J)$, we may assume that $\pi(p)=I_{k}\oplus 0_{n-k}$. By Lemma \ref{lemma:homotopy->unitaryeq}, there exist $j_3,j_4>0$ and a unitary $u\in U^{3270L}_{j_1+j_2+j_3+j_4}(A)$ such that
			$$p\oplus I_{j_1-k}\oplus 0_{j_2-n+k}\oplus I_{j_3}\oplus 0_{j_4}=u(I_{j_1}\oplus 0_{j_2}\oplus I_{j_3}\oplus 0_{j_4})u^*,$$
			 and $u$ is homotopic to the identity in $U^{4174L}_{j_1+j_2+j_3+j_4}(A)$.
			
			By switching rows and columns if necessary, we may find $j>0$ and a projection $q\in P^L_{2j}(J^+)$ such that $\pi(q)=I_j\oplus 0_j$ and $q$ is homotopic to $p\oplus I_{j-k}\oplus 0_{j+k-n}$ in $P^L_{2j}(J^+)$. Furthermore, we also obtain  a unitary  $v\in U^{3270L}_{2j}(A)$ such that $q=v(I_j\oplus 0_j)v^*$ and  $v$ is homotopic to the identity in $U^{4174L}_{2j}(A)$.
			
			As $\pi(q)=I_{j}\oplus 0_j=\pi(v)(I_j\oplus 0_j)\pi(v^*)$, we see that $\pi(v)=v_1\oplus v_2$ with 
			$$v_1,v_2\in U^{3270L}_{j}(Q)\subset U^{4367L}_{j}(Q).$$
			 By the construction in Lemma \ref{lemma:K_*^Lisabelian}, we see that $I_j\oplus v_1v_2$ is homotopic to the identity in $U^{4174L}_{2j}(Q)$. Thus by Lemma \ref{lemma:liftunitary}, there exists a unitary $w\in U^{14F_l(4174L)}_{2j}(A)$ such that $\pi(w)= v_1v_2\oplus I_j$
			 and $w$ is homotopic to the identity in $U^{14F_l(4174L)}_{2j}(A)$.
			 Set 
			 $$w'=(v\oplus I_j)(I_j\oplus w^*)\in U^{14F_l(4174L)+3270L}_{3j}(A)\subset U^{14F_l(4367L)}_{3j}(A).$$
			 Therefore $\pi(w')=v_1\oplus v_1^*\oplus I_j$. We also observe that
			 $$w'(I_j\oplus 0_{2j})(w')^*=
			 (v\oplus I_j)(I_j\oplus 0_{2j})(v\oplus I_j)=q\oplus 0_{j}.
			 $$
			 This shows that for $v_1\in U^{4367L}_{j}(Q)$, we have
			 $$\partial_{4367L}[v_1]=[w'(I_j\oplus 0_{2j})(w')^*,j]=[q,j]=[p,k],
			 $$
			 i.e., $[p,k]$ lies in the image of 
			$\partial_{4367L}\colon K_1^{4367L}(Q)\to K_0^{126F_l(13101L)}(J).$
			\descitem{Exactness at $\{K_1^L(Q)\}_{L\geqslant 0}$ with control function 
				$$L\mapsto 412034F_l(3L).$$} 
			Suppose $u\in U^L_n(Q)$ is in the kernel of the map $\partial_L$, that is, $w(I_n\oplus 0_n)w^*$ is homotopic to $I_n\oplus 0_n$ in $P^{252F_l(3L)}(J^+)$, where $w\in U^{14F_l(L)}_n(A)$ and $\pi(w)=u\oplus u^*$. By Lemma \ref{lemma:homotopy->unitaryeq}, there exist $j_1,j_2>0$ and $v\in U^{412020F_l(3L)}_{2n+j_1+j_2}(J^+)$ such that 
			$$v(w(I_n\oplus 0_n)w^*\oplus I_{j_1}\oplus 0_{j_2})v^*= I_n\oplus 0_n\oplus I_{j_1}\oplus 0_{j_2}.$$ We may assume that $\pi(v)=I_{2n+j_1+j_2}$.
			
			By switching rows and columns if necessary, there exist $j>0$, $w'\in U^{14F_l(L)}_{2j}(A)$, and $v'\in U^{412020F_l(3L)}_{2j}(J^+)$ such that
			$\pi(v')=I_{2j}$, $\pi(w')=u\oplus I_{j-n}\oplus u^*\oplus I_{j-n}$, and			
			$$v'w'(I_j\oplus 0_j)(v'w')^*=I_j\oplus 0_j.$$
			
			As $v'w'$ commutes with $I_j\oplus 0_j$, we have $v'w'=u_1\oplus u_2$, where 
			$$u_1,u_2\in U^{412020F_l(3L)+14F_l(L)}_{j}(A)\subset U^{412034F_l(3L)}_{j}(A).$$
			Therefore $\pi(u_1)=u\oplus I_{j-n}$. This finishes the proof. 
		\end{itemize}
	\end{proof}
	\subsection{Suspensions}\label{sec:suspen}
	In this subsection, we extend the asymptotically exact sequence in Proposition \ref{prop:longexact-v1} further to the left by using suspensions.
	
	For a Lipschitz filtered $C^*$-algebra $A$, we define 
	$$SA=\{f\colon S^1\to A\mid f(1)=0\}\text{ and } SA_L=\{f\in SA\mid f(z)\in A_L,\forall z\in S^1\}.$$
	It is easy to verify that $\{SA_L\}_{L\geqslant 0}$ is a Lipschitz filtration of $SA$ that satisfies the conditions (L1)-(L6) (cf. Definition \ref{def:Lipfiltration} and Remark \ref{rk:extra}).
	If 	$$0\longrightarrow J\xlongrightarrow{r}A\xlongrightarrow{\pi}Q\longrightarrow0
	$$ is a controlled short exact sequence of Lipschitz filtered $C^\ast$-algebras as in Definition $\ref{def:filtered-exact-sequence}$, then by suspension we have  the following controlled short exact sequence
	$$0\longrightarrow SJ\xlongrightarrow{Sr}SA\xlongrightarrow{S\pi}SQ\longrightarrow0
	$$
	where the control function for the controlled surjection $S\pi$ is also $F_l$. We denote by $S\partial$ the boundary map 
	$$\{K_1^L(SQ)\}_{L\geqslant 0}\xlongrightarrow{S\partial} \{K_0^L(
	SJ)\}_{L\geqslant 0}.$$
	\begin{proposition}\label{prop:suspensioniso}
		Let $A$ be a Lipschitz filtered $C^*$-algebra. There are  natural controlled homomorphisms
		$$\theta_A\colon \{K_1^L(A)\}_{L\geqslant 0}\xlongrightarrow{}\{K_0^L(SA)\}_{L\geqslant 0}\text{ and } 
		\gamma_A\colon \{K_0^L(SA)\}_{L\geqslant 0}\xlongrightarrow{}\{K_1^L(A)\}_{L\geqslant 0}$$
		such that $\theta_A\gamma_A$ and $\gamma_A\theta_A$ are controlled  equivalent to the identity map respectively. Moreover, all control functions here can be chosen to be  independent of $A$.
	\end{proposition}
\begin{proof}
	Set 
		$$CA=\{f: [0,1]\to A\ |\ f(0)=0\},\ CA_L=\{f\in CA\ |\ f(t)\in A_L,\forall t\in [0,1] \}.$$
			It is easy to verify that $\{CA_L\}_{L\geqslant 0}$ is a Lipschitz filtration of $CA$ satisfying the conditions (L1)-(L6) (cf. Definition \ref{def:Lipfiltration} and Remark \ref{rk:extra}).
			
	Consider the short exact sequence
	$$0\longrightarrow SA\longrightarrow CA\xlongrightarrow{\rho} A\longrightarrow 0.$$
	Note that for any $a\in A$, the element $\tilde a \in CA$ given by  $\tilde a(t) = ta$ is a lift of $a$. It follows that $\rho\colon CA \to A$ is a controlled surjection with control function given by $L\mapsto L$. Therefore by Proposition \ref{prop:longexact-v1}, we obtain an asymptotic exact sequence
	$$		\{K_1^L(CA)\}_{L\geqslant 0}\xrightarrow{}\{K_1^L(A)\}_{L\geqslant 0}
	\xrightarrow{\theta_A}\{K_0^L(SA)\}_{L\geqslant 0}\xrightarrow{} \{K_0^L(CA)\}_{L\geqslant 0}.$$
	
	For every $s\in[0,1]$,  we define $\varphi_s\colon CA\to CA$ by
	$$\varphi_s(g)(t)=\begin{cases}
	0&0\leqslant t\leqslant s\\
	g(t-s)&s\leqslant t\leqslant 1
	\end{cases}$$
	Then $\varphi_s$, $s\in [0, 1]$,  a strongly continuous path of filtered homomorphisms from $CA$ to itself connecting the identity map and the zero map. It follows from  Proposition \ref{prop:functoriality+homotopy-invariance} that $K_*^L(CA)=0$ for any $L\geqslant 0$. By part (2) of Lemma \ref{lemma:epi-mono}, there exists a controlled homomorphism  $\gamma_A\colon\{K_0^L(SA)\}_{L\geqslant 0}\to\{K_1^L(A)\}_{L\geqslant 0}$ such that $\theta_A\gamma_A$ and $\gamma_A\theta_A$ are controlled equivalent to the identity map respectively. The naturality of $\gamma_A$ follows from Remark \ref{rk:naturalofsplit}.
\end{proof}
	\begin{proposition}\label{prop:longexact}
	Let 	$$0\longrightarrow J\xlongrightarrow{r}A\xlongrightarrow{\pi}Q\longrightarrow0
		$$ be a controlled short exact sequence of Lipschitz filtered $C^\ast$-algebras as in Definition $\ref{def:filtered-exact-sequence}$,
		Then we have the following asymptotically exact sequence
 		\begin{align*}
		\{K_1^L(SA)\}\xlongrightarrow{S\pi_1}\{K_1^L(SQ)\}\xlongrightarrow{(S\partial)\circ\theta_J}\{K_1^L(J)\}& \\
		\xlongrightarrow{r_1} \{K_1^L(A)\}\xlongrightarrow{\pi_1}\{K_1^L(Q)\}
		\xlongrightarrow{\partial}\{K_0^L(J)\}\xlongrightarrow{r_0} \{K_0^L(A)\}\xlongrightarrow{\pi_0}\{K_0^L(Q)\}&
		\end{align*}
	where the control function for the asymptotic exactness at each term only depends on the control function $F_l$ of the controlled surjection $\pi$.
	\end{proposition}
\begin{proof}	Consider the short exact sequence
	$$0\longrightarrow SJ\xlongrightarrow{Sr}SA\xlongrightarrow{S\pi}SQ\longrightarrow0
$$
where the control function of the controlled surjection $S\pi$ is also $F_l$. By Proposition \ref{prop:longexact-v1}, we have an asymptotically exact sequence 
\begin{align*}
\{K_1^L(SA)\}_{L\geqslant 0}\xrightarrow{S\pi_1}\{K_1^L(SQ)\}_{L\geqslant 0}
\xrightarrow{S\partial}\{K_0^L(SJ)\}_{L\geqslant 0}\xrightarrow{Sr_0} \{K_0^L(SA)\}_{L\geqslant 0}.
\end{align*}
By Proposition \ref{prop:suspensioniso}, we have natural controlled isomorphisms $\{K_0^L(SA)\}_{L\geqslant 0} \cong \{K_1^L(A)\}_{L{\geqslant 0}}$  and   $\{K_0^L(SJ)\}_{L\geqslant 0} \cong \{K_1^L(J)\}_{L{\geqslant 0}}$.  Now the proposition follows from  Proposition \ref{prop:longexact-v1}. 
\end{proof}
\begin{corollary}\label{cor:splitexact}
	Let $A,J,Q$ be Lipschitz filtered $C^*$-algebras and $r$ and $\pi$ filtered homomorphisms that fit into a short exact sequence
			$$0\longrightarrow J\xlongrightarrow{r}A\xlongrightarrow{\pi}Q\longrightarrow0.
	$$
	If the exact sequence splits, i.e., there exists filtered homomorphism $s\colon Q\to A$ such that $\pi\circ s=id$, then the short exact sequence is controlled \textup{(}in the sense of Definition $\ref{def:filtered-exact-sequence}$\textup{)} with control function $L\mapsto L$. 
	
	Furthermore, there exist controlled homomorphisms
	$$p_*\colon\{K_*^L(A) \}_{L\geqslant 0}\to \{K_*^L(J)\}_{L\geqslant 0},\ *=0,1$$
	such that $p_*\circ r_*$ are controlled equivalent to the identity respectively with control functions independent of the short exact sequence.
\end{corollary}
\begin{proof}
	For any $x\in M_n(Q)_L$, $s(x)$ is a required lift of $x$ as in (2) of Definition \ref{def:filtered-exact-sequence} with control function $L\mapsto L$. The same holds for the suspension
	$$0\longrightarrow SJ\xlongrightarrow{Sr}SA\xlongrightarrow{S\pi}SQ\longrightarrow0.
	$$
	Consider the long exact sequence in Proposition \ref{prop:longexact}. As $\pi_*\circ s_*$ is the identity map for $*=0, 1$, we have $(S\partial)\circ\theta_J$ and $\partial$ are controlled equivalent to zero respectively by Lemma \ref{lemma:epi-mono}. Therefore the existence of $p_*$ follows from Lemma \ref{lemma:splitexact}.
\end{proof}
	
	\subsection{Controlled Bott periodicity}
	In this subsection, we prove the controlled Bott periodicity for Lipschitz controlled $K$-theory . 
	
	Let $\cK$ be the algebra of compact operators. First we prove that Lipschitz controlled $K$-theory is invariant (in a controlled way) under Morita equivalences. More precisely, we have the following proposition. 
	\begin{proposition}\label{prop:Morita}
	Given a Lipschitz filtered $C^*$-algebra $A$, we equip $A\otimes \K$ with a lipschitz filtration by setting
		$$(A\otimes \K)_L:=\overline{\bigcup_{n\geqslant 1}M_n(A)_L }.$$
		Then 
		$\{K_*^L(A)\}_{L\geqslant 0}$ and $\{K_*^L(A\otimes \K)\}_{L\geqslant 0}$ are naturally controlled isomorphic with control functions independent of $A$. 
	\end{proposition}
	\begin{proof}
		Without loss of generality, we assume that $A$ is non-unital. We regard $\K$ as the collection of compact operators on $l^2(\N)$. In particular, we view $M_n(A)$ as a subalgebra of $A\otimes \K$. The controlled homomorphism from $\{K_*^L(A)\}_{L\geqslant 0}$ to $\{K_*^L(A\otimes \K)\}_{L\geqslant 0}$ is given by:
		\begin{itemize}
			\item $
\xi_0\colon\{K_0^L(A)\}_{L\geqslant 0}\to \{K_0^L(A\otimes \K)\}_{L\geqslant 0},\ 
[p,k]\mapsto[p\oplus 0,k]$
			\item$\xi_1\colon\{K_1^L(A)\}_{L\geqslant 0}\to \{K_0^L(A\otimes \K)\}_{L\geqslant 0},\ [u]\mapsto [u\oplus 1]$
		\end{itemize}
	The control function of $\xi_*$ is the identity function $L\mapsto L$.
	
	Conversely, let $p_n$ be the projection from $l^2(\N)$ to $l^2(\{1,2,\cdots,n\})$. If $p$ is a projection in $M_k((A\otimes \K)^+)_L$, then there exists $n$ such that $\|p_npp_n-p\|\leqslant \frac{1}{165}$. Thus the spectrum of $p_npp_n$ is contained in the $\frac{1}{165}$-neighborhood of the set $\{0, 1\}$.
	If we define
	$$\Theta(x)=\begin{cases}
	0& x<\frac 1 2;\\
	1& x\geqslant \frac 1 2.\\
	\end{cases},$$
	then by (L6) in Remark \ref{rk:extra} we have $\Theta(p_npp_n)\in P^{3L}_k(M_n(A^+))$. By switching the rows and columns if necessary, we obtain a homomorphism
	$$(\eta_{0})_L:K_0^L(A\otimes \K)\to K_0^{3L}(A)$$ by $[p,k]\mapsto [\Theta(p_npp_n),k]$.
	
	Furthermore, by (L6) in Remark \ref{rk:extra} we have
	$$\|\Theta(p_npp_n)-p\|\leqslant \frac{1}{12}.$$
	 By Lemma \ref{lemma:tech}, there is a homotopy in $P^{9L}_k(M_n(A^+))$ that connects $\Theta(p_npp_n)$ and $p$ in  $P^{9L}_k(M_n(A^+))$. Therefore $\xi_0\eta_0$ is controlled equivalent to the identity map,   with control function $L\mapsto 9L/2$. On the other hand, the controlled equivalence  $\eta_0\xi_0\sim\id$ is obvious.
	 
	 As for the $K_1$-case, if we have a unitary $u\in U^L_k((A\otimes \K)^+)$, then there exists $n$ such that $\|p_nup_n-u\|\leqslant \frac{1}{14}$. Set
	 $$w=(p_nup_n)(p_nu^*p_nup_n)^{-\frac 1 2}.$$
	 As $\|p_nu^*p_nup_n-1\|\leqslant \frac 1 7$, we have 
	 $$\|(p_nu^*p_nup_n)^{-\frac 1 2}-1\|\leqslant \frac{\sqrt{7}}{\sqrt{6}}-1 
	 \text{ and } 
	 (p_nu^*p_nup_n)^{-\frac 1 2} \textup{ lies in } U^{L(\frac{7}{6})^{\frac 3 2}}_k(M_n(A^+)).$$
	 Therefore $w$ is a unitary in $U_k^{\frac{2\sqrt 7}{\sqrt 6}L+L(\frac{7}{6})^{\frac 3 2}}(M_n(A^+))\subset U^{4L}_k(M_n(A^+))$. 
	 Hence by switching rows and columns if necessary, we obtain a homomorphism $$(\eta_{1})_L\colon K_1^L(A\otimes\K)\to K_1^{4L}(A)$$ by $[u]\mapsto [w]$. 
	 
	 Furthermore, we see that 
	 $$\Big\| (w\oplus 1)-u\Big\|\leqslant \frac{1}{\sqrt{12\cdot 14}}+\big(\frac{\sqrt 7}{\sqrt 6}-1\big)\leqslant \frac 1 6.$$
	 By Lemma \ref{lemma:tech}, we see that $(w\oplus 1)$ is homotopic to $u$ in $U^{16L}((A\otimes \K)^+)$. Therefore $\xi_1\eta_1$ is controlled equivalent to the identity map, with control function $L\mapsto 8L$. On the other hand, the controlled equivalence  $\eta_1\xi_1\sim\id$ is obvious. This finishes the proof. 
	\end{proof}
	
Let us now prove the controlled Bott periodicity for Lipschitz controlled $K$-theory. We adapt Cuntz' proof of the classical Bott periodicity into our current setting \cite[\S4]{Cuntz1894} (cf. \cite[\S9.4]{Blackadar}).
	\begin{theorem}\label{prop:Bottper}
		Let $A$ be a Lipschitz filtered $C^*$-algebra. There are natural controlled homomorphisms
		$$\beta_A\colon \{K_0^L(A)\}_{L\geqslant 0}\xlongrightarrow{}\{K_1^L(SA)\}_{L\geqslant 0}\text{ and } 
		\phi_A\colon\{K_1^L(SA)\}_{L\geqslant 0}\xlongrightarrow{}\{K_0^L(A)\}_{L\geqslant 0}$$
		such that $\phi_A\beta_A$ and $\beta_A\phi_A$ are controlled equivalent to the identity map respectively. Moreover, all control functions here can be chosen to be  independent of $A$.
	\end{theorem}
\begin{proof}

	Let $T$ be the Toeplitz algebra, which is generated by the unilateral shift $u$ on the Hilbert space $l^2(\N)$ with $e=1-uu^*$ a rank-one projection. Then $T$ contains $\K$ as an ideal and $T/\K\cong C(S^1)$. Let $T_0$ be the non-unital $C^*$-subalgebra of $T$ generated by $(u-1)$. Set
	\begin{equation}\label{eq:bartoep}
\overline T=\{(x,y)\in (T\otimes T)\oplus T: (\pi\otimes\id)(x)=\pi(y)\otimes 1\in C(S^1)\otimes T \},
	\end{equation}
	where $\pi$ is the quotient map $T\to C(S^1)$.
	
	Let $B$ be any of the $C^*$-algebras mentioned above: $T,T\otimes T,T_0,\overline T,C(S^1)$. We equip $B\otimes A$ with the filtration:
	$$(B\otimes A)_L:=\{t\in B\otimes A:\chi(t)\in A_L,\ \forall\chi\in B^*,\ \|\chi\|\leqslant1\}.$$
	As each of $T,T\otimes T,T_0,\overline T,C(S^1)$ is a nuclear $C^\ast$-algebra, there is no ambiguity as to which $C^\ast$-norm $B\otimes A$ is equipped with.  
	
	This filtration is Lipschitz in the sense of Definition \ref{def:Lipfiltration}, due to the GNS construction and particularly (L5) of $A$. We furthermore assume (L6) in Remark \ref{rk:extra} for $A\otimes B$, which is satisfied by the concrete examples we work with in this paper (e.g. when $A = C_0(X)$ of some locally compact metric space $X$). Here the condition (L6) is only assumed here for convenience. Without the condition (L6) from Remark \ref{rk:extra}, the same proof below still proves the theorem, except we need to use almost projections or unitaries instead (cf. \cite{Oyono-OyonoYu}).
	
	Note that if $B$ is $S\C$, the filtration on $S\C\otimes A\cong SA$ coincides with the one we defined in the previous subsection (Section \ref{sec:suspen}), because the convex hull generated by the evaluation maps at points in $S^1$ is weakly dense in the unital ball of $(S\C)^*$.
	
Now  let $q\colon T\to \C$ be the composition of the quotient map from $T$ to $C(S^1)$ followed by the evaluation map at $x = 1\in S^1$, and $j\colon\C\to T$ the map that takes $1\in \C$ to $1\in T$. As $q\circ j=\id_\C$, we have
	$$(qj\otimes\id)_*\colon\{K_*^L(A)\}_{L\geqslant 0}\to \{K_*^L(A)\}_{L\geqslant 0}$$
	 is equal to the identity map. We shall show that $(jq\otimes\id)_*$ is controlled equivalent to the identity map on $\{K_*^L(T\otimes A)\}_{L\geqslant 0}$.
	
	Set 
	\begin{align*}
	z_0=&u(1-e)u^*\otimes 1+eu^*\otimes u+ue\otimes u^*+e\otimes e,\\
	z_1=&u(1-e)u^*\otimes 1+eu^*\otimes 1+ue\otimes 1,
	\end{align*}
	in $T\otimes T$. A direct computation shows that  $z_0$ and $z_1$ are self-adjoint unitaries. Hence both $z_0$ and $z_1$ are homotopic to the identity within the unitaries by applying a standard rotation argument to their spectral projections at $-1$ and at $1$. Consequently,  there exists a path of unitaries $z_t\in T\otimes T$ that connects $z_0$ and $z_1$. Since $(\pi\otimes\id)(z_0)$ and $(\pi\otimes\id)(z_1)$ are both equal to  the identity in $C(S^1)\otimes T$, we may assume that
	$(\pi\otimes\id)(z_t)$ is equal to the identity in $C(S^1)\otimes T$ for all $t\in [0, 1]$.
	
	Set $v_t=z_t(u\otimes 1)$. Recall that $T$ is the universal $C^\ast$ algebra by an isometry, that is, if $v$ is an isometry of any $C^\ast$-algebra, there is a canonical homomorphism from T onto $C^\ast(v)$ which maps $u$ to $v$. It follows that there exists a canonical homomorphism $\phi_t\colon T\to T\otimes T$ sending $u$ to $v_t$. Let $\overline T$ be the $C^\ast$-algebra in line \eqref{eq:bartoep}.  Define $\psi_t,\psi\colon T\to \overline T$ by setting
	$$\psi_t(u)=(\phi_t(u),u)\text{ and }\psi(u)=(u(1-e)\otimes 1,u).$$
	Define $\ell\colon \K\otimes T\to \overline T$ by $\ell(x)=(x,0)$, and $w\colon T\to \K\otimes T$ by $w(u)=e\otimes u$. As
	\begin{align*}
z_0(u\otimes 1)=&u(1-e)\otimes 1+e\otimes u=\psi(u)+\ell wjq(u),\\
z_1(u\otimes 1)=&u(1-e)\otimes 1+e\otimes 1=\psi(u)+\ell w(u),
\end{align*}
we have $\psi_0=\psi+\ell w j q$ and $\psi_1=\psi+\ell w$. As $\psi$ and $\ell$ have orthogonal images, we conclude 	
$$(\ell wjq\otimes\id)_*=(\ell w\otimes\id)_*\colon \{K_*^L(\bar T\otimes A)\}_{L\geqslant 0}\to \{K_*^L(T\otimes A)\}_{L\geqslant 0}.$$

Consider the short exact sequence
$$0\longrightarrow \K\otimes T\otimes A\xlongrightarrow{\ell\otimes\id} \overline T\otimes A\longrightarrow T\otimes A\longrightarrow 0,$$
which  splits by $\psi\otimes\id\colon T\otimes A\to \overline T\otimes A$. By Corollary \ref{cor:splitexact}, there exist controlled homomorphisms
$$\ell^{-1}_*\colon\{ K_*^L(\overline T\otimes A)\}_{L\geqslant 0}\to\{K_*^L(\K\otimes T\otimes A)\}_{L\geqslant 0},~*=0,1$$
such that $\ell^{-1}_*\circ (\ell\otimes\id)_*$ is controlled equivalent to the identity map.

By Proposition \ref{prop:Morita}, there also exist controlled homomorphisms
$$w^{-1}_*\colon\{ K_*^L(\K\otimes T\otimes A)\}_{L\geqslant 0}\to\{K_*^L(T\otimes A)\}_{L\geqslant 0},~*=0,1$$
such that $w^{-1}_*\circ (w\otimes\id)_*$ is controlled equivalent to the identity map. Therefore
$$(jq\otimes\id)_*: \{K_*^L(T\otimes A)\}_{L\geqslant 0}\to \{K_*^L(T\otimes A)\}_{L\geqslant 0},~*=0,1$$
is controlled equivalent to the identity map.

Similarly, consider the short exact sequence
$$0\longrightarrow T_0\otimes A\xlongrightarrow{i\otimes\id} T\otimes A\xlongrightarrow{q\otimes\id} A\longrightarrow 0,$$
which splits by $j\otimes\id\colon A = \C\otimes A \mapsto T\otimes A$. Therefore by Corollary \ref{cor:splitexact}, there exist controlled homomorphisms
$$i^{-1}_*\colon\{K_*^L(T\otimes A)\}_{L\geqslant 0}\to \{K_0^L(T_0\otimes A)\}_{L\geqslant 0},~*=0,1$$
such that $i_*^{-1}(i\otimes\id)_*$ is controlled equivalent to the identity map. Since the identity map on $\{K_*^L(T\otimes A)\}_{L\geqslant 0}$ is controlled equivalent to $(jq\otimes\id)_*$, we see that the identity map on $\{K_*^L(T_0\otimes A)\}_{L\geqslant 0}$ is controlled equivalent to zero. 

Now consider the short exact sequence
$$0\longrightarrow \K\otimes A\longrightarrow T_0\otimes A\xlongrightarrow{\pi\otimes\id} SA\longrightarrow 0.$$
For $x=\sum_{i=1}^n f_i\otimes a_i\in SA$, we lift $x$ under the map $\pi\otimes \id$ to $y=\sum_{i=1}^n T_{f_i}\otimes a_i$, where $T_{f_i} \in T_0$ is the Toeplitz operator with symbol $f_i$ satisfying $\pi(T_{f_i}) = f_i$. Using this choice of lift, we see that $\pi\otimes\id$ is a controlled surjection with control function $L\mapsto L$ in the sense of Definition \ref{def:filtered-exact-sequence}. As the identity map on $\{K_*^L(T_0\otimes A)\}_{L\geqslant 0}$ is controlled equivalent to zero, it follows from Proposition \ref{prop:longexact-v1} and Lemma \ref{lemma:epi-mono} that there exist controlled homomorphisms
$$\beta_A\colon \{K_0^L(A)\}_{L\geqslant 0}\xlongrightarrow{}\{K_1^L(SA)\}_{L\geqslant 0}\text{ and } 
\phi_A\colon\{K_1^L(SA)\}_{L\geqslant 0}\xlongrightarrow{}\{K_0^L(A)\}_{L\geqslant 0}$$
such that $\phi_A\beta_A$ and $\beta_A\phi_A$ are controlled equivalent to the identity map respectively. Their naturality follows from Remark \ref{rk:naturalofsplit}. This finishes the proof. 
\end{proof}

By combining Proposition \ref{prop:longexact} and Proposition \ref{prop:Bottper}, we obtain the following six-term asymptotically exact sequence.
\begin{theorem}\label{thm:six-term}
	Suppose we have a controlled short exact sequence 
	$$0\longrightarrow J\xlongrightarrow{r}A\xlongrightarrow{\pi}Q\longrightarrow0
	$$
	where $\pi$ is a controlled surjection with control function $F_l$ in the sense of Definition $\ref{def:filtered-exact-sequence}$. Then there is a six-term asymptotically exact sequence
	$$\xymatrix{
\{K_0^L(J)\}_{L\geqslant 0}\ar[r]&\{K_0^L(A)\}_{L\geqslant 0}\ar[r]&\{K_0^L(Q)\}_{L\geqslant 0}\ar[d]\\
\{K_1^L(Q)\}_{L\geqslant 0}\ar[u]&\{K_1^L(A)\}_{L\geqslant 0}\ar[l]&\{K_1^L(J)\}_{L\geqslant 0}\ar[l]
}	
	$$
	where the control functions of all controlled homomorphisms and the control function for the asymptotic exactness at each term only depend on $F_l$.
\end{theorem}

By checking carefully the proofs of Proposition \ref{prop:longexact} and Theorem \ref{prop:Bottper}, we see that the control functions of all controlled homomorphisms and the control function for the asymptotic exactness at each term in Theorem \ref{thm:six-term} can be chosen to be of  the form $F_1\circ F_l\circ F_2$, where $F_1,F_2$ are linear functions that are universal and independent of the controlled short exact sequence.
\subsection{Controlled Mayer-Vietoris sequence}
In this subsection, we prove a controlled Mayer-Vietoris sequence. 
\begin{definition}
	A pullback diagram of $C^*$-algebras is a diagram of the following form
	$$\xymatrix{
P\ar[r]^-{\rho_A}\ar[d]_-{\rho_B}&A\ar[d]^-{\pi_A}\\
B\ar[r]^-{\pi_B}&Q	
}$$
where $P=\{(a,b) \in A\oplus B:\pi_A(a)=\pi_B(b)\}$ and $\pi_A$ is surjective.
\end{definition}
Let $A$ and $B$ be Lipschitz filtered $C^*$-algebras. We equip $A\oplus B$ a Lipschitz filtration given by $(A\oplus B)_L=A_L\oplus B_L$. If we equip $P$ with the Lipschitz filtration given by  $P_L=P\cap (A\oplus B)_L$, then  $\rho_A$ and $\rho_B$ become filtered homomorphisms.
\begin{definition}
	 We say a pullback diagram is controlled if $\rho_A,\rho_B,\pi_A,\pi_B$ are filtered homomorphisms and $\pi_A$ is a controlled surjection with control function $F_l$ in the sense of Definition \ref{def:filtered-exact-sequence}.
\end{definition}

\begin{theorem}\label{thm:six-term-for-pullback}
	If we have a controlled pullback diagram of $C^\ast$-algebras 
	$$\xymatrix{
	P\ar[r]^-{\rho_A}\ar[d]_-{\rho_B}&A\ar[d]^-{\pi_A}\\
	B\ar[r]^-{\pi_B}&Q	
}$$
	with $\pi_A$ a controlled surjection with control function $F_l$, then there is a six-term asymptotically exact sequence
	$$\xymatrix{
		\{K_0^L(P)\}_{L\geqslant 0}\ar[r]&\{K_0^L(A)\oplus K_0^L(B)\}_{L\geqslant 0}\ar[r]&\{K_0^L(Q)\}_{L\geqslant 0}\ar[d]\\
		\{K_1^L(Q)\}_{L\geqslant 0}\ar[u]&\{K_1^L(A)\oplus K_1^L(B)\}_{L\geqslant 0}\ar[l]&\{K_1^L(P)\}_{L\geqslant 0}\ar[l]
	}	
	$$
	where the control functions of all controlled homomorphisms and the control function for the asymptotic exactness at each term only depend on $F_l$.
\end{theorem}
\begin{proof}
	Consider the following commutative diagram
$$\xymatrix{
0\ar[rr]&&\ker(\rho_B)\ar@{=}[d]\ar[rr]&& P\ar[rr]^{\rho_B}\ar[d]^{\rho_A}&&B\ar[rr]\ar[d]^{\pi_B}&&0\\
0\ar[rr]&&\ker(\pi_A)\ar[rr]&& A\ar[rr]^{\pi_A}&&Q\ar[rr]&&0
}$$
where  $\ker(\rho_B)=\ker(\pi_A)$ is equipped with the Lipschitz filtration inherited from that of $P$ (or equivalently from that of $A$). Therefore both rows are controlled short exact sequences. Hence by Theorem \ref{thm:six-term}, we have the following diagram that commutes up to controlled equivalences:
$$
\xymatrixcolsep{0.65pc}
\xymatrix{
\{K_0^L(\ker(\rho_B))\}\ar[rr]\ar@{=}[dr]&&\{K_0^L(P)\}\ar[rr]\ar[d]&&\{K_0^L(B)\}\ar[ddd]\ar[dl]\\
&\{K_0^L(\ker(\pi_A))\}\ar[r]&\{K_0^L(A)\}\ar[r]&\{K_0^L(Q)\}\ar[d] \\
&\{K_1^L(Q)\}\ar[u]&\{K_1^L(A)\}\ar[l]&\{K_1^L(\ker(\pi_A))\}\ar[l]\ar@{=}[dr]\\
\{K_1^L(B)\}\ar[uuu]\ar[ur]\ar[rr]&&\{K_1^L(P)\}\ar[u]\ar[rr]&&\{K_1^L(\ker(\rho_B))\}
}
$$
where both the inner and outer six-term sequences are asymptotically exact. Now a standard diagram chasing proves the theorem. 
\end{proof}
\nocite{Rudolf2019,Rudolf2020,GXY2020,Simone1,Simone2,Gromov4lectures2019,MR1389019, Gromov5fold,MR3822551}


\end{document}